\documentclass[oneside,11pt]{amsart}
\usepackage[utf8]{inputenc}
\usepackage{mathrsfs}
\usepackage{bm}
\usepackage{amsbsy}
\usepackage{amstext}
\usepackage{amsthm}
\usepackage{amssymb}
\usepackage[all]{xy}

\makeatletter
\numberwithin{equation}{section}
\numberwithin{figure}{section}
  \theoremstyle{remark}
  \newtheorem*{acknowledgement*}{\protect\acknowledgementname}
\theoremstyle{plain}
\newtheorem{thm}{\protect\theoremname}[section]
  \theoremstyle{remark}
  \newtheorem{notation}[thm]{\protect\notationname}
  \theoremstyle{definition}
  \newtheorem{defn}[thm]{\protect\definitionname}
  \newtheorem{exam}[thm]{\protect\examplename}
  \theoremstyle{plain}
  \newtheorem{prop}[thm]{\protect\propositionname}
  \theoremstyle{remark}
  \newtheorem{rem}[thm]{\protect\remarkname}
  \theoremstyle{plain}
  \newtheorem{cor}[thm]{\protect\corollaryname}
  \theoremstyle{plain}
  \newtheorem{lem}[thm]{\protect\lemmaname}

\usepackage{a4wide}
\usepackage{url}

\usepackage{xcolor}
\usepackage{version}

\excludeversion{NB}

\newcommand{\Hom}{\operatorname{Hom}}
\newcommand{\End}{\operatorname{End}}
\newcommand{\Ext}{\operatorname{Ext}}

\newcommand{\Uq}{\mathbf{U}_q}
\newcommand{\Rq}{\mathbf{R}_q}
\newcommand{\cUq}{\check{\mathbf{U}}_q}
\newcommand{\Aq}{\mathbf{A}_q}
\newcommand{\Qq}{\mathbf{Q}}

\newcommand{\wt}{\operatorname{wt}}
\newcommand{\height}{\operatorname{ht}}

\newcommand{\up}{\mathrm{up}}
\newcommand{\low}{\mathrm{low}}
\newcommand{\Glow}{G^{\low}}
\newcommand{\Gup}{G^{\up}}

\newcommand{\Ker}{\operatorname{Ker}}
\newcommand{\Image}{\operatorname{Im}}
\newcommand{\soc}{\operatorname{soc}}

\newcommand{\dimv}{\operatorname{\underline{dim}}}

\makeatother

  \providecommand{\acknowledgementname}{Acknowledgement}
  \providecommand{\corollaryname}{Corollary}
  \providecommand{\definitionname}{Definition}
  \providecommand{\examplename}{Example}
  \providecommand{\lemmaname}{Lemma}
  \providecommand{\notationname}{Notation}
  \providecommand{\propositionname}{Proposition}
  \providecommand{\remarkname}{Remark}
\providecommand{\theoremname}{Theorem}

\begin{document}


\title[Twist automorphisms on quantum unipotent cells]{Twist automorphisms on quantum unipotent cells and dual canonical bases}

\author{Yoshiyuki Kimura}

\address[Yoshiyuki Kimura]{Faculty of Liberal Arts and Sciences, Osaka Prefecture
	University, 1-1, Gakuen-cho, Naka-ku, Sakai, Osaka 599-8531, Japan}

\email{ysykimura@las.osakafu-u.ac.jp}

\author{Hironori Oya}

\address[Hironori Oya]{Department of Mathematical Sciences, Shibaura Institute of Technology, 307 Fukasaku, Minuma-ku, Saitama-shi, Saitama, 337-8570, Japan}

\email{hoya@shibaura-it.ac.jp}

\thanks{The work of the first author was supported by JSPS Grant-in-Aid for Scientific Research (S) 24224001 and JSPS Grant-in-Aid for Young Scientists (B) 17K14168.}

\thanks{The work of the second author was supported by Grant-in-Aid for JSPS Fellows (No.~15J09231) and the Program for Leading Graduate Schools, MEXT, Japan. It was also supported by the European Research Council under the European Union's Framework Programme H2020 with ERC Grant Agreement number 647353 Qaffine, during the revision of this paper.}

\begin{abstract}
In this paper, we construct twist automorphisms on quantum unipotent
cells, which are quantum analogues of the Berenstein-Fomin-Zelevinsky
twist automorphisms on unipotent cells. We show that those quantum
twist automorphisms preserve the dual canonical bases of quantum unipotent
cells. 

Moreover we prove that quantum twist automorphisms are described by the syzygy functors for representations of preprojective algebras in the symmetric case. This is the quantum analogue of Gei\ss-Leclerc-Schr\"oer's description, and Gei\ss-Leclerc-Schr\"oer's results are essential in our proof. As a consequence,
we show that quantum twist automorphisms are compatible with quantum
cluster monomials. The 6-periodicity of specific quantum twist automorphisms is also verified.
\end{abstract}

\maketitle
\tableofcontents{}

\section{Introduction}

\label{s:intro}

\subsection{Canonical bases and cluster algebras}

Let $G$ be a connected, simply-connected, complex
simple algebraic group with a fixed maximal torus $H$, a pair of
Borel subgroups $B_{\pm}$ with $B_{+}\cap B_{-}=H$, the Weyl group
$W=\mathrm{Norm}_{G}\left(H\right)/H$ and the maximal unipotent subgroups
$N_{\pm}\subset B_{\pm}$ (In the main body of this paper, we deal with ``the Kac-Moody groups''). Let $\mathbf{U}_{q}\left(\mathfrak{g}\right)$
be the Drinfeld-Jimbo quantized enveloping algebra of the corresponding
Lie algebra $\mathfrak{g}$, and $\mathbf{U}_{q}^{-}\left(\mathfrak{g}\right)$
be its negative part which arises from the triangular decomposition
of $\mathfrak{g}$. In \cite{MR1035415}, Lusztig constructed the
canonical bases $\mathbf{B}$ of $\mathbf{U}_{q}^{-}\left(\mathfrak{g}\right)$
using perverse sheaves on the varieties of quiver representations
when $\mathfrak{g}$ is simply-laced. In \cite{MR1115118}, Kashiwara
constructed the lower global bases $G^{\mathrm{low}}\left(\mathscr{B}\left(\infty\right)\right)$
of $\mathbf{U}_{q}^{-}\left(\mathfrak{g}\right)$ in general. In simply-laced
case, Lusztig \cite{MR1182165} proved that the two bases of $\mathbf{U}_{q}^{-}\left(\mathfrak{g}\right)$
coincide. In this paper, we call the bases the \emph{canonical base}s.
The canonical bases have interesting structures; one is positivity
of structure constants of multiplications and (twisted) comultiplication,
and another is the combinatorial structure which is called Kashiwara
crystal structure. Using the positivity of the canonical bases, Lusztig
\cite{MR1327548} generalized the notion of the total positivity for
reductive groups and related algebraic varieties.

Since $\mathbf{U}_{q}^{-}$ has a natural non-degenerate Hopf pairing
which makes it into a (twisted) self-dual bialgebra, we can consider
$\mathbf{U}_{q}^{-}$ as a quantum analogue of the coordinate rings
$\mathbb{C}\left[N_{-}\right]$. The combinatorial structure of $\mathbf{B}^{\mathrm{low}}$
and its dual basis $\mathbf{B}^{\mathrm{up}}$ (with respect to the
non-degenerate Hopf pairing), called the \emph{dual canonical bases},
has been intensively studied by Lusztig \cite[Chapter 42]{Lus:intro} and
Berenstein-Zelevinsky \cite{MR1237826,MR1387682} (in the type $\mathrm{A}$-case)
and it became one of the origins of cluster algebras introduced by
Fomin-Zelevinsky \cite{MR1887642}.

\subsection{Quantum unipotent subgroups and dual canonical bases }

For a Weyl group element $w\in W$ (and a lift $\dot{w}\in\mathrm{Norm}_{G}\left(H\right)$),
the unipotent root subgroups $N_{-}\left(w\right):=N_{-}\cap\dot{w}N\dot{w}^{-1}$
and the Schubert cells $B_{+}\dot{w}B_{+}/B_{+}$ in the full flag varieties
$G/B_{+}$ have attracted much attention in the development of the
theory of total positivity for reductive groups. Gei\ss-Leclerc-Schr\"oer
\cite{MR2822235} introduced a cluster algebra structure on $\mathbb{C}\left[N_{-}\left(w\right)\right]$
using representation theory of preprojective algebras, called an additive
categorification. They also proved that the dual semicanonical basis
$\mathcal{S}^{*}$ is compatible with $\mathbb{C}\left[N_{-}\left(w\right)\right]$, that is, $\mathcal{S}^{*}\cap\mathbb{C}\left[N_{-}\left(w\right)\right]$ gives a basis of $\mathbb{C}\left[N_{-}\left(w\right)\right]$, and
the set of cluster monomials 
is contained in the dual semicanonical basis
$\mathcal{S}^{*}$. Here we note that we identify the coordinate rings
$\mathbb{C}\left[N_{-}\left(w\right)\right]$ of the unipotent subgroups
$N_{-}\left(w\right)$ as invariant subalgebras $\mathbb{C}\left[N_{-}\right]^{N_{-}\cap\dot{w}N_{-}\dot{w}^{-1}}$
fixing a splitting $N_{-}\simeq\left(N_{-}\cap\dot{w}N_{-}\dot{w}^{-1}\right)\times N_{-}\left(w\right)$
as varieties.

For a nilpotent Lie algebra $\mathfrak{n}_{-}\left(w\right)$ associated
with the subgroup $N_{-}\left(w\right)$, a quantum analogue $\mathbf{U}_{q}^{-}\left(w\right)$
of the universal enveloping algebras $\mathbf{U}\left(\mathfrak{n}_{-}\left(w\right)\right)$
has been introduced by De Concini-Kac-Procesi \cite{MR1351503} and
also by Lusztig \cite{Lus:intro} as subalgebras of the quantized
enveloping algebras $\mathbf{U}_{q}^{-}$. They are defined as subalgebras
which are generated by quantum root vectors defined by Lusztig's braid
group symmetry on the quantized enveloping algebras $\mathbf{U}_{q}\left(\mathfrak{g}\right)$. Meanwhile they are the linear spans of their Poincaré-Birkhoff-Witt type orthogonal monomials with respect to the non-degenerate pairing on
$\mathbf{U}_{q}^{-}$. In \cite{MR2914878}, the first author proved
that the subalgebras $\mathbf{U}_{q}^{-}\left(w\right)$ are compatible
with the dual canonical bases, that is $\mathbf{B}^{\mathrm{up}}\cap\mathbf{U}_{q}^{-}\left(w\right)$
is a base of $\mathbf{U}_{q}^{-}\left(w\right)$ and the specialization
of $\mathbf{U}_{q}^{-}\left(w\right)$ (using the dual canonical basis)
at $q=1$ is isomorphic to the coordinate ring $\mathbb{C}\left[N_{-}\left(w\right)\right]$,
hence $\mathbf{U}_{q}^{-}\left(w\right)$ is also considered as a
quantum analogue of the coordinate ring $\mathbb{C}\left[N_{-}\left(w\right)\right]$
of the unipotent subgroup.

Gei\ss-Leclerc-Schr\"oer \cite{MR3090232} proved that $\mathbf{U}_{q}^{-}\left(w\right)$
admits a quantum cluster algebra structure in the sense of Berenstein-Zelevinsky
if $\mathfrak{g}$ is symmetric via the additive categorification
and Goodearl-Yakimov \cite{MR3263301, Goodearl:2013qw} proved the
result using the framework of quantum nilpotent algebras in the symmetrizable case.
Kang-Kashiwara-Kim-Oh \cite{KKKO} showed that the set of quantum
cluster monomials is contained in the dual canonical bases via symmetric
quiver Hecke algebras when $\mathfrak{g}$ is symmetric. See \cite[Introduction]{KKKO} for the history of this topic. 

\subsection{Unipotent cells and cluster structure}

For a pair $\left(w_{+},w_{-}\right)$ of Weyl group elements, the
intersections $G^{w_{+},w_{-}}:=B_{+}\dot{w}_{+}B_{+}\cap B_{-}\dot{w}_{-}B_{-}$
are called double Bruhat cells and the maximal torus $H$ acts $G^{w_{+},w_{-}}$
by left (or right) multiplication. For a certain lift $\overline{w}_{-}\in G$
of $w_{-}\in W$, the intersection $B_{+}\dot{w}_{+}B_{+}\cap N_{-}\overline{w}_{-}N_{-}$
is a section of the quotient $G^{w_{+},w_{-}}\to H\setminus G^{w_{+},w_{-}}$.
The unipotent cells $N_{-}^{w}:=B_{+}\dot{w}_{+}B_{+}\cap N_{-}$
are special cases of reduced double Bruhat cells where $w_{-}$ is
the unit of $W$. The (upper) cluster structure of the double Bruhat
cells and unipotent cells have been studied in details, see Berenstein-Fomin-Zelevinsky
\cite{MR2110627} (see also Gei\ss-Leclerc-Schr\"oer \cite{MR2822235}
and Williams \cite{MR3096792}). In fact, in \cite{MR2822235}, it
is shown that the coordinate ring of the unipotent subgroup has a
cluster algebra structure with unlocalized frozen variables, and that
the coordinate ring of the unipotent cell has a cluster algebra structure
with fully localized frozen variables.

For a Weyl group element $w\in W$, Berenstein-Fomin-Zelevinsky \cite{MR1405449}
(in the type $\mathrm{A}$-case) and Berenstein-Zelevinsky \cite{MR1456321}
(in general) introduced twist automorphisms which are automorphisms
on unipotent cells $N_{-}^{w}$ for solving the factorization problems,
called the Chamber Ansatz, which describe the inverse of the ``toric
chart'' of the Schubert varieties.

In \cite{MR2822235,MR2833478}, Gei\ss-Leclerc-Schr\"oer studied the additive
categorification of the twist automorphism using representation theory
of preprojective algebras, where it is given by the syzygy on the
Frobenius subcategory associated with $w$. They treated
the coordinate ring of the unipotent cells as the localization of
the coordinate rings unipotent subgroups with respect to the (unipotent)
minors associated with Weyl group elements. They also introduced the ``dual semicanonical bases'' of the coordinate ring of the unipotent cells, using the ``multiplicative property'' of dual semicanonical bases.

In this paper, we study the construction of a quantum analogue of
the twist automorphisms on the quantum unipotent cells, which are
the ``quantized coordinate rings of the unipotent cells'', and its
relation to the additive categorification. 

\subsection{Quantum unipotent cells}

Quantum coordinate rings of double Bruhat cells, called quantum double
Bruhat cells, are introduced by De Concini-Procesi \cite{MR1635678}
in the study of representation theory of quantum groups at root of
unity and also intensively studied by Joseph \cite{MR1315966} in
the study of prime spectra of quantized coordinate ring of $G$. Berenstein-Zelevinsky
\cite{MR2146350} conjectured that quantum double Bruhat cells admit
a structure of quantum cluster algebras via quantum minors. Goodearl-Yakimov
\cite{goodearl2016berenstein} proved the conjecture using a quantum
analogue of the Fomin-Zelevinsky twist of the double Bruhat cells.

In \cite{MR1635678}, De Concini-Procesi studied the relation between
the quantum unipotent subgroups and the quantum unipotent cells in
finite type case. In \cite{MR2914878}, the injectivity result of
De Concini-Procesi is generalized via the study of crystal bases. 

Berenstein-Rupel \cite{MR3397447} studied the quantum unipotent cells
via the Hall algebra technique and they constructed quantum analogue
of the twist maps under the conjecture concerning the quantum cluster
algebra structure and they showed that the quantum twist automorphisms
preserve the triangular bases (in the sense of Berenstein-Zelevinsky
\cite{MR3180605}) of the quantum unipotent cells when the Weyl group
element $w$ is the square of an acyclic Coxeter element $c$
with $\ell\left(w\right)=2\ell\left(c\right)$. 
We note
that Qin \cite{Qin:2016eu} proved that the triangular bases (=localized
dual canonical bases) in the sense of \cite{Qin:2015db} coincide
with the triangular bases in the sense of Berenstein-Zelevinsky \cite{MR3180605}
when $\mathfrak{g}$ is symmetric.

\subsection{Quantum unipotent cells and the dual canonical bases}

Our main results in this paper are the following:
\begin{enumerate}
\item We prove the De Concini-Procesi isomorphisms between the localizations $\mathbf{A}_{q}\!\left[N_{-}\left(w\right)\cap\dot{w}G_{0}^{\min}\right]$ of the quantum unipotent subgroups $\mathbf{A}_{q}\!\left[N_{-}\left(w\right)\right]$ and the quantum unipotent cells $\mathbf{A}_{q}\!\left[N_{-}^{w}\right]$ for arbitrary symmetrizable Kac-Moody cases (Theorem \ref{thm:DeCP}). The quantum cluster structure on the quantum unipotent cells can be proved as a corollary of the existence of the De Concini-Procesi isomorphisms (Corollary \ref{c:GLSqclus}). 

We should remark that the original De Concini-Procesi isomorphisms \cite[Theorem 3.2]{MR1635678} were given under the assumption that $\mathfrak{g}$ is of finite type. In \cite{MR1635678}, their existence was proved by downward induction on the length of elements of the Weyl group $W$ from the longest element, which exists only in finite type cases. 
\item We introduce a quantum analogue $\gamma_{w}$ of the twist isomorphism between the unipotent cells $N_{-}^{w}$ and $N_{-}\left(w\right)\cap\dot{w}G_{0}^{\min}$
which is defined using the Gauss decomposition (Theorem \ref{t:BZisom}).
\item We introduce a quantum analogue of the twist automorphism of unipotent
cells on the quantum coordinate ring $\mathbf{A}_{q}\!\left[N_{-}^{w}\right]$
of the unipotent cells (without referring the quantum cluster algebra
structure) and show that the quantum twist preserves the dual canonical
bases (Theorem \ref{t:BZauto}). In fact, we introduce a quantum
analogue of the twist automorphism as a composite of the De Concini-Procesi isomorphism and the quantum twist isomorphism. The result that the
dual canonical bases are preserved under the twist automorphism from
2 is proved as a consequence of the properties of two isomorphisms
and the dual canonical bases. We note that our construction is independent
of the construction by Berenstein-Rupel \cite{MR3397447}.
\item We relate the quantum twist automorphisms and the quantum cluster
structure under the additive categorification (Theorem \ref{t:qGLS}).
We also prove the $6$-periodicity of the twist automorphisms associated to the longest elements of the Weyl groups in finite type cases (Theorem \ref{t:period}).
\end{enumerate}

\subsection{Outline of the paper}

The paper is organized as follows. In section \ref{sec:Preliminaries},
we prepare the notations for Kac-Moody Lie algebras, Kac-Moody groups, and flag schemes. Moreover, we give a description of the coordinate rings of unipotent cells, and express ``classical twist maps'', which are defined by Berenstein-Zelevinsky \cite{MR1456321}, in terms of matrix coefficients. 
In section \ref{sec:Preliminaries2}, we give a brief review of quantum unipotent subgroups, quantum closed unipotent cells and canonical/dual canonical bases. The main result in this section is ``a crystalized Kumar-Peterson identity'' (Theorem \ref{t:cryKP}). In section \ref{sec:Quantum-unipotent-cell}, we define the dual canonical bases of the localized quantum coordinate rings and prove the De Concini-Procesi isomorphisms under the arbitrary symmetrizable Kac-Moody setting. In section \ref{sec:Quantum-twist-isomorphisms}, a quantum analogue of the twist isomorphism is introduced. In section \ref{sec:Twist-aut}, we define a quantum analogue of the twist automorphism as a composite of the quantum twist isomorphism and the De Concini-Procesi isomorphism. In section \ref{sec:GLS}, we relate the quantum twist automorphisms to the quantum cluster algebra structures via Gei\ss-Leclerc-Schr\"{o}er's additive categorification. In section \ref{sec:fin}, we study the periodicity of the twist automorphisms associated to the longest elements in finite type cases.

\subsection{Further work}

The comparison with the construction by Berenstein-Rupel \cite{MR3397447} and a quantum analogue of the Chamber Ansatz will be discussed in another paper\footnote{After the submission of the present paper, the paper corresponding to these topics by the second author appeared as \cite{Oya}.}. 

There is another type of ``quantum twist map'' which is not an automorphism,
introduced by Lenagan-Yakimov \cite{lenagan2015prime}. This is a
quantum analogue of the Fomin-Zelevinsky twist isomorphism \cite{MR1652878}.
The authors showed that it also preserves the dual canonical basis
of $\mathbf{A}_{q}\!\left[N_{-}\left(w\right)\right]$ \cite{Kimura:2016vn}.
However the authors do not know any explicit relations between this
quantum twist map and the quantum twist automorphisms in this paper.

\subsection{Basic notation}

\noindent\textup{(1)} Let $k$ be a field. For a $k$-vector space $V$, set
$V^{\ast}:=\Hom_{k}(V,k)$. Denote by $\langle\ ,\ \rangle\colon V^{\ast}\times V\to k$,
$(f,v)\mapsto\langle f,v\rangle$ the canonical pairing.

\noindent\textup{(2)} For a $k$-algebra $\mathcal{A}$, we set $[a_{1},a_{2}]:=a_{1}a_{2}-a_{2}a_{1}$
for $a_{1},a_{2}\in\mathcal{A}$. An Ore set $\mathcal{M}$ of $\mathcal{A}$
stands for a left and right Ore set consisting of non-zero divisors. Denote by $\mathcal{A}[\mathcal{M}^{-1}]$ the algebra of
fractions with respect to the Ore set $\mathcal{M}$. In this case,
$\mathcal{A}$ is naturally a subalgebra of $\mathcal{A}[\mathcal{M}^{-1}]$.
See \cite[Chapter 6]{MR1020298}.

\noindent\textup{(3)} An $\mathcal{A}$-module $V$ means a left $\mathcal{A}$-module.
The action of $\mathcal{A}$ on $V$ is denoted by $a.v$ for $a\in\mathcal{A}$
and $v\in V$. In this case, $V^{\ast}$ is regarded as a right $\mathcal{A}$-module
by $\langle f.a,v\rangle=\langle f,a.v\rangle$ for $f\in V^{\ast},a\in\mathcal{A}$
and $v\in V$.

\noindent\textup{(4)} For two symbols $i,j$, the notation $\delta_{ij}$
stands for the Kronecker delta.

\section{Preliminaries (1) : Kac-Moody Lie algebras and associated flag schemes}\label{sec:Preliminaries}
In this section, we fix the notation concerning (symmetrizable) Kac-Moody Lie algebras $\mathfrak{g}$ and associated Kac-Moody groups $G$, $G^{\min}$ and (not necessarily a group) schemes $\boldsymbol{G}$. See Kashiwara \cite{MR1463702} (see also Kashiwara-Tanisaki \cite{MR1317626}) for more details. In subsection \ref{ss:unipcellautom}, we describe the coordinate rings of unipotent cells explicitly, and review ``classical twist maps'', which are defined by Berenstein-Zelevinsky \cite{MR1456321}, in terms of matrix coefficients.

\subsection{Kac-Moody Lie algebras and their representations}

\begin{defn}
\label{d:rootdatum} A root datum $\left(I,\mathfrak{h},P,\left\{ \alpha_{i}\right\} _{i\in I},\left\{ h_{i}\right\} _{i\in I},\left(\ , \right)\right)$
consists of the following data 
\begin{enumerate}
\item $I$ : a finite index set, 
\item $\mathfrak{h}$ : a finite dimensional $\mathbb{Q}$-vector space, 
\item $P\subset\mathfrak{h}^{*}$ : a lattice, called the weight lattice, 
\item $P^{\ast}=\left\{ h\in\mathfrak{h}\mid\left\langle h,P\right\rangle \subset\mathbb{Z}\right\} $,
called the coweight lattice, with the canonical pairing $\left\langle -,-\right\rangle \colon P^{\ast}\otimes_{\mathbb{Z}}P\to\mathbb{Z}$, 
\item $\left\{ \alpha_{i}\right\} _{i\in I}\subset P$ : a subset, called
the set of simple roots, 
\item $\left\{ h_{i}\right\} _{i\in I}\subset P^{\ast}$ : a subset, called
the set of simple coroots, 
\item $\left(\ ,\ \right)\colon P\times P\to\mathbb{Q}$ : a $\mathbb{Q}$-valued
symmetric $\mathbb{Z}$-bilinear form on $P$, 
\end{enumerate}
satisfying the following conditions: 
\begin{enumerate}
\item[(a)] $\left(\alpha_{i},\alpha_{i}\right)\in2\mathbb{Z}_{>0}$ for $i\in I$, 
\item[(b)] $\left\langle h_{i},\mu\right\rangle =2\left(\alpha_{i},\mu\right)/\left(\alpha_{i},\alpha_{i}\right)$
for $\mu\in P$ and $i\in I$, 
\item[(c)] $A=\left(\left\langle h_{i},\alpha_{j}\right\rangle \right)_{i,j\in I}$
is a symmetrizable generalized Cartan matrix, that is $\left\langle h_{i},\alpha_{i}\right\rangle =2$,
$\left\langle h_{i},\alpha_{j}\right\rangle \in\mathbb{Z}_{\leq0}$
for $i\neq j$ and $\left\langle h_{i},\alpha_{j}\right\rangle =0$
is equivalent to $\left\langle h_{j},\alpha_{i}\right\rangle =0$, 
\item[(d)] $\left\{ \alpha_{i}\right\} _{i\in I}\subset\mathfrak{h}^{*}$, $\left\{ h_{i}\right\} _{i\in I}\subset\mathfrak{h}$
are linearly independent subsets. 
\end{enumerate}
The $\mathbb{Z}$-submodule $Q=\sum_{i\in I}\mathbb{Z}\alpha_{i}\subset P$
is called the root lattice, $Q^{\vee}=\sum_{i\in I}\mathbb{Z}h_{i}\subset P^{\ast}$ is called the coroot lattice. We set $Q_{+}=\sum_{i\in I}\mathbb{Z}_{\geq0}\alpha_{i}\subset Q$
and $Q_{-}=-Q_{+}$. For $\xi=\sum_{i\in I}\xi_{i}\alpha_{i}\in Q$,
we set $\height\left(\xi\right)=\sum_{i\in I}\xi_{i}\in\mathbb{Z}$.
Let $P_{+}:=\left\{ \lambda\in P\mid\left\langle h_{i},\lambda\right\rangle \in\mathbb{Z}_{\geq0}\;\text{for all}\;i\in I\right\} $
and we assume that there exists $\left\{ \varpi_{i}\right\} _{i\in I}\subset P_{+}$
such that $\left\langle h_{i},\varpi_{j}\right\rangle =\delta_{ij}$.
Set $\rho:=\sum_{i\in I}\varpi_{i}\in P_{+}$. 
\end{defn}

The quadruple $\left(\mathfrak{h},\left\{ \alpha_{i}\right\} _{i\in I},\left\{ h_{i}\right\} _{i\in I},\left(\ ,\ \right)\right)$
is called a realization of $A$. Let $\mathfrak{g}$ be the associated
Kac-Moody Lie algebra, that is, the Lie algebra $\mathfrak{g}$ over $\mathbb{C}$ which is generated by $\{e_{i},f_{i}\mid i\in I\}\cup \mathfrak{h}$ with the following relations: 
\begin{enumerate}
\item $\mathfrak{h}$ is a vector subspace of $\mathfrak{g}$,
\item $[h,h']=0$ for $h,h'\in\mathfrak{h}$,
\item $[h,e_{i}]=\langle h,\alpha_{i}\rangle e_{i}$ and $[h,f_{i}]=-\langle h,\alpha_{i}\rangle f_{i}$
for $h\in\mathfrak{h}$ and $i\in I$,
\item $[e_{i},f_{j}]=\delta_{ij}h_{i}$ for $i,j\in I$, 
\item $\mathrm{ad}(e_{i})^{1-a_{ij}}(e_{j})=\mathrm{ad}(f_{i})^{1-a_{ij}}(f_{j})=0$
for $i,j\in I$ with $i\neq j$, where $\mathrm{ad}(x)(y)=[x,y]$.
\end{enumerate}
Let $\mathfrak{n}_{+}$(resp. $\mathfrak{n}_{-}$) be the Lie subalgebra
of $\mathfrak{g}$ generated by $\left\{ e_{i}\mid i\in I\right\} $
(resp. $\left\{ f_{i}\mid i\in I\right\} $). Then we have $\mathfrak{g}=\mathfrak{n}_{-}\oplus\mathfrak{h}\oplus\mathfrak{n}_{+}$,
and it is called a triangular decomposition of $\mathfrak{g}$. Let
$\mathfrak{p}_{i}^{+}=\mathfrak{n}_{+}\oplus\mathfrak{h}\oplus\mathbb{C}f_{i}$ and $\mathfrak{p}_{i}^{-}=\mathfrak{n}_{-}\oplus\mathfrak{h}\oplus\mathbb{C}e_{i}$.

Let $\mathfrak{g}=\bigoplus_{\alpha\in\mathfrak{h}^{*}}\mathfrak{g}_{\alpha}$ be its root space decomposition, $\Delta=\left\{ \alpha\in\mathfrak{h}^{*}\mid\mathfrak{g}_{\alpha}\neq0\right\} \setminus\left\{ 0\right\} $
be the set of roots, and $\Delta_{\pm}$ be the subsets of positive
and negative roots. For a Lie algebra $\mathfrak{s}$, its universal
enveloping algebra is denoted by $\mathbf{U}\left(\mathfrak{s}\right)$. 

Let $W$ be the Weyl group associated with the above root datum, that
is the subgroup of $GL\left(\mathfrak{h}^{*}\right)$ which is generated
by simple reflections $\{s_{i}\}_{i\in I}$, where 
\[
s_{i}\left(\mu\right)=\mu-\left\langle h_{i},\mu\right\rangle \alpha_{i}\;\left(\mu\in\mathfrak{h}^{*}\right),
\]
and $\ell\colon W\to\mathbb{Z}_{\geq0}$ be the length function, that
is $\ell\left(w\right)$ is the smallest integer such that there exists
$i_{1},\dots,i_{\ell}\in I$ with $w=s_{i_{1}}s_{i_{2}}\dots s_{i_{\ell}}$.
For $w\in W$, set 
\begin{equation}
I(w):=\{\bm{i}=(i_{1},\dots,i_{\ell(w)})\in I^{\ell(w)}\mid w=s_{i_{1}}\cdots s_{i_{\ell(w)}}\}.\label{eq:reducedexp}
\end{equation}
An element of $I(w)$ is called a reduced word of $w$.

Let $\Delta^{\mathrm{re}}:=W\left\{ \alpha_{i}\right\} _{i\in I}\subset\Delta$
be the set of real roots and we set $\Delta_{\pm}^{\mathrm{re}}:=\Delta_{\pm}\cap\Delta^{\mathrm{re}}$.

\begin{defn}
\textup{(1)} For $\lambda\in P_{+}$, let $V_{\mathbb{C}}\left(\lambda\right)$
be the integrable highest weight $\mathfrak{g}$-module with highest
weight vector $u_{\lambda}$ of highest weight $\lambda$.

\textup{(2)} Let $\mathcal{O}_{\mathrm{int}}\left(\mathfrak{g}\right)$
be the category of integrable $\mathfrak{g}$-modules
$M$ satisfying the following condition:
\begin{enumerate}
\item $M=\bigoplus_{\mu\in P}M_{\mu}$ with $M_{\mu}=\left\{ m\in M\mid h.m=\left\langle h,\mu\right\rangle m\;\text{for all}\;h\in\mathfrak{h}\right\} $
and $\dim M_{\mu}<\infty$ for $\mu\in P$,
\item there exists finitely many $\lambda_{1},\cdots,\lambda_{k}\in P_{+}$
such that $P\left(M\right):=\left\{ \mu\in P\mid M_{\mu}\neq0\right\} \subset\bigcup_{1\leq j\leq k}(\lambda_{j}+Q_{-})$.
\end{enumerate}
\end{defn}
By definition, for a finitely generated (not necessarily integrable) $\mathfrak{g}$-module $M$ satisfying the condition 1 above, the condition for $M\in\mathcal{O}_{\mathrm{int}}\left(\mathfrak{g}\right)$ is equivalent to 
$\dim_{\mathbb{C}}\mathbf{U}\left(\mathfrak{p_{i}^{+}}\right)m<\infty$ for all $i\in I$ and $m\in M$. It is well-known that $\mathcal{O}_{\mathrm{int}}\left(\mathfrak{g}\right)$
is semisimple with its simple object being isomorphic to the integrable
highest weight modules $\left\{ V_{\mathbb{C}}\left(\lambda\right)\mid\lambda\in P_{+}\right\} $. 

Let $\varphi\colon\mathfrak{g}\to\mathfrak{g}$ be the anti-involution
defined by $\varphi\left(e_{i}\right)=f_{i},\varphi\left(f_{i}\right)=e_{i},\varphi\left(h\right)=h$
for $i\in I$ and $h\in\mathfrak{h}$. For $M\in\mathcal{O}_{\mathrm{int}}\left(\mathfrak{g}\right)$,
we denote by $\mathbf{D}_{\varphi}M$ the $\mathfrak{g}$-module $\bigoplus_{\mu\in M}\mathrm{Hom}\left(M_{\mu},\mathbb{C}\right)$
whose $\mathfrak{g}$-module structure is given by 
\[
\left\langle x.f,m\right\rangle =\left\langle f,\varphi\left(x\right).m\right\rangle \;\text{for}\;x\in\mathfrak{g}\;\text{and}\;m\in M.
\]
We note that $\mathbf{D}_{\varphi}M\in\mathcal{O}_{\mathrm{int}}\left(\mathfrak{g}\right)$.
For a $\mathfrak{g}$-module $M$, we denote by $M^{\mathrm{r}}$
the $\mathfrak{g}^{\mathrm{op}}$-module $\left\{ m^{\mathrm{r}}\mid m\in M\right\} $
whose $\mathfrak{g}^{\mathrm{op}}$-module structure is given by
\[
x.\left(m^{\mathrm{r}}\right)=\left(\varphi\left(x\right).m\right)^{\mathrm{r}}\;\text{for}\;x\in\mathfrak{g}\;\text{and}\;m\in M.
\]
We denote by $\mathcal{O}_{\mathrm{int}}^{\mathrm{r}}\left(\mathfrak{g}\right)$
be the category of integrable $\mathfrak{g}^{\mathrm{op}}$-modules
$M^{\mathrm{r}}$ such that $M\in\mathcal{O}_{\mathrm{int}}\left(\mathfrak{g}\right)$.
We interpret the category of $\mathfrak{g}^{\mathrm{op}}$-modules
as the category of right $\mathbf{U}\left(\mathfrak{g}\right)$-modules.

\subsection{(Pro-)unipotent subgroups}\label{ss:prounip}

A subset $\Theta$ of $\Delta_{\pm}$ is called \emph{closed} (resp. \emph{an ideal})
if it satisfies $\left(\Theta+\Theta\right)\cap\Delta_{\pm}\subset\Theta$
(resp. $\left(\Theta+\Delta_{\pm}\right)\cap\Delta_{\pm}\subset\Theta$).
For a closed subset (resp. an ideal) $\Theta\subset\Delta_{\pm}$,
$\mathfrak{n}_{\pm}\left(\Theta\right):=\bigoplus_{\alpha\in\Theta}\mathfrak{g}_{\alpha}$
is a Lie subalgebra (resp. an ideal) of $\mathfrak{n}_{\pm}$.
\begin{exam}
(1) For a Weyl group element $w\in W$, the subsets $\Delta_{\pm}\left(\leq w\right):=\Delta_{\pm}\cap w\Delta_{\mp}$
and $\Delta_{\pm}\left(>w\right):=\Delta_{\pm}\cap w\Delta_{\pm}$
are closed. Let $\mathfrak{n}_{\pm}\left(\leq w\right):=\mathfrak{n}_{\pm}\left(\Delta_{\pm}\left(\leq w\right)\right)$
and $\mathfrak{n}_{\pm}\left(>w\right):=\mathfrak{n}_{\pm}\left(\Delta_{\pm}\left(>w\right)\right)$
be the corresponding subalgebras. We have direct sum decompositions
$\mathfrak{n}_{\pm}=\mathfrak{n}_{\pm}\left(\leq w\right)\oplus\mathfrak{n}_{\pm}\left(>w\right)$
for $w\in W$. For a simple reflection $s_{i}$, we have $\Delta_{+}\cap s_{i}\Delta_{-}=\left\{ \alpha_{i}\right\} $
and $\Delta_{+}\cap s_{i}\Delta_{+}=\Delta_{+}\setminus\left\{ \alpha_{i}\right\} $.
Hence we have direct sum decompositions $\mathfrak{n}_{\pm}=\mathfrak{g}_{\pm\alpha_{i}}\oplus\mathfrak{n}_{i}^{\pm}$,
where $\mathfrak{n}_{i}^{+}=\mathfrak{n}_{+}\left(\Delta_{+}\setminus\left\{ \alpha_{i}\right\} \right)$
and $\mathfrak{n}_{i}^{-}=\mathfrak{n}_{-}\left(\Delta_{-}\setminus\left\{ -\alpha_{i}\right\} \right)$.

(2) For $k\in\mathbb{Z}_{\geq0}$, we set $\Delta_{\pm}^{\geq k}:=\left\{ \alpha\in\Delta_{\pm}\mid\pm\mathrm{ht}\left(\alpha\right)\geq k\right\} $
and $\mathfrak{n}_{\pm}^{\geq k}:=\mathfrak{n}_{\pm}\left(\Delta_{\pm}^{\geq k}\right)$.
Then we have $\left(\Delta_{\pm}^{\geq k}+\Delta_{\pm}\right)\cap\Delta_{\pm}\subset\Delta_{\pm}^{\geq k}$.
Hence $\mathfrak{n}_{\pm}^{\geq k}$ is an ideal of $\mathfrak{n_{\pm}}$.
\end{exam}
It is clear that $\mathfrak{n}_{\pm}/\mathfrak{n}_{\pm}^{\geq k}$
is a finite dimensional nilpotent Lie algebra. We set 
\[
\hat{\mathfrak{n}}_{\pm}=\lim_{\longleftarrow}\mathfrak{n}_{\pm}/\mathfrak{n}_{\pm}^{\geq k}=\prod_{\alpha\in\Delta_{\pm}}\mathfrak{g}_{\alpha}.
\]

Let $\boldsymbol{N}_{\pm}$ be the pro-unipotent group scheme whose
pro-nilpotent pro-Lie algebra is $\hat{\mathfrak{n}}_{\pm}$ that
is defined by 
\[
\boldsymbol{N}_{\pm}=\lim_{\longleftarrow}\exp\left(\mathfrak{n}_{\pm}/\mathfrak{n}_{\pm}^{\geq k}\right)=\mathrm{Spec}\left(\mathbf{U}\left(\mathfrak{n}_{\pm}\right)_{\mathrm{gr}}^{*}\right),
\]
where $\exp\left(\mathfrak{n}_{\pm}/\mathfrak{n}_{\pm}^{\geq k}\right)$
is an unipotent algebraic group whose Lie algebra is the nilpotent
Lie algebra $\mathfrak{n}_{\pm}/\mathfrak{n}_{\pm}^{\geq k}$ and
$\mathbf{U}\left(\mathfrak{n}_{\pm}\right)_{\mathrm{gr}}^{*}$ is
the graded dual of $\mathbf{U}(\mathfrak{n}_{\pm})$ with respect
to the natural $Q_{\pm}$-grading on $\mathbf{U}\left(\mathfrak{n}_{\pm}\right)$
(the degrees of $e_{i}$ and $f_{i}$ are $\alpha_i$, and $-\alpha_{i}$,
respectively). Note that the commutative algebra structure of $\mathbf{U}\left(\mathfrak{n}_{\pm}\right)_{\mathrm{gr}}^{*}$ is induced from the cocommutative usual coalgebra structure of $\mathbf{U}\left(\mathfrak{n}_{\pm}\right)$. Then we have $\mathbb{C}\left[\boldsymbol{N}_{\pm}\right]=\mathbf{U}\left(\mathfrak{n}_{\pm}\right)_{\mathrm{gr}}^{*}$.
It is known that there exists an isomorphism of $\mathbb{C}$-schemes
$\mathrm{Exp}\colon\widehat{\mathfrak{n}}_{\pm}\to\boldsymbol{N}_{\pm}$. 

For a subset $\Theta$ of $\Delta_{+}$ (resp. $\Delta_{-}$), we
set 
\begin{align*}
\widehat{\mathfrak{n}}_{\pm}\left(\Theta\right) & :=\prod_{\alpha\in\Theta}\mathfrak{g}_{\alpha},&
\boldsymbol{N}_{\pm}\left(\Theta\right) & :=\mathrm{Exp}\left(\widehat{\mathfrak{n}}_{\pm}\left(\Theta\right)\right).
\end{align*}
Then $\boldsymbol{N}_{\pm}\left(\Theta\right)$ is a closed subgroup
of $\boldsymbol{N}_{\pm}$ if $\Theta$ is closed and is a normal
subgroup of $\boldsymbol{N}_{\pm}$ if $\Theta$ is an ideal. Let $N_{\pm}\subset\boldsymbol{N}_{\pm}$
be the subgroup which is generated by $\left\{ \boldsymbol{N}_{\pm}\left(\pm\alpha\right)\mid\alpha\in\Delta_{+}^{\mathrm{re}}\right\} $,
which has an ind-group scheme structure.

For a Weyl group element $w\in W$ and $i\in I$, let 
\begin{align*}
N_{\pm}\left(w\right) & :=\boldsymbol{N}_{\pm}\left(\Delta_{\pm}\left(\leq w\right)\right),& \boldsymbol{N}'_{\pm}\left(w\right) & :=\boldsymbol{N}_{\pm}\left(\Delta_{\pm}\left(>w\right)\right), &
\boldsymbol{N}_i^{\pm}&:=\boldsymbol{N}'_{\pm}\left(\Delta_{\pm}\left(>s_i\right)\right). 
\end{align*}
Since $\Delta_{\pm}\cap w\Delta_{\mp}\subset\Delta_{\pm}^{\mathrm{re}}$,
we have $\boldsymbol{N}_{\pm}\left(\Delta_{\pm}\left(\leq w\right)\right)\subset N_{\pm}$.
In fact, $N_{\pm}\left(w\right)$ are unipotent subgroups of $\boldsymbol{N}_{\pm}$
with $\dim\left(N_{\pm}\left(w\right)\right)=\ell\left(w\right)$.

We have the following isomorphisms 
\begin{align*}
\boldsymbol{N}_{\pm} & \simeq\boldsymbol{N}'_{\pm}\left(w\right)\times\left(N_{\pm}\left(w\right)\right) \\
 & \simeq\left(N_{\pm}\left(w\right)\right)\times\boldsymbol{N}'_{\pm}\left(w\right),
\end{align*}
as schemes, see \cite[Lemma 6.1.2]{MR1923198}. We set $N'_{\pm}\left(w\right):=N_{\pm}\cap\boldsymbol{N}'_{\pm}\left(w\right)$.
We also have the decompositions $N_{\pm}\simeq N'_{\pm}\left(w\right)\times N_{\pm}\left(w\right)\simeq N_{\pm}\left(w\right)\times N'_{\pm}\left(w\right)$.

\subsection{Borel subgroups and minimal parabolic subgroups}

Let us fix a root datum $\left(A,P,P^{\vee},\left\{ \alpha_{i}\right\} _{i\in I},\left\{ h_{i}\right\} _{i\in I}\right)$ which gives a realization of $A$. Set $H:=\mathrm{Spec}\left(\mathbb{C}\left[P\right]\right)$. Then $H$ is the algebraic torus whose character lattice is $P$ and whose $\mathbb{C}$-valued points are given by $\mathrm{Hom}_{\mathbb{Z}}\left(P,\mathbb{C}^{*}\right)$. Since 
$\mathbb{C}\left[\boldsymbol{N}_{\pm}\right]=\mathbf{U}\left(\mathfrak{n}_{\pm}\right)_{\mathrm{gr}}^{*}$ are $Q(\subset P)$-graded algebras, we have $H$-actions on $\boldsymbol{N}_{\pm}$. Moreover, since $\boldsymbol{N}_{\pm}\left(\pm\alpha\right)$, $\alpha\in\Delta_{+}^{\mathrm{re}}$ are preserved by these $H$-actions, the subgroups $N_{\pm}$ are also preserved by these $H$-actions.  Let $\boldsymbol{B}_{\pm}=H\ltimes\boldsymbol{N}_{\pm}$, $B_{\pm}=H\ltimes N_{\pm}$ be the semi-direct product groups. 

For $i\in I$, let $G_{i}$ be the reductive group scheme whose Lie
algebra is $\mathfrak{h}\oplus\mathbb{C}e_{i}\oplus\mathbb{C}f_{i}$
with $H$ a Cartan subgroup. Let $\gamma_{i}\colon SL\left(2,\mathbb{C}\right)\to G_{i}$
be the morphism of algebraic groups which is induced by the homomorphism
of Lie algebras given by $e\mapsto e_{i}$ and $f\mapsto f_{i}$.
For a simple reflection $s_{i}$, let $\overline{s}_{i}\in G_{i}$
and $\overline{\overline{s_{i}}}\in G_{i}$ be the lift defined by
\begin{align*}
\overline{s_{i}} & =\gamma_{i}\left(\left[\begin{array}{cc}
0 & -1\\
1 & 0
\end{array}\right]\right)=\exp\left(-e_{i}\right)\exp\left(f_{i}\right)\exp\left(-e_{i}\right),\\
\overline{\overline{s_{i}}} & =\gamma_{i}\left(\left[\begin{array}{cc}
0 & 1\\
-1 & 0
\end{array}\right]\right)=\exp\left(e_{i}\right)\exp\left(-f_{i}\right)\exp\left(e_{i}\right).
\end{align*}

Let $G_{i}^{+}$ (resp. $G_{i}^{-}$) be the subgroup of $G_{i}$
with $\mathfrak{h}\oplus\mathbb{C}e_{i}$ (reps. $\mathfrak{h}\oplus\mathbb{C}f_{i}$)
as its Lie algebra. We have $G_{i}^{\pm}=G_{i}\cap\boldsymbol{B}_{\pm}$
and isomorphism $\boldsymbol{B}_{\pm}=G_{i}^{\pm}\times\boldsymbol{N}_{i}^{\pm}$
as schemes.

For $i\in I$, let $\left(\mathfrak{p}_{i}^{\pm},H\right)\text{-}\mathsf{mod}$
(resp. $\left(\mathfrak{p}_{i}^{\pm},H\right)^{\mathrm{op}}\text{-}\mathsf{mod}$)
be the category of left (resp. right) finite dimensional $P$-weighted
$\mathfrak{h}$-semisimple $\mathbf{U}\left(\mathfrak{p}_{i}^{\pm}\right)$-modules.

Let us consider the following $\mathbb{C}$-algebras:
\[
\mathbb{C}\left[\boldsymbol{P}_{i}^{\pm}\right]:=\left\{ f\in\mathrm{Hom}_{\mathbb{C}}\left(\mathbf{U}\left(\mathfrak{p}_{i}^{\pm}\right),\mathbb{C}\right)\middle|\substack{\mathbf{U}\left(\mathfrak{p}_{i}^{\pm}\right)f\in\left(\mathfrak{p}_{i}^{\pm},H\right)\text{-}\mathsf{mod}\\
f\mathbf{U}\left(\mathfrak{p}_{i}^{\pm}\right)\in\left(\mathfrak{p}_{i}^{\pm},H\right)^{\mathrm{op}}\text{-}\mathsf{mod}
}
\right\} ,
\]
where we consider the $\mathbf{U}\left(\mathfrak{p}_{i}^{\pm}\right)$-bimodule
structure on $\mathrm{Hom}_{\mathbb{C}}\left(\mathbf{U}\left(\mathfrak{p}_{i}^{\pm}\right),\mathbb{C}\right)$
defined by 
\[
\left\langle x.f.y,u\right\rangle =\left\langle f,y.u.x\right\rangle \;\left(x,y\in\mathfrak{p}_{i}^{\pm},f\in\mathrm{Hom}_{\mathbb{C}}\left(\mathbf{U}\left(\mathfrak{p}_{i}^{\pm}\right),\mathbb{C}\right),u\in\mathbf{U}\left(\mathfrak{p}_{i}^{\pm}\right)\right).
\]

Then the coproduct $\mathbf{U}\left(\mathfrak{p}_{i}^{\pm}\right)\to\mathbf{U}\left(\mathfrak{p}_{i}^{\pm}\right)\otimes\mathbf{U}\left(\mathfrak{p}_{i}^{\pm}\right)$
induces a commutative algebra structure on $\mathrm{Hom}_{\mathbb{C}}\left(\mathbf{U}\left(\mathfrak{p}_{i}^{\pm}\right),\mathbb{C}\right)$
and $\mathbb{C}\left[\boldsymbol{P}_{i}^{\pm}\right]$ is a subalgebra
of $\mathrm{Hom}_{\mathbb{C}}\left(\mathbf{U}\left(\mathfrak{p}_{i}^{\pm}\right),\mathbb{C}\right)$.
We define a schemes $\boldsymbol{P}_{i}^{\pm}:=\mathrm{Spec}\left(\mathbb{C}\left[\boldsymbol{P}_{i}^{\pm}\right]\right)$
as spectrum. The product $\mathbf{U}\left(\mathfrak{p}_{i}^{\pm}\right)\otimes\mathbf{U}\left(\mathfrak{p}_{i}^{\pm}\right)\to\mathbf{U}\left(\mathfrak{p}_{i}^{\pm}\right)$
induces the morphism of schemes $\boldsymbol{P}_{i}^{\pm}\times\boldsymbol{P}_{i}^{\pm}\to\boldsymbol{P}_{i}^{\pm}$
and it gives the structure of group scheme on $\boldsymbol{P}_{i}^{\pm}$
and we have decomposition $\boldsymbol{P}_{i}^{\pm}\cong G_{i}\ltimes\boldsymbol{N}_{i}^{\pm}$
and $\boldsymbol{P}_{i}^{\pm}\supset\boldsymbol{B}_{\pm}$ for $i\in I$. See \cite{MR1317626} for more details. 

\subsection{Kac-Moody groups and flag schemes}

Let $G$ be the ``maximal'' Kac-Moody group over $\mathbb{C}$ completed
along the positive roots which is defined in Kumar \cite[6.1.16]{MR1923198}
and let $G^{\min}\subset G$ be the ``minimal'' Kac-Moody group
over $\mathbb{C}$ defined in Kumar \cite[7.4.1]{MR1923198}. They satisfy $\boldsymbol{B}_{+}\subset G$ and $B_+\subset G^{\min}$. See \cite{MR1923198} for details. 
We also introduce the scheme $\boldsymbol{G}_{\infty}$ and its open subscheme
$\boldsymbol{G}$ following Kashiwara \cite{MR1463702} (see also
Kashiwara-Tanisaki \cite{MR1317626})

We define the scheme $\boldsymbol{G}_{\infty}:=\mathrm{Spec}\left(\mathbf{R}_{\mathbb{C}}\left(\mathfrak{g}\right)\right)$
as the spectrum of the ring of ``strongly regular functions'' introduced
by Kac-Peterson \cite{MR717610}, that is 
\begin{align*}
\mathbf{R}_{\mathbb{C}}\left(\mathfrak{g}\right) & :=\left\{ f\in\mathrm{Hom}_{\mathbb{C}}\left(\mathbf{U}\left(\mathfrak{g}\right),\mathbb{C}\right)\middle|\substack{\mathbf{U}\left(\mathfrak{g}\right)f\in\mathcal{O}_{\mathrm{int}}\left(\mathfrak{g}\right)\\
f\mathbf{U}\left(\mathfrak{g}\right)\in\mathcal{O}_{\mathrm{int}}^{\mathrm{r}}\left(\mathfrak{g}\right)
}
\right\} ,
\end{align*}
where we consider the bimodule structure on $\mathrm{Hom}_{\mathbb{C}}\left(\mathbf{U}\left(\mathfrak{g}\right),\mathbb{C}\right)$
defined by 
\[
\left\langle x.f.y,u\right\rangle =\left\langle f,y.u.x\right\rangle \;\left(x,y\in\mathfrak{g},f\in\mathrm{Hom}_{\mathbb{C}}\left(\mathbf{U}\left(\mathfrak{g}\right),\mathbb{C}\right),u\in\mathbf{U}\left(\mathfrak{g}\right)\right).
\]

Let
\[
\Phi=\sum_{\lambda\in P_{+}}\Phi_{\lambda}\colon \bigoplus_{\lambda\in P_+}V_{\mathbb{C}}\left(\lambda\right)^{\mathrm{r}}\otimes V_{\mathbb{C}}\left(\lambda\right)\to\mathbf{R}_{\mathbb{C}}\left(\mathfrak{g}\right)
\]
be the map defined by $\left\langle \Phi_{\lambda}\left(v_{1}^{\mathrm{r}}\otimes v_{2}\right),u\right\rangle =\left(v_{1},u.v_{2}\right)_{\lambda}$
for $v_{1},v_{2}\in V_{\mathbb{C}}\left(\lambda\right)$ and $u\in\mathbf{U}\left(\mathfrak{g}\right)$,
where $\left(\ ,\ \right)_{\lambda}\colon V_{\mathbb{C}}\left(\lambda\right)\otimes V_{\mathbb{C}}\left(\lambda\right)\to\mathbb{C}$
is the symmetric bilinear form on $V\left(\lambda\right)$ such that
$\left(u_{\lambda},u_{\lambda}\right)_{\lambda}=1$ and $\left(x.v_{1},v_{2}\right)_{\lambda}=\left(v_{1},\varphi\left(x\right).v_{2}\right)_{\lambda}$
for $v_{1},v_{2}\in V_{\mathbb{C}}\left(\lambda\right)$ and $x\in\mathfrak{g}$.
It is known \cite[Theorem 1]{MR717610} that $\Phi$ is an isomorphism
of bimodules, called \emph{the Peter-Weyl isomorphism} for symmetrizable
Kac-Moody Lie algebras.

The multiplications $\mathbf{U}\left(\mathfrak{p}_{i}^{-}\right)\otimes\mathbf{U}\left(\mathfrak{g}\right)\to\mathbf{U}\left(\mathfrak{g}\right)$
and $\mathbf{U}\left(\mathfrak{g}\right)\otimes\mathbf{U}\left(\mathfrak{p}_{i}^{+}\right)\to\mathbf{U}\left(\mathfrak{g}\right)$
induce coaction morphisms $\mathbf{R}_{\mathbb{C}}\left(\mathfrak{g}\right)\to\mathbb{C}\left[\boldsymbol{P}_{i}^{-}\right]\otimes\mathbf{R}_{\mathbb{C}}\left(\mathfrak{g}\right)$
and $\mathbf{R}_{\mathbb{C}}\left(\mathfrak{g}\right)\to\mathbf{R}_{\mathbb{C}}\left(\mathfrak{g}\right)\otimes\mathbb{C}\left[\boldsymbol{P}_{i}^{+}\right]$.
Hence we have the morphisms of schemes $\boldsymbol{P}_{i}^{-}\times\boldsymbol{G}_{\infty}\to\boldsymbol{G}_{\infty}$
and $\boldsymbol{G}_{\infty}\times\boldsymbol{P}_{i}^{+}\to\boldsymbol{G}_{\infty}$
which give rise to the left action of $\boldsymbol{P}_{i}^{-}$ and
the right action of $\boldsymbol{P}_{i}^{+}$ on $\boldsymbol{G}_{\infty}$.
The scheme $\boldsymbol{G}_{\infty}$ contains a canonical point $e$.
\begin{defn}
Let $\boldsymbol{G}$ be the open subset of $\boldsymbol{G}_{\infty}$
which is given by the union of subsets $\boldsymbol{P}_{i_{1}}^{-}\cdots\boldsymbol{P}_{i_{m}}^{-}e\boldsymbol{P}_{j_{1}}^{+}\cdots\boldsymbol{P}_{j_{n}}^{+}\subset\boldsymbol{G}_{\infty}$,
that is,
\[
\boldsymbol{G}=\bigcup_{i_{1},\cdots,i_{m},j_{1},\cdots,j_{n}\in I}\boldsymbol{P}_{i_{1}}^{-}\cdots\boldsymbol{P}_{i_{m}}^{-}e\boldsymbol{P}_{j_{1}}^{+}\cdots\boldsymbol{P}_{j_{n}}^{+}.
\]
\end{defn}
\begin{prop}
\textup{(1)} The left $\boldsymbol{P}_{i}^{-}$action on $\boldsymbol{G}$
is free and the right $\boldsymbol{P}_{i}^{+}$ action on $\boldsymbol{G}$
is free.

\textup{(2)} The restricted left $\boldsymbol{B}_{-}$ actions from
the left action of $\boldsymbol{P}_{i}^{-}$ on $\boldsymbol{G}$
and the restricted right $\boldsymbol{B}_{+}$ action on $\boldsymbol{G}$
from the right action of $\boldsymbol{P}_{i}^{+}$ on $\boldsymbol{G}$
are independent of $i\in I$.

\textup{(3)} For $i\in I$ and $g\in G_{i}$, we have $ge=eg$, where
the left action of $G_{i}$ and the right left action of $G_{i}$
are defined via the left action of $\boldsymbol{P}_{i}^{-}$ and right
action of $\boldsymbol{P}_{i}^{+}$ using the decomposition $\boldsymbol{P}_{i}^{\pm}=G_{i}\ltimes\boldsymbol{N}_{i}^{\pm}$.

\end{prop}
Let $\boldsymbol{N}_{-}\times H\times\boldsymbol{N}_{+}\to\boldsymbol{G}$
be the open immersion defined by the ``multiplication'' $\left(x,y,z\right)\mapsto xyz$
and denote its image by $\boldsymbol{G}_{0}$. Let 
\[
\left[\phantom{g}\right]_{-}\times\left[\phantom{g}\right]_{0}\times\left[\phantom{g}\right]_{+}\colon\boldsymbol{G}_{0}\to\boldsymbol{N}_{-}\times H\times\boldsymbol{N}_{+}
\]
 be the inverse morphism of the ``multiplication''. We note that
we use only the left $\boldsymbol{B}_{-}$-action and the right $\boldsymbol{B}_{+}$-action
on $\boldsymbol{G}$. For the minimal Kac-Moody group $G^{\min}$,
it is known the same result holds, see Gei\ss-Leclerc-Schr\"oer \cite[Proposition 7.1]{MR2822235}.

For the Lie algebra anti-involution $\varphi\colon\mathfrak{g}\to\mathfrak{g}$,
let $\varphi\colon\mathbf{U}\left(\mathfrak{g}\right)\to\mathbf{U}\left(\mathfrak{g}\right)$
be the induced anti-involution as a $\mathbb{C}$-algebra. We note
that $\varphi$ induces the anti-isomorphism of group schemes $\boldsymbol{P}_{i}^{\pm}\xrightarrow{\sim}\boldsymbol{P}_{i}^{\mp}$
for $i\in I$ and we have the following commutative diagram
\[
\xymatrix{\mathbf{U}\left(\mathfrak{p_{i}^{-}}\right)\otimes\mathbf{U}\left(\mathfrak{g}\right)\ar[r]\ar[d]_{\mathrm{flip}\circ\left(\varphi\otimes\varphi\right)} & \mathbf{U}\left(\mathfrak{g}\right)\ar[d]^{\varphi}\\
\mathbf{U}\left(\mathfrak{g}\right)\otimes\mathbf{U}\left(\mathfrak{p}_{i}^{+}\right)\ar[r] & \mathbf{U}\left(\mathfrak{g}\right),
}
\]
where the horizontal homomorphisms are multiplications.  
Let $\left(\phantom{x}\right)^{T}\colon\boldsymbol{G}_{\infty}\to\boldsymbol{G}_{\infty}$
be the induced morphism of schemes which intertwines the left $\boldsymbol{P}_{i}^{-}$-action
into the right $\boldsymbol{P}_{i}^{+}$-action and vice versa. It
is clear that $\left(\phantom{x}\right)^{T}$ preserves $\boldsymbol{G}$
and $\boldsymbol{G}_{0}$ by its construction. We denote by $\left(\phantom{x}\right)^{T}$
the restriction of $\left(\phantom{x}\right)^{T}$ to $\boldsymbol{G}$
and $\boldsymbol{G}_{0}$ by abuse of notation. For each real root
$\alpha\in\Delta_{+}$, $\left(\phantom{x}\right)^{T}$ maps $\boldsymbol{N}_{+}\left(\left\{ \alpha\right\} \right)$
to $\boldsymbol{N}_{-}\left(\left\{ -\alpha\right\} \right)$, so
$\left(\phantom{x}\right)^{T}$ induces an involutive map on $G^{\min}$.

\subsection{Schubert cells and Schubert varieties}

For a Weyl group element $w\in W$, we specify two lifts $\overline{w},\overline{\overline{w}}\in G$
of $w\in W$. It is known that $\left\{ \overline{s_{i}}\right\} _{i\in I}$
and $\left\{ \overline{\overline{s_{i}}}\right\} _{i\in I}$ satisfy
braid relations. It follows that the lifts $\overline{w}$ and $\overline{\overline{w}}$
can be uniquely defined by the condition 
\begin{align*}
\overline{w'w''} & =\overline{w'}\cdot\overline{w''},\\
\overline{\overline{w'w''}} & =\overline{\overline{w'}}\cdot\overline{\overline{w''}}
\end{align*}
for $w',w''\in W$ with $\ell\left(w'w''\right)=\ell\left(w'\right)+\ell\left(w''\right)$.
\begin{prop}[{\cite[Proposition 2.1]{MR1652878}}]
 For $w\in W$, we have the following properties: 
\begin{align*}
\overline{w}^{-1} & =\overline{w}^{T}=\overline{\overline{w^{-1}}},\\
\overline{w^{-1}} & =\overline{\overline{w}}^{T}=\overline{\overline{w}}^{-1}.
\end{align*}
\end{prop}
\begin{defn}
The flag scheme $\boldsymbol{X}$ is defined as a quotient scheme
$\boldsymbol{G}/\boldsymbol{B}_{+}$.
\end{defn}
It is known that $\boldsymbol{X}$ is an essentially
smooth and separated (in general, not quasi-compact) scheme over $\mathbb{C}$.
Let $e_{\boldsymbol{X}}=\boldsymbol{B}_{+}/\boldsymbol{B}_{+}\in\boldsymbol{X}$
be the image of $e\in\boldsymbol{G}$.
\begin{notation}
For a set $Y$ with a left (resp.~right) $H$-action and $w\in W$, we write 
$\overline{w}Y$ as $wY$ (resp.~$Y\overline{w}$ as $Yw$).
\end{notation}
\begin{defn}
\textup{(1)} For $w\in W$, we set $\mathring{X}_{w}=wN_{-}\left(w^{-1}\right)e_{X}\subset \boldsymbol{X}$ to
be the locally closed subscheme of $\boldsymbol{X}$. Let $X_{w}$ be the Zariski
closure of $\mathring{X}_{w}$ endowed with the reduced scheme structure.
$X_{w}$ (resp. $\mathring{X}_{w}$) are called \emph{(finite) Schubert
varieties (resp. cells)}.

\textup{(2)} For $w\in W$, we set $\mathring{\boldsymbol{X}}^{w}:=\boldsymbol{B}_{-}\overline{w}e_{\boldsymbol{X}}=\boldsymbol{N}_{-}\overline{w}e_{\boldsymbol{X}}\subset\boldsymbol{X}$
to be the locally closed subscheme of $\boldsymbol{X}$. Let $\boldsymbol{X}^{w}$
be the Zariski closure of $\mathring{\boldsymbol{X}}^{w}$ endowed
with the reduced scheme structure. $\boldsymbol{X}^{w}$(resp. $\mathring{\boldsymbol{X}}^{w}$)
are called \emph{cofinite Schubert schemes (resp. cells)}.

\textup{(2)} For $w\in W$, we set $\boldsymbol{U}_{w}:=w\boldsymbol{B}_{-}e_{\boldsymbol{X}}=w\boldsymbol{N}_{-}e_{\boldsymbol{X}}\subset\boldsymbol{X}$.
\end{defn}
\begin{prop}[{\cite[Proposition 1.3.2]{MR1317626}}]

\textup{(1)} $X_{w}$ is the smallest subscheme of $\boldsymbol{X}$
that is invariant by $G_{i}^{+}$'s and contains $we_{X}$.

\textup{(2)} There is an isomorphism
\[
N_{+}\left(w\right)\to\mathring{X}_{w}
\]
given by $x\mapsto xwe_{X}$. In particular $\mathring{X}_{w}$ is
isomorphic to the affine space $\mathbb{A}^{\ell\left(w\right)}$.

\textup{(3)} We have $X_{w}=\bigsqcup_{y\leq w}\mathring{X}_{y}$,
where $\leq$ is the Bruhat order on $W$.
\end{prop}
We note that the morphism $N_{-}\left(w\right)\times\boldsymbol{B}_{+}\to\boldsymbol{B}_{+}\overline{w}\boldsymbol{B}_{+}$
given by $\left(x,y\right)\mapsto x^{T}\overline{w}y$ is an isomorphism.
\begin{rem}
We note that the union $X:=\bigcup_{w\in W}X_{w}\subset\boldsymbol{X}$
has a structure of an ind-scheme over $\mathbb{C}$ and it is also
called the flag variety. We have isomorphisms $G^{\min}/B_{+}\xrightarrow{\sim}G/\boldsymbol{B}_{+}=X$,
see  \cite[7.4.5 Proposition]{MR1923198}. 
\end{rem}
\begin{prop}[{\cite[Proposition 1.3.1]{MR1317626}}]

\textup{(1)} There is an isomorphism
\[
\boldsymbol{N}'_{-}\left(w\right)\to\mathring{\boldsymbol{X}}^{w}
\]
given by $x\mapsto x\overline{w}e_{\boldsymbol{X}}$. In particular,
$\mathring{\boldsymbol{X}}^{w}$ is affine space with $\mathrm{codim}\boldsymbol{X}^{w}=\ell\left(w\right)$.

\textup{(2)} We have $\boldsymbol{X}=\bigsqcup_{w\in W}\mathring{\boldsymbol{X}}^{w}$.

\textup{(3)} We have $\boldsymbol{X}^{w}=\bigsqcup_{y\geq w}\mathring{\boldsymbol{X}}^{y}$
for $w\in W$.
\end{prop}
\begin{cor}[Birkhoff decomposition]
 \label{cor:Birkhoff}We have $\boldsymbol{G}=\bigsqcup_{w\in W}\boldsymbol{B}_{-}\overline{w}\boldsymbol{B}_{+}$.
\end{cor}
\begin{prop}[\cite{MR1463702}]\label{p:openimm}
$\boldsymbol{U}_{w}$ is an affine open subset of $\boldsymbol{X}$
and there is an isomorphism
\[
\left(\boldsymbol{N}'_{-}\left(w\right)\right)\times\left(N_{+}\left(w\right)\right)\xrightarrow{\sim}\boldsymbol{U}_{w}
\]
which is given by $\left(x,y\right)\mapsto xywe_{X}$. In particular,
$\boldsymbol{N}_{-}\to\boldsymbol{X}$ given by
$n_{-}\mapsto n_{-}e_{\boldsymbol{X}}$ is an open immersion.
\end{prop}

\subsection{Unipotent cells and their automorphisms}\label{ss:unipcellautom}

In this subsection, we consider $\boldsymbol{N}_{-}$ as an open subscheme of $\boldsymbol{X}$ via the open immersion in Proposition \ref{p:openimm}. 
\begin{defn}
For $w\in W$, we set
\begin{align*}
N_{-}^{w} & :=\boldsymbol{N}_{-}\cap\mathring{X}_{w}\subset\boldsymbol{X},
\end{align*}
$N_{-}^{w}$ is called \emph{the unipotent cell.}
\end{defn}
Since $wN_{-}\left(w^{-1}\right)\boldsymbol{B}_{+}\subset G$, we
have 
\begin{align*}
wN_{-}\left(w^{-1}\right)\boldsymbol{B}_{+}\cap\boldsymbol{N}_{-}\boldsymbol{B}_{+} & =wN_{-}\left(w^{-1}\right)\boldsymbol{B}_{+}\cap\boldsymbol{N}_{-}\boldsymbol{B}_{+}\cap G\\
 & =wN_{-}\left(w^{-1}\right)\boldsymbol{B}_{+}\cap N_{-}\boldsymbol{B}_{+}.
\end{align*}
Therefore 
we have $\boldsymbol{N}_{-}\cap\mathring{X}_{w}= N_{-}\cap\mathring{X}_{w}$. Similarly, it can be shown
that $\boldsymbol{N}_{-}\cap X_{w}= N_{-}\cap X_{w}$. Since $\boldsymbol{N}_{-}\subset\boldsymbol{X}$ is a Zariski open immersion, $N_{-}\cap X_{w}=\boldsymbol{N}_{-}\cap X_{w}\subset X_{w}$ is an
open immersion and $N_{-}\cap X_{w}$ is a closed
(affine) subscheme of $\boldsymbol{N}_{-}$. Moreover $N_{-}\cap X_{w}$ is reduced. It also coincides with
the scheme-theoretic intersection. 

We shall describe the coordinate ring of $N_{-}\cap X_{w}$ explicitly, that is, we describe the kernel of the quotient map $\mathbb{C}\left[\boldsymbol{N}_{-}\right]\twoheadrightarrow\mathbb{C}\left[N_{-}\cap X_{w}\right]$, after preparing some notations. For a reduced expression $\boldsymbol{i}=\left(i_{1},\cdots,i_{\ell}\right)\in I\left(w\right)$
of a Weyl group element $w\in W$, we consider a morphism $y_{\boldsymbol{i}}\colon\mathbb{A}^{\ell\left(w\right)}\to\boldsymbol{N}_{-}$
defined by 
\[
y_{\boldsymbol{i}}\left(z_{1},\cdots,z_{\ell}\right):=\exp\left(z_{1}f_{i_{1}}\right)\cdots\exp\left(z_{\ell}f_{i_{\ell}}\right).
\]
We note that the associated ring homomorphism $y_{\boldsymbol{i}}^{*}\colon\mathbb{C}\left[\boldsymbol{N}_{-}\right]\to\mathbb{C}\left[\mathbb{A}^{\ell\left(w\right)}\right]=\mathbb{C}\left[z_{1},\dots,z_{\ell}\right]$
is nothing but the (classical) Feigin map or Gei\ss-Leclerc-Schr\"oer's
$\varphi$-map (see Gei\ss-Leclerc-Schr\"oer \cite[Section 6]{MR2822235}). Moreover, set 
\[
\mathbf{U}_{w}^{-}:=\sum_{a_{1},\cdots,a_{\ell}\in \mathbb{Z}_{\geq 0}}\mathbb{C}f_{i_{1}}^{a_{1}}\cdots f_{i_{\ell}}^{a_{\ell}}\subset\mathbf{U}\left(\mathfrak{n}_{-}\right). 
\]
Then $\mathbf{U}_{w}^{-}$ is independent of the choice of $\bm{i}\in I(w)$ (see also Proposition \ref{p:Demazure} and Remark \ref{r:specialDem}). 
\begin{prop}
\label{p:closedcoord}For $w\in W$, we have isomorphisms of $\mathbb{C}$-algebras:
\[
\mathbb{C}\left[\boldsymbol{N}_{-}\right]/\left(\mathbf{U}_{w}^{-}\right)^{\perp}\xrightarrow{\sim}\mathbb{C}\left[N_{-}\cap X_{w}\right],
\]
here $\left(\mathbf{U}_{w}^{-}\right)^{\perp}:=\left\{ f\in\mathbf{U}\left(\mathfrak{n}_{-}\right)_{\mathrm{gr}}^{*}\mid f(\mathbf{U}_{w}^{-})=0\right\} $
(recall that $\mathbb{C}\left[\boldsymbol{N}_{-}\right]=\mathbf{U}\left(\mathfrak{n}_{-}\right)_{\mathrm{gr}}^{*}$
). 
\end{prop}
\begin{proof}
Let us consider the morphism $y_{\boldsymbol{i}}\colon\mathbb{A}^{\ell\left(w\right)}\xrightarrow{}\boldsymbol{N}_{-}$.
It can be shown that the set-theoretic image of the morphism $y_{\boldsymbol{i}}$
is included in $N_{-}\cap X_{w}$ and the set-theoretic image of $y_{\boldsymbol{i}}\mid_{\mathbb{G}_m^{\ell(w)}}$ is dense in 
$N_{-}\cap\mathring{X}_{w}$ (cf.~\cite[Proposition 7.1.15]{MR1923198}). Since $\mathbb{A}^{\ell\left(w\right)}$
is reduced and the Zariski closure of $N_{-}\cap\mathring{X}_{w}$
is $N_{-}\cap X_{w}$, the scheme-theoretic image of $y_{\boldsymbol{i}}$
(into $\boldsymbol{N}_{-}$) is $N_{-}\cap X_{w}$. The claim follows from the
claim $\left(\mathbf{U}_{w}^{-}\right)^{\perp}=\left\{ f\in\mathbb{C}\left[\boldsymbol{N}_{-}\right]\mid y_{\boldsymbol{i}}^{*}\left(f\right)=0\right\} $
which is clear from the definition of Gei\ss-Leclerc-Schr\"oer's $\varphi$-map.
\end{proof}
We next describe the coordinate rings of $\mathbb{C}\left[N_{-}^{w}\right]$
and $\mathbb{C}\left[N_{-}\left(w\right)\cap wG_{0}^{\min}\right]$. For $\lambda\in P_{+}$, we set $u_{w\lambda}:=\overline{w}u_{\lambda}$.
We set $\Delta_{w\lambda,\lambda}:=\Phi_{\lambda}\left(u_{w\lambda}^{\mathrm{r}}\otimes u_{\lambda}\right)\in R_{\mathbb{C}}\left(\mathfrak{g}\right)$.
We regard $\Delta_{w\lambda,\lambda}$ as a regular function on $\boldsymbol{G}_{\infty}$ and its restriction to $\boldsymbol{G}$. We have the following recognizing criterion of the point in the Schubert cells in the Schubert variety in terms of (unipotent) minors $\Delta_{w\lambda,\lambda}$. It is proved by Williams \cite[Lemma 4.15]{MR3096792}
in the ``minimal'' Kac-Moody group setting. 
\begin{lem}
For $g\in G^{\min}$ and a point $ge_{X}$ on the Schubert variety
$X_{w}$ belongs to the Schubert cell $\mathring{X}_{w}$ if and only
if $\Delta_{w\lambda,\lambda}\left(ge_{X}\right)\neq0$ for $\lambda\in P_{+}$,
where $e_{X}:=B_{+}/B_{+}\in G^{\min}/B_{+}$.
\end{lem}

Since $N_{-}^{w}$ is also reduced, the set-theoretic intersection
$N_{-}\cap\mathring{X}_{w}$ coincides with the scheme-theoretic intersection,
we obtain the following corollary.
\begin{cor}\label{c:cellcoord}
For $w\in W$, we have isomorphisms of $\mathbb{C}$-algebras:
\[
\left(\mathbb{C}\left[\boldsymbol{N}_{-}\right]/\left(\mathbf{U}_{w}^{-}\right)^{\perp}\right)\left[\left[D_{w\lambda,\lambda}^{\mathbb{C}}\right]_{w}^{-1}\mid\lambda\in P_{+}\right]\xrightarrow{\sim}\mathbb{C}\left[N_{-}^{w}\right].
\]
where $\left[D_{w\lambda,\lambda}^{\mathbb{C}}\right]_{w}=\Delta_{w\lambda,\lambda}|_{N_{-}\cap X_{w}}\colon N_{-}\cap X_{w}\to\mathbb{C}$
is the restriction of $\Delta_{w\lambda,\lambda}$ to $N_{-}\cap X_{w}$.
\end{cor}
By Corollary \ref{cor:Birkhoff}, we have 
\[
w\boldsymbol{G}_{0}=\left\{ g\in\boldsymbol{G}\mid\Delta_{w\lambda,\lambda}\left(g\right)\neq0\right\} .
\]
By \cite[Proposition 7.3]{MR2822235} we have $G_{0}^{\min}=\boldsymbol{G}_{0}\cap G^{\min}$,
in particular, we obtain $N_{-}\left(w\right)\cap wG_{0}^{\min}=N_{-}\left(w\right)\cap w\boldsymbol{G}_{0}$.
Hence we have 
\begin{align}
\mathbb{C}\left[N_{-}\left(w\right)\right]\left[(D_{w\lambda,\lambda}^{\mathbb{C}})^{-1}\mid\lambda\in P_{+}\right] \xrightarrow{\sim} \mathbb{C}\left[N_{-}\left(w\right)\cap wG_{0}^{\min}\right],\label{eq:grpcoord}
\end{align}
here $D_{w\lambda,\lambda}^{\mathbb{C}}:=\Delta_{w\lambda,\lambda}|_{N_{-}\left(w\right)}$. 

Our next goal is to show Corollary \ref{c:spusualisom}. This is a classical counterpart of the De Concini-Procesi isomorphisms, which we prove in subsection \ref{ss:DeCP}. We first recall the (classical) twist isomorphism $\gamma_{w}$ and the (classical) twist automorphism $\eta_w$ following Berenstein-Zelevinsky
and Gei\ss-Leclerc-Schr\"oer.
\begin{defn}
	For $w\in W$, let $\boldsymbol{O}_{w}:=\boldsymbol{N}_{-}\cap w\boldsymbol{G}_{0}$.
	We define a map $\tilde{\boldsymbol{\gamma}}_{w}\colon\boldsymbol{O}_{w}\to\boldsymbol{N}_{-}$
	by 
	\[
	\tilde{\boldsymbol{\gamma}}_{w}\left(z\right)=\left[z^{T}\overline{w}\right]_{-}.
	\]
\end{defn}

The following is proved by Berenstein-Zelevinsky \cite{MR1456321}
(see also Geiß-Leclerc-Schröer \cite[Proposition 8.4, Proposition 8.5]{MR2822235}). 
\begin{prop}
	\label{p:BZGLS} The following properties hold:
	
	\textup{(1)} The map $\tilde{\boldsymbol{\gamma}}_{w}\colon\boldsymbol{O}_{w}\to\boldsymbol{N}_{-}$
	is a morphism between schemes.
	
	\textup{(2)} The image of $\tilde{\boldsymbol{\gamma}}_{w}$ is $N_{-}^{w}$.
	
	\textup{(3)} The restriction $\gamma_{w}:=\tilde{\boldsymbol{\gamma}}_{w}|_{N_{-}\left(w\right)\cap wG_{0}^{\min}}\colon N_{-}\left(w\right)\cap wG_{0}^{\min}(=N_{-}\left(w\right)\cap w\boldsymbol{G}_{0})\to N_{-}^{w}$
	is an isomorphism.
	
	\textup{(4)} We have $N_{-}^{w}\subset wG_{0}^{\min}(\subset w\boldsymbol{G}_{0})$
	and $\eta_{w}:=\tilde{\boldsymbol{\gamma}}_{w}|_{N_{-}^{w}}\colon N_{-}^{w}\to N_{-}^{w}$
	is an automorphism.
	
	\textup{(5)} Let $\pi_{w}\colon\boldsymbol{N}_{-}\to N_{-}(w)$ be
	the projection for the isomorphism $\boldsymbol{N}_{-}'\left(w\right)\times N_{-}\left(w\right)\xrightarrow{\sim}\boldsymbol{N}_{-}$
	given by multiplication (see subsection \ref{ss:prounip}). Then $\pi_{w}$
	restricts to $N_{-}^{w}\to N_{-}\left(w\right)\cap wG_{0}^{\min}$,
	and $\eta_{w}=\gamma_{w}\circ\pi_{w}|_{N_{-}^{w}}$. 
\end{prop}

\begin{rem}
	In \cite{MR2822235}, they define a twist isomorphism and and a twist
	automorphism as restrictions of the morphism $\tilde{\gamma}_{w}\colon N_{-}\cap wG_{0}^{\min}\to N_{-},z\mapsto\left[z^{T}\overline{w}\right]_{-}$
	between ind-schemes. Eventually, it turns out that this twist isomorphism
	(resp. twist automorphism) coincides with our $\gamma_{w}$ (resp.
	$\eta_{w}$).
\end{rem}
Let $\pi_{w}^{\ast}\colon \mathbb{C}\left[N_{-}\left(w\right)\right]\hookrightarrow\mathbb{C}\left[\boldsymbol{N}_{-}\right]$ be the $\mathbb{C}$-algebra homomorphism induced from $\pi_{w}$. Then, by \cite[Proposition 8.2]{MR2822235}, the image of $\pi_{w}^{\ast}$ consists of the left $\boldsymbol{N}_{-}'\left(w\right)$-invariant functions in $\mathbb{C}\left[\boldsymbol{N}_{-}\right]$. Note that our convention is the transpose of  Gei\ss-Leclerc-Sch\"orer's convention. Moreover, by the calculation in  \cite[subsection 8.2]{MR2822235}, a function $\Phi_{\lambda}\left(u_{w\lambda}^{\mathrm{r}}\otimes u\right)|_{\boldsymbol{N}_{-}}$ is left $\boldsymbol{N}_{-}'\left(w\right)$-invariant for all $u\in V\left(\lambda\right)$, $\lambda\in P_+$. Hence 
\begin{align}
\pi_{w}^{\ast}(\Phi_{\lambda}\left(u_{w\lambda}^{\mathrm{r}}\otimes u\right)|_{N_{-}\left(w\right)})=\Phi_{\lambda}\left(u_{w\lambda}^{\mathrm{r}}\otimes u\right)|_{\boldsymbol{N}_{-}}.\label{eq:projection} 
\end{align}
\begin{cor}
\label{c:spusualisom} For $w\in W$, we have an isomorphism of $\mathbb{C}$-algebras:
\[
\mathbb{C}\left[N_{-}\left(w\right)\cap wG_{0}^{\min}\right]\xrightarrow{\sim}\mathbb{C}\left[N_{-}^{w}\right],
\]
which is induced by localizing the homomorphism $\mathbb{C}\left[N_{-}\left(w\right)\right]\xrightarrow[]{\pi_{w}^{\ast}}\mathbb{C}\left[\boldsymbol{N}_{-}\right]\twoheadrightarrow\mathbb{C}\left[N_{-}\cap X_{w}\right]$
with respect to $\left\{ D_{w\lambda,\lambda}^{\mathbb{C}}\mid\lambda\in P_{+}\right\}$.
\end{cor}
\begin{proof}
By definition, the composite map $\iota\colon \mathbb{C}\left[N_{-}\left(w\right)\right]\xrightarrow[]{\pi_{w}^{\ast}}\mathbb{C}\left[\boldsymbol{N}_{-}\right]\twoheadrightarrow\mathbb{C}\left[N_{-}\cap X_{w}\right]$ is induced from the morphism of schemes $\pi_w|_{N_{-}\cap X_{w}}\colon N_{-}\cap X_{w}\to N_{-}\left(w\right)$. Moreover, by Corollary \ref{c:cellcoord} and \eqref{eq:grpcoord}, the inclusions $N_{-}^{w}\to N_{-}\cap X_{w}$ and $N_{-}\left(w\right)\cap wG_{0}^{\min}\to N_{-}\left(w\right)$ corresponds to the canonical $\mathbb{C}$-algebra homomorphisms 
\begin{align*}
&\iota_1\colon \mathbb{C}\left[N_{-}\cap X_{w}\right]\to \mathbb{C}\left[N_{-}\cap X_{w}\right]\left[\left[D_{w\lambda,\lambda}^{\mathbb{C}}\right]_{w}^{-1}\mid\lambda\in P_{+}\right]\simeq \mathbb{C}\left[N_{-}^{w}\right],\\
&\iota_2\colon \mathbb{C}\left[N_{-}\left(w\right)\right]\to \mathbb{C}\left[N_{-}\left(w\right)\right]\left[(D_{w\lambda,\lambda}^{\mathbb{C}})^{-1}\mid\lambda\in P_{+}\right] \simeq \mathbb{C}\left[N_{-}\left(w\right)\cap wG_{0}^{\min}\right].
\end{align*}
Therefore the composite map $\iota_1\circ \iota\colon \mathbb{C}\left[N_{-}\left(w\right)\right]\to \mathbb{C}\left[N_{-}^{w}\right]$ is induced from $\pi_w|_{N_{-}^{w}}\colon N_{-}^{w}\to N_{-}\left(w\right)$. Moreover, by \eqref{eq:projection}, $(\iota_1\circ \iota)(D_{w\lambda,\lambda}^{\mathbb{C}})=\left[D_{w\lambda,\lambda}^{\mathbb{C}}\right]_{w}$ for all $u\in V\left(\lambda\right)$, $\lambda\in P_+$. Hence, by the universality of localization, $\iota_1\circ \iota$ extends to $\mathbb{C}\left[N_{-}\left(w\right)\cap wG_{0}^{\min}\right]\to \mathbb{C}\left[N_{-}^{w}\right]$. By construction this is induced from $\pi_{w}|_{N_{-}^{w}}\colon N_{-}^{w}\to N_{-}\left(w\right)\cap wG_{0}^{\min}$, which  is an isomorphism of schemes by Proposition \ref{p:BZGLS}. Hence we obtain the desired isomorphism $\mathbb{C}\left[N_{-}\left(w\right)\cap wG_{0}^{\min}\right]\to \mathbb{C}\left[N_{-}^{w}\right]$. 
\end{proof}
We conclude this subsection by describing the classical twist isomorphism $\gamma_{w}$ in terms of matrix coefficients. 
\begin{prop}
\label{prop:clstwist} Let $\gamma_{w}^{*}\colon \mathbb{C}\left[N_{-}^{w}\right]\to \mathbb{C}\left[N_{-}\left(w\right)\cap wG_{0}^{\min}\right]$ be the isomorphism of $\mathbb{C}$-algebras induced from $\gamma_{w}$. Then, for $w\in W$, $\lambda\in P_{+}$ and $u\in V\left(\lambda\right)$, we have 
\[
\gamma_{w}^{*}\left(\left[D_{u,u_{\lambda}}^{\mathbb{C}}\right]_w\right)=\frac{D_{u_{w\lambda},u}^{\mathbb{C}}}{D_{u_{w\lambda},u_{\lambda}}^{\mathbb{C}}}, 
\]
here $\left[D_{u,u_{\lambda}}^{\mathbb{C}}\right]_w:=\Phi_{\lambda}\left(u^{\mathrm{r}}\otimes u_{\lambda}\right)|_{N_{-}^w}$ and $D_{u_{w\lambda},u}^{\mathbb{C}}:=\Phi_{\lambda}\left(u_{w\lambda}^{\mathrm{r}}\otimes u\right)|_{N_{-}\left(w\right)\cap wG_{0}^{\min}}$ (cf.~Corollary \ref{c:cellcoord}, \eqref{eq:grpcoord}). 
\end{prop}
\begin{proof}
We compute the value of functions. For $z\in N_{-}\left(w\right)\cap O_{w}$,
we have 
\begin{align*}
\left\langle \gamma_{w}^{*}\left(D_{u,u_{\lambda}}^{\mathbb{C}}\right),z\right\rangle  & =\left(u, \gamma_{w}\left(z\right)u_{\lambda}\right)_{\lambda}=\left(u,\left[z^{T}\overline{w}\right]_{-}u_{\lambda}\right)_{\lambda}\\
 & =\left(u,z^{T}\overline{w}u_{\lambda}\right)_{\lambda}/\left(u_{\lambda},\left[z^{T}\overline{w}\right]_{0}u_{\lambda}\right)_{\lambda}\\
 & =\left(u,z^{T}\overline{w}u_{\lambda}\right)_{\lambda}/\left(u_{\lambda},z^{T}\overline{w}u_{\lambda}\right)_{\lambda}\\
 & =\left(zu,u_{w\lambda}\right)_{\lambda}/\left(zu_{\lambda},u_{w\lambda}\right)_{\lambda}.
\end{align*}
Hence we obtained the claim. 
\end{proof}

\section{Preliminaries (2) : Quantized enveloping algebras and canonical bases}\label{sec:Preliminaries2}

\subsection{Quantized enveloping algebras}

In this subsection, we present the definitions of quantized enveloping
algebras. Let $q$ be an indeterminate. 

\begin{notation} \label{n:indeterminate} Set 
\[
\begin{array}{l}
q_{i}:=q^{\frac{(\alpha_{i},\alpha_{i})}{2}},\ {\displaystyle [n]:=\frac{q^{n}-q^{-n}}{q-q^{-1}}\ \text{for\ }n\in\mathbb{Z},}\\
{\displaystyle \left[\begin{array}{c}
n\\
k
\end{array}\right]:=\begin{cases}
{\displaystyle \frac{[n][n-1]\cdots[n-k+1]}{[k][k-1]\cdots[1]}} & \text{if\ }n\in\mathbb{Z},k\in\mathbb{Z}_{>0},\\
1 & \text{if\ }n\in\mathbb{Z},k=0,
\end{cases}}\\
{\displaystyle [n]!:=[n][n-1]\cdots[1]\text{\ for\ }n\in\mathbb{Z}_{>0},[0]!:=1.}
\end{array}
\]
For a rational function $R\in\mathbb{Q}(q)$, we define $R_{i}$ as
the rational function obtained from $R$ by substituting $q$ by $q_{i}$
($i\in I$). \end{notation} 
\begin{defn}
\label{d:QEA} The quantized enveloping algebra $\Uq$ associated
with a root datum $\left(I,\mathfrak{h},P,\left\{ \alpha_{i}\right\} _{i\in I},\left\{ h_{i}\right\} _{i\in I},\left(\ ,\ \right)\right)$
is the unital associative $\mathbb{Q}(q)$-algebra defined by the
generators 
\[
e_{i},f_{i}\;(i\in I),q^{h}\;(h\in P^{*}),
\]
and the relations (i)-(iv) below: 
\begin{enumerate}
\item[(i)] $q^{0}=1,\;q^{h}q^{h'}=q^{h+h'}$ for $h,h'\in P^{*}$,
\item[(ii)] $q^{h}e_{i}=q^{\langle h,\alpha_{i}\rangle}e_{i}q^{h},\;q^{h}f_{i}=q^{-\langle h,\alpha_{i}\rangle}f_{i}q^{h}$
for $h\in P^{*},i\in I$,
\item[(iii)] ${\displaystyle \left[e_{i},f_{j}\right]=\delta_{ij}\frac{t_{i}-t_{i}^{-1}}{q_{i}-q_{i}^{-1}}}$
for $i,j\in I$ where $t_{i}:=q^{\frac{(\alpha_{i},\alpha_{i})}{2}h_{i}}$,
\item[(iv)] ${\displaystyle \sum_{k=0}^{1-a_{ij}}(-1)^{k}\left[\begin{array}{c}
1-a_{ij}\\
k
\end{array}\right]_{i}x_{i}^{k}x_{j}x_{i}^{1-a_{ij}-k}=0}$ for $i,j\in I$ with $i\neq j$, and $x=e,f$. 
\end{enumerate}
The $\mathbb{Q}(q)$-subalgebra of $\Uq$ generated by $\{e_{i}\}_{i\in I}$
(resp.~$\{f_{i}\}_{i\in I}$, $\{q^{h}\}_{h\in P^{*}}$, $\{e_{i},q^{h}\}_{i\in I,h\in P^{*}}$,
$\{f_{i},q^{h}\}_{i\in I,h\in P^{*}}$) will be denoted by $\Uq^{+}$
(resp.~$\Uq^{-}$, $\Uq^{0}$, $\Uq^{\geq0}$, $\Uq^{\leq0}$).

For $\alpha\in Q$, write $(\Uq)_{\alpha}:=\{x\in\Uq\mid q^{h}xq^{-h}=q^{\langle h,\alpha\rangle}x\;\text{for all}\;h\in P^{*}\}$.
The elements of $(\Uq)_{\alpha}$ are said to be homogeneous. For
a homogeneous element $x\in(\Uq)_{\alpha}$, we set $\wt x=\alpha$.
For any subset $X\subset\Uq$ and $\alpha\in Q$, we set $X_{\alpha}:=X\cap(\Uq)_{\alpha}$.

The algebra $\Uq$ has a Hopf algebra structure. In this paper, we
take the coproduct $\Delta\colon\Uq\to\Uq\otimes\Uq$, the counit
$\varepsilon\colon\Uq\to\mathbb{Q}(q)$ and the antipode $S\colon\Uq\to\Uq$
as follows: 
\begin{align*}
\Delta\left(e_{i}\right) & =e_{i}\otimes t_{i}^{-1}+1\otimes e_{i}, & \varepsilon\left(e_{i}\right) & =0, & S\left(e_{i}\right) & =-e_{i}t_{i},\\
\Delta\left(f_{i}\right) & =f_{i}\otimes1+t_{i}\otimes f_{i}, & \varepsilon\left(f_{i}\right) & =0, & S\left(f_{i}\right) & =-t_{i}^{-1}f_{i},\\
\Delta\left(q^{h}\right) & =q^{h}\otimes q^{h}, & \varepsilon\left(q^{h}\right) & =1, & S\left(q^{h}\right) & =q^{-h}.
\end{align*}
\end{defn}
\begin{defn}
\label{d:autom} Let $\vee\colon\Uq\to\Uq$ be the $\mathbb{Q}(q)$-algebra
involution defined by 
\begin{align*}
e_{i}^{\vee} & =f_{i}, &  & f_{i}^{\vee}=e_{i}, &  & \left(q^{h}\right)^{\vee}=q^{-h}.
\end{align*}
Let $\overline{\phantom{x}}\colon\mathbb{Q}(q)\to\mathbb{Q}(q)$ and
$\overline{\phantom{x}}\colon\Uq\to\Uq$ be the $\mathbb{Q}$-algebra
involutions defined by 
\begin{align*}
\overline{q}=q^{-1}, &  & \overline{e_{i}}=e_{i}, &  & \overline{f_{i}}=f_{i}, &  & \overline{q^{h}} & =q^{-h}.
\end{align*}
Let $*,\varphi,\psi\colon\Uq\to\Uq$ be the $\mathbb{Q}\left(q\right)$-algebra
anti-involutions defined by 
\begin{align*}
*(e_{i}) & =e_{i}, & *(f_{i}) & =f_{i}, & *\left(q^{h}\right) & =q^{-h},\\
\varphi\left(e_{i}\right) & =f_{i}, & \varphi\left(f_{i}\right) & =e_{i}, & \varphi\left(q^{h}\right) & =q^{h}\\
\psi\left(e_{i}\right) & =q_{i}^{-1}t_{i}^{-1}f_{i}, & \psi\left(f_{i}\right) & =q_{i}^{-1}t_{i}e_{i}, & \psi\left(q^{h}\right) & =q^{h}.
\end{align*}
Remark that $\psi$ is also a $\mathbb{Q}(q)$-coalgebra homomorphism,
and $\varphi=\vee\circ\ast=\ast\circ\vee$. 
\end{defn}
In this paper, we also use the following variant $\cUq$ of the quantized
enveloping algebra $\Uq$. 
\begin{defn}
\label{d:check_QEA} A variant $\cUq$ of the quantized enveloping
algebra $\Uq$ is the unital associative $\mathbb{Q}(q)$-algebra
defined by the generators 
\[
e_{i},f_{i}\;(i\in I),q^{\mu}\;(\mu\in P),
\]
and the relations (i)-(iv) below: 
\begin{enumerate}
\item[(i)] $q^{0}=1,\;q^{\mu}q^{\mu'}=q^{\mu+\mu'}$ for $\mu,\mu'\in P$,
\item[(ii)] $q^{\mu}e_{i}=q^{(\mu,\alpha_{i})}e_{i}q^{\mu},\;q^{\mu}f_{i}=q^{-(\mu,\alpha_{i})}f_{i}q^{\mu}$
for $\mu\in P,i\in I$,
\item[(iii)] ${\displaystyle \left[e_{i},f_{j}\right]=\delta_{ij}\frac{t_{i}-t_{i}^{-1}}{q_{i}-q_{i}^{-1}}}$
for $i,j\in I$ where $t_{i}:=q^{\alpha_{i}}$ (abuse of notation),
\item[(iv)] ${\displaystyle \sum_{k=0}^{1-a_{ij}}(-1)^{k}\left[\begin{array}{c}
1-a_{ij}\\
k
\end{array}\right]_{i}x_{i}^{k}x_{j}x_{i}^{1-a_{ij}-k}=0}$ for $i,j\in I$ with $i\neq j$, and $x=e,f$. 
\end{enumerate}
The $\mathbb{Q}(q)$-algebra $\cUq$ has a Hopf algebra structure
given by the same formulae as $\Uq$. The notions, notations and maps
defined in Definition \ref{d:QEA} and \ref{d:autom} are immediately
translated into those for $\cUq$. Note that $\cUq^{\pm}$ can be
identified with $\Uq^{\pm}$ via $e_i\mapsto e_i$ and $f_i\mapsto f_i$, respectively. 
\end{defn}

\subsection{Drinfeld pairings and Lusztig pairings}

Some non-degenerate bilinear forms play a role of bridges between quantized
enveloping algebras and their dual objects.
\begin{prop}[\cite{MR934283,MR1187582}]
\label{p:formU} There uniquely exists a $\mathbb{Q}(q)$-bilinear
map $(\ ,\ )_{D}\colon\cUq^{\geq0}\times\Uq^{\leq0}\to\mathbb{Q}(q)$
such that 
\begin{itemize}
\item[(i)] $(\Delta(x),y_{1}\otimes y_{2})_{D}=(x,y_{1}y_{2})_{D}$ for $x\in\cUq^{\geq0},y_{1},y_{2}\in\Uq^{\leq0}$, 
\item[(ii)] $(x_{2}\otimes x_{1},\Delta(y))_{D}=(x_{1}x_{2},y)_{D}$ for $x_{1},x_{2}\in\cUq^{\geq0},y\in\Uq^{\leq0}$, 
\item[(iii)] $(e_{i},q^{h})_{D}=(q^{\mu},f_{i})_{D}=0$ for $i\in I$ and $h\in P^{\ast},\;\mu\in P$, 
\item[(iv)] $(q^{\mu},q^{h})_{D}=q^{-\langle h,\mu\rangle}$ for $\mu\in P,h\in P^{\ast}$, 
\item[(v)] ${\displaystyle (e_{i},f_{j})_{D}=-\delta_{ij}\frac{1}{q_{i}-q_{i}^{-1}}}$
for $i,j\in I$ ,
\end{itemize}
here the $\mathbb{Q}(q)$-bilinear map $(\ ,\ )_{D}\colon\cUq^{\geq0}\otimes\cUq^{\geq0}\times\Uq^{\leq0}\otimes\Uq^{\leq0}\to\mathbb{Q}(q)$
is defined by $(x_{1}\otimes x_{2},y_{1}\otimes y_{2})_{D}=(x_{1},y_{1})_{D}(x_{2},y_{2})_{D}$
for $x_{1},x_{2}\in\cUq^{\geq0},y_{1},y_{2}\in\Uq^{\leq0}$. 
\end{prop}

The bilinear
map $(\ ,\ )_{D}$ is called \emph{the Drinfeld pairing}. It has the following
properties: 
\begin{enumerate}
\item For $\alpha,\beta\in Q_{+}$, ${(\ ,\ )_{D}}\mid_{(\cUq^{\geq0})_{\alpha}\times(\Uq^{\leq0})_{-\beta}}=0$
unless $\alpha=\beta$. 
\item For $\alpha\in Q_{+}$, $\left.(\ ,\ )_{D}\right|_{(\Uq^{+})_{\alpha}\times(\Uq^{-})_{-\alpha}}$
is non-degenerate.
\item $(q^{\mu}x,q^{h}y)_{D}=q^{-\langle h,\mu\rangle}(x,y)_{D}$ for $\mu\in P,\;h\in P^{\ast}$
and $x\in\Uq^{+},y\in\Uq^{-}$.
\end{enumerate}
\begin{defn}
\label{d:qderiv} For $i\in I$, define the $\mathbb{Q}(q)$-linear
maps $e'_{i}$ and $_{i}e'\colon\Uq^{-}\to\Uq^{-}$ by 
\begin{align*}
e'_{i}\left(xy\right) & =e'_{i}\left(x\right)y+q_{i}^{\langle h_{i},\wt x\rangle}xe'_{i}\left(y\right), & e'_{i}(f_{j}) & =\delta_{ij},\\
_{i}e'\left(xy\right) & =q_{i}^{\langle h_{i},\wt y\rangle}{_{i}e'}\left(x\right)y+x{_{i}e'}\left(y\right), & {_{i}e'}(f_{j}) & =\delta_{ij}
\end{align*}
for homogeneous elements $x,y\in\Uq^{-}$. For $i\in I$, define the
$\mathbb{Q}(q)$-linear maps $f'_{i}$ and $_{i}f'\colon\Uq^{+}\to\Uq^{+}$
by 
\begin{align*}
f'_{i}\left(xy\right) & =f'_{i}\left(x\right)y+q_{i}^{-\langle h_{i},\wt x\rangle}xf'_{i}\left(y\right), & f'_{i}(e_{j}) & =\delta_{ij},\\
_{i}f'\left(xy\right) & =q_{i}^{-\langle h_{i},\wt y\rangle}{_{i}f'}\left(x\right)y+x{_{i}f'}\left(y\right), & {_{i}f'}(e_{j}) & =\delta_{ij}
\end{align*}
for homogeneous elements $x,y\in\Uq^{+}$. 
\end{defn}
\begin{defn}
\label{d:Lusform} Define the $\mathbb{Q}(q)$-bilinear form $(\ ,\ )_{L}\colon\Uq^{-}\times\Uq^{-}\to\mathbb{Q}(q)$
by $(x,y)_{L}:=(\psi(x),y)_{D}$ for $x,y\in\Uq^{-}$. Note that $x$
is regarded as an element of $\cUq^{\leq0}$, while $y$ is considered
as an element of $\Uq^{\leq0}$. See Definition \ref{d:check_QEA}.
Then this bilinear form satisfies 
\begin{align*}
(1,1)_{L}=1 & , &  & (f_{i}x,y)_{L}=\frac{1}{1-q_{i}^{2}}(x,e'_{i}(y))_{L}, &  & (xf_{i},y)_{L}=\frac{1}{1-q_{i}^{2}}(x,{_{i}e'}(y))_{L}.
\end{align*}
This is a symmetric bilinear form, called \emph{the Lusztig pairing}. In
fact, $(\ ,\ )_{L}$ is the unique symmetric $\mathbb{Q}(q)$-bilinear
form satisfying the properties above. Moreover, $(\ ,\ )_{L}$ is
non-degenerate and has the following property: 
\begin{equation}
\left(\ast(x),\ast(y)\right)_{L}=\left(x,y\right)_{L}\label{astinv}
\end{equation}
for all $x,y\in\Uq^{-}$.

Similarly, define the $\mathbb{Q}(q)$-bilinear form $(\ ,\ )_{L}^{+}\colon\Uq^{+}\times\Uq^{+}\to\mathbb{Q}(q)$
by $(x,y)_{L}^{+}:=(x,\psi(y))_{D}$ for $x,y\in\Uq^{+}$. Then this
bilinear form satisfies 
\begin{align*}
(1,1)_{L}^{+}=1 & , &  & (e_{i}x,y)_{L}^{+}=\frac{1}{1-q_{i}^{2}}(x,f'_{i}(y))_{L}^{+}, &  & (xe_{i},y)_{L}^{+}=\frac{1}{1-q_{i}^{2}}(x,{_{i}f'}(y))_{L}^{+}.
\end{align*}
The forms $(\ ,\ )_{L}$ and $(\ ,\ )_{L}^{+}$ are related as follows:
\begin{align}
\left(x,y\right)_{L}=\left(x^{\vee},y^{\vee}\right)_{L}^{+}
\label{formrel}
\end{align}
for all $x,y\in\Uq^{-}$. See \cite[Chapter 1]{Lus:intro} for more
details.
\end{defn}
The following Lemma can be proved easily from the definition, it is
left as an exercise for readers.
\begin{lem}
\label{l:DLrel} For $\mu\in P$, $h\in P^{\ast}$, $y_{1},y_{2}\in\Uq^{-}$
and $x_{1},x_{2}\in\Uq^{+}$, we have 
\begin{align*}
(\psi(y_{1}q^{\mu}),y_{2}q^{h})_{D} & =q^{-\langle h,\mu\rangle}(y_{1},y_{2})_{L}, & (x_{1}q^{\mu},\psi(x_{2}q^{h}))_{D} & =q^{-\langle h,\mu\rangle}(x_{1},x_{2})_{L}^{+}.
\end{align*}
\end{lem}
\begin{defn}
\label{d:dualbar} For a homogeneous $x\in\Uq^{-}$, we define $\sigma\left(x\right)=\sigma_{L}\left(x\right)\in\Uq^{-}$
by the property that
\[
\left(\sigma\left(x\right),y\right)_{L}=\overline{\left(x,\overline{y}\right)_{L}}
\]
for an arbitrary $y\in\Uq^{-}$. By the non-degeneracy of $(\;,\;)_{L}$,
the element $\sigma\left(x\right)$ is well-defined. This map $\sigma\colon\Uq^{-}\to\Uq^{-}$
is called \emph{the dual bar-involution}.
\end{defn}

\begin{prop}[{{\cite[Proposition 3.2]{MR2914878}, \cite[Proposition 2.6]{Kimura:2016vn}}}]
\label{p:dualbar} For a homogeneous element $x\in\Uq^{-}$, we have
\[
\sigma\left(x\right)=\left(-1\right)^{\height\left(\wt x\right)}q^{\left(\wt x,\wt x\right)/2-\left(\wt x,\rho\right)}\left(\overline{\phantom{x}}\circ*\right)\left(x\right).
\]
In particular, for homogeneous elements $x,y\in\Uq^{-}$, we have
\[
\sigma(xy)=q^{(\wt x,\wt y)}\sigma(y)\sigma(x).
\]
\end{prop}
\begin{defn}
\label{d:twdualbar} Define a $\mathbb{Q}(q)$-linear isomorphism
$c_{\mathrm{tw}}\colon\Uq^{-}\to\Uq^{-}$ by 
\[
x\mapsto q^{\left(\wt x,\wt x\right)/2-\left(\wt x,\rho\right)}x
\]
for every homogeneous element $x\in\Uq^{-}$. Set $\sigma':=c_{\mathrm{tw}}^{-1}\circ\sigma\colon\Uq^{-}\to\Uq^{-}$.
We call $\sigma'$ \emph{the twisted dual bar involution}. By Proposition
\ref{p:dualbar}, $\sigma'(x)=\left(-1\right)^{\height\left(\wt x\right)}(\overline{\phantom{x}}\circ*)(x)$
for every homogeneous element $x\in\Uq^{-}$. In particular, $\sigma'$
is a $\mathbb{Q}$-algebra anti-involution. 
\end{defn}
\begin{rem}
\label{r:twdualbar} Let $x\in\Uq^{-}$ be a homogeneous element.
Then, 
\begin{center}
$\sigma(x)=x$ if and only if $\sigma'(x)=q^{-\left(\wt x,\wt x\right)/2+\left(\wt x,\rho\right)}x$. 
\par\end{center}
\end{rem}

\subsection{Canonical/Dual canonical bases}

In this subsection, we briefly review the properties of canonical/dual
canonical bases of the quantized enveloping algebras and its integrable
highest weight modules. See, for example, \cite{MR1357199} for the
fundamental results on crystal bases and canonical bases.
\begin{defn}
\label{d:highrep} For $\lambda\in P_{+}$, denote by $V(\lambda)$
the integrable highest weight $\Uq$-module generated by a highest
weight vector $u_{\lambda}$ of weight $\lambda$. Define the surjective
$\Uq^{-}$-module homomorphism $\pi_{\lambda}\colon\mathbf{U}_{q}^{-}\to V\left(\lambda\right)$
by 
\[
\pi_{\lambda}\left(y\right)=y.u_{\lambda}.
\]
There exists a unique $\mathbb{Q}(q)$-bilinear form $(\;,\;)_{\lambda}^{\varphi}\colon V(\lambda)\times V(\lambda)\to\mathbb{Q}(q)$
such that 
\begin{align*}
\left(u_{\lambda},u_{\lambda}\right)_{\lambda}^{\varphi} & =1 & (x.u_{1},u_{2})_{\lambda}^{\varphi} & =(u_{1},\varphi(x).u_{2})_{\lambda}^{\varphi}
\end{align*}
for $u_{1},u_{2}\in V(\lambda)$ and $x\in\Uq$. Then the form $(\;,\;)_{\lambda}^{\varphi}$
is non-degenerate and symmetric. See, for example, \cite[subsection 2.2, the equality (3.10)]{MR3090232}. 
\end{defn}
Set $\mathcal{A}:=\mathbb{Q}[q^{\pm1}]$ and $x_{i}^{(n)}:=x_{i}^{n}/[n]_{i}!\in\Uq$
for $i\in I$, $n\in\mathbb{Z}_{\geq0}$, $x=e,f$. Denote by $\mathbf{U}_{\mathcal{A}}^{-}$
the $\mathcal{A}$-subalgebra of $\Uq^{-}$ generated by the elements
$\{f_{i}^{(n)}\}_{i\in I,n\in\mathbb{Z}_{\geq0}}$ and we set 
\[
\mathbf{A}_{\mathbb{Q}\left[q^{\pm1}\right]}\!\left[\boldsymbol{N}_{-}\right]:=\left\{ x\in\Uq^{-}\mid\left(x,\mathbf{U}_{\mathcal{A}}^{-}\right)_{L}\subset\mathcal{A}\right\}.
\]

Lusztig \cite{MR1035415,MR1088333,Lus:intro} and Kashiwara \cite{MR1115118}
have constructed the specific $\mathbb{Q}(q)$-basis $\mathbf{B}^{\mathrm{low}}$
(resp.~$\mathbf{B}^{\mathrm{low}}(\lambda)$, $\lambda\in P_{+}$)
of $\Uq^{-}$ (resp.~$V(\lambda)$), called \emph{the canonical basis} (or
\emph{the lower global basis}), which is also an $\mathcal{A}$-basis of
$\mathbf{U}_{\mathcal{A}}^{-}$ (resp.~$V\left(\lambda\right)_{\mathcal{A}}:=\mathbf{U}_{\mathcal{A}}^{-}u_{\lambda}$).
Moreover the elements of $\mathbf{B}^{\mathrm{low}}$ (resp.~$\mathbf{B}^{\mathrm{low}}(\lambda)$)
are parametrized by \emph{the Kashiwara crystal} $\mathscr{B}(\infty)$ (resp.~$\mathscr{B}(\lambda)$).
We write 
\begin{align*}
\mathbf{B}^{\mathrm{low}}&=\{\Glow(b)\mid b\in\mathscr{B}(\infty)\}&\mathbf{B}^{\mathrm{low}}(\lambda)&=\{\Glow_{\lambda}(b)\mid b\in\mathscr{B}(\lambda)\}.
\end{align*}
We follow the notation in \cite{MR1357199} concerning the crystal
$(\mathscr{B}(\infty);\wt,\{\tilde{e}_{i}\}_{i\in I},\{\tilde{f}_{i}\}_{i\in I},\{\varepsilon_{i}\}_{i\in I},\{\varphi_{i}\}_{i\in I})$,
$(\mathscr{B}(\lambda);\wt,\{\tilde{e}_{i}\}_{i\in I},\{\tilde{f}_{i}\}_{i\in I},\{\varepsilon_{i}\}_{i\in I},\{\varphi_{i}\}_{i\in I})$.
The unique element of $\mathscr{B}(\infty)$ with weight $0$ is denoted
by $u_{\infty}$, and the unique element of $\mathscr{B}(\lambda)$
with weight $w\lambda$ is denoted by $u_{w\lambda}$ for $\lambda\in P_{+}$
and $w\in W$ by abuse of notation.

Denote by $\mathbf{B}^{\mathrm{up}}$ (resp.~$\mathbf{B}^{\mathrm{up}}(\lambda)$)
the basis of $\Uq^{-}$ (resp.~$V(\lambda)$) dual to $\mathbf{B}^{\mathrm{low}}$
(resp.~$\mathbf{B}^{\mathrm{low}}(\lambda)$) with respect to the
bilinear form $(\ ,\ )_{L}$ (resp.~$(\;,\;)_{\lambda}^{\varphi}$),
that is, $\mathbf{B}^{\mathrm{up}}=\{\Gup(b)\}_{b\in\mathscr{B}(\infty)}$
(resp.~$\mathbf{B}^{\mathrm{up}}(\lambda)=\{\Gup_{\lambda}(b)\}_{b\in\mathscr{B}(\lambda)}$)
such that 

\[
\begin{aligned}(\Glow(b),\Gup(b'))_{L}= & \delta_{b,b'} & (\text{resp.~}(\Glow_{\lambda}(b),\Gup_{\lambda}(b'))_{\lambda}^{\varphi}= & \delta_{b,b'})\end{aligned}
\]

for any $b,b'\in\mathscr{B}(\infty)$ (resp.~$b,b'\in\mathscr{B}(\lambda)$). 
\begin{prop}[{{\cite[Theorem 5, Lemma 7.3.2]{MR1115118}, \cite[Theorem 14.4.11]{Lus:intro}}}]
\label{p:pilambda} Let $\lambda\in P_{+}$. There exists a surjective
map $\pi_{\lambda}\colon\mathscr{B}(\infty)\to\mathscr{B}(\lambda)\coprod\{0\}$
such that 
\[
\pi_{\lambda}(\Glow(b))=\Glow_{\lambda}(\pi_{\lambda}(b))
\]
for $b\in\mathscr{B}(\infty)$, here we set $\Glow_{\lambda}(0):=0$
as a convention. Moreover, $\pi_{\lambda}$ induces a bijection $\pi_{\lambda}^{-1}(\mathscr{B}(\lambda))\to\mathscr{B}(\lambda)$. 
\end{prop}
\begin{defn}
\label{d:jlambda} Let $\lambda\in P_{+}$. Define $j_{\lambda}\colon V\left(\lambda\right)\hookrightarrow\Uq^{-}$
as the dual homomorphism of $\pi_{\lambda}$ given by the non-degenerate
bilinear forms $\left(\;,\;\right)_{\lambda}^{\varphi}\colon V\left(\lambda\right)\times V\left(\lambda\right)\to\mathbb{Q}\left(q\right)$
and $\left(\;,\;\right)_{L}\colon\mathbf{U}_{q}^{-}\times\mathbf{U}_{q}^{-}\to\mathbb{Q}\left(q\right)$,
that is 
\begin{align*}
\left(j_{\lambda}\left(v\right),y\right)_{L} & =\left(v,\pi_{\lambda}\left(y\right)\right)_{\lambda}^{\varphi}(=\left(v,y.u_{\lambda}\right)_{\lambda}^{\varphi}).
\end{align*}
\end{defn}
The following proposition immediately follows from Proposition \ref{p:pilambda}. 
\begin{prop}
\label{p:jlambda} There is an injective map $\overline{\jmath}_{\lambda}\colon\mathscr{B}\left(\lambda\right)\hookrightarrow\mathscr{B}\left(\infty\right)$
such that 
\[
\left(G_{\lambda}^{\mathrm{up}}\left(b\right),G^{\mathrm{low}}\left(b'\right).u_{\lambda}\right)_{\lambda}^{\varphi}=\delta_{b',\overline{\jmath}_{\lambda}\left(b\right)}
\]
for any $b\in\mathscr{B}\left(\lambda\right)$ and $b'\in\mathscr{B}\left(\infty\right)$.
That is, we have $j_{\lambda}\left(G_{\lambda}^{\mathrm{up}}\left(b\right)\right)=G^{\mathrm{up}}\left(\overline{\jmath}_{\lambda}\left(b\right)\right)$. 
\end{prop}
\begin{rem}
\label{r:jlambda} Let $\lambda\in P_{+}$. Then, 
\begin{itemize}
\item $\wt\overline{\jmath}_{\lambda}(b)=\wt b-\lambda$ for $b\in\mathscr{B}(\lambda)$,
and 
\item $\overline{\jmath}_{\lambda}\left(\pi_{\lambda}(b)\right)=b$ for
$b\in\pi_{\lambda}^{-1}(\mathscr{B}(\lambda))$. 
\end{itemize}
\end{rem}
\begin{prop}[{{\cite[Lemma 7.3.4]{MR1115118}, \cite[13.1.11]{Lus:intro}}}]
\label{p:barinv} For all $b\in\mathscr{B}(\infty)$, we have $\overline{\Glow(b)}=\Glow(b)$.
\end{prop}
Note that Proposition \ref{p:barinv} immediately implies $\sigma(\Gup(b))=\Gup(b)$ for $b\in\mathscr{B}(\infty)$. 
\begin{prop}[{{\cite[Theorem 2.1.1]{MR1240605}}}]
\label{p:Kasinv} There exists an bijection $\ast\colon\mathscr{B}(\infty)\rightarrow\mathscr{B}(\infty)$
such that 
\[
\ast(\Glow(b))=\Glow(\ast b)
\]
for $b\in\mathscr{B}(\infty)$.  
\end{prop}
Remark that Proposition \ref{p:Kasinv} implies $\ast(\Gup(b))=\Gup(\ast b)$ for $b\in\mathscr{B}(\infty)$.
See the equality (\ref{astinv}).
\begin{defn}
The bijections $\ast$ give new crystal structures on $\mathscr{B}(\infty)$,
defined by the maps 
\[
\wt^{\ast}:=\wt\circ\ast=\wt,\;\varepsilon_{i}^{\ast}:=\varepsilon\circ\ast,\;\varphi_{i}^{\ast}:=\varphi\circ\ast,\;\tilde{e}_{i}^{\ast}:=\ast\circ\tilde{e}_{i}\circ\ast,\;\tilde{f}_{i}^{\ast}:=\ast\circ\tilde{f}_{i}\circ\ast.
\]
\end{defn}
\begin{prop}[{\cite[Theorem 7]{MR1115118}\cite[Theorem 2.1.1]{MR1240605}}]
\label{p:jlambdaimage} Let $\lambda\in P_{+}$. Then we have 
\[
\overline{\jmath}_{\lambda}(\mathscr{B}(\lambda))=\{\tilde{b}\in\mathscr{B}(\infty)\mid\varepsilon_{i}^{\ast}(\tilde{b})\leq\langle h_{i},\lambda\rangle\;\text{for all}\;i\in I\}.
\]
\end{prop}
\begin{prop}[{\cite[Lemma 5.1.1]{MR1203234}}]
\label{p:EFaction}For $i\in I$, $\lambda\in P$, $b\in\mathscr{B}(\lambda)$
and $\tilde{b}\in\mathscr{B}(\infty)$, we have 
\begin{align*}
e_{i}^{(\varepsilon_{i}(b))}.\Gup_{\lambda}(b) & =\Gup_{\lambda}(\tilde{e}_{i}^{\varepsilon_{i}(b)}b) & e_{i}^{(k)}.\Gup_{\lambda}(b) & =0\ \text{if}\ k>\varepsilon_{i}(b),\\
f_{i}^{(\varphi_{i}(b))}.\Gup_{\lambda}(b) & =\Gup_{\lambda}(\tilde{f}_{i}^{\varphi_{i}(b)}b) & f_{i}^{(k)}.\Gup_{\lambda}(b) & =0\ \text{if}\ k>\varphi_{i}(b),\\
(e_{i}')^{(\varepsilon_{i}(\tilde{b}))}\Gup(\tilde{b}) & =(1-q_{i}^{2})^{\varepsilon_{i}(\tilde{b})}\Gup(\tilde{e}_{i}^{\varepsilon_{i}(\tilde{b})}\tilde{b}) & (e_{i}')^{(k)}\Gup(\tilde{b}) & =0\ \text{if}\ k>\varepsilon_{i}(\tilde{b}),\\
({_{i}e'})^{(\varepsilon_{i}^{\ast}(\tilde{b}))}\Gup(\tilde{b}) & =(1-q_{i}^{2})^{\varepsilon_{i}^{\ast}(\tilde{b})}\Gup((\tilde{e}_{i}^{\ast})^{\varepsilon_{i}^{\ast}(\tilde{b})}\tilde{b}) & ({_{i}e'})^{(k)}\Gup(\tilde{b}) & =0\ \text{if}\ k>\varepsilon_{i}^{\ast}(\tilde{b}).
\end{align*}
Here $(e_{i}')^{(n)}:=(e_{i}')^{n}/[n]_{i}!$ and $({_{i}e'})^{(n)}:=({_{i}e'})^{n}/[n]_{i}!$
for $n\in\mathbb{Z}_{\geq0}$. 
\end{prop}

\subsection{Quantum unipotent subgroups}

In this subsection, we review the quantum unipotent subgroup
$\mathbf{A}_{q}\!\left[N_{-}\left(w\right)\right]$ which is a quantum
analogue of the coordinate ring $\mathbb{C}\left[N_{-}\left(w\right)\right]$
of the unipotent subgroup $N_{-}\left(w\right)$ associated with $w\in W$. See Theorem \ref{t:grpspecial} below for the precise statement. 
\begin{defn}
\label{d:braidaction} Following Lusztig \cite[Section 37.1.3]{Lus:intro},
we define the $\mathbb{Q}\left(q\right)$-algebra automorphism $T_{i}\colon\Uq\to\Uq$
for $i\in I$ by the following formulae: \begin{subequations} 
\begin{align*}
T_{i}\left(q^{h}\right) & =q^{s_{i}\left(h\right)},\\
T_{i}\left(e_{j}\right) & =\begin{cases}
-f_{i}t_{i} & \text{for}\;j=i,\\
{\displaystyle \sum_{r+s=-\left\langle h_{i},\alpha_{j}\right\rangle }\left(-1\right)^{r}q_{i}^{-r}e_{i}^{\left(s\right)}e_{j}e_{i}^{\left(r\right)}} & \text{for}\;j\neq i,
\end{cases}\\
T_{i}\left(f_{j}\right) & =\begin{cases}
-t_{i}^{-1}e_{i} & \text{for}\;j=i,\\
{\displaystyle \sum_{r+s=-\left\langle h_{i},\alpha_{j}\right\rangle }\left(-1\right)^{r}q_{i}^{r}f_{i}^{\left(r\right)}f_{j}f_{i}^{\left(s\right)}} & \text{for}\;j\neq i.
\end{cases}
\end{align*}
\end{subequations} Its inverse map is given by \begin{subequations}
\begin{align*}
T_{i}^{-1}\left(q^{h}\right) & =q^{s_{i}\left(h\right)},\\
T_{i}^{-1}\left(e_{j}\right) & =\begin{cases}
-t_{i}^{-1}f_{i} & \text{for}\;j=i,\\
{\displaystyle \sum_{r+s=-\left\langle h_{i},\alpha_{j}\right\rangle }\left(-1\right)^{r}q_{i}^{-r}e_{i}^{\left(r\right)}e_{j}e_{i}^{\left(s\right)}} & \text{for}\;j\neq i,
\end{cases}\\
T_{i}^{-1}\left(f_{j}\right) & =\begin{cases}
-e_{i}t_{i} & \text{for}\;j=i,\\
{\displaystyle \sum_{r+s=-\left\langle h_{i},\alpha_{j}\right\rangle }\left(-1\right)^{r}q_{i}^{r}f_{i}^{\left(s\right)}f_{j}f_{i}^{\left(r\right)}} & \text{for}\;j\neq i.
\end{cases}
\end{align*}
\end{subequations} The maps $T_{i}$ and $T_{i}^{-1}$ are denoted
by $T''_{i,1}$ and $T'_{i,-1}$ respectively in \cite{Lus:intro}.

It is known that $\left\{ T_{i}\right\} _{i\in I}$ satisfies the
braid relations, that is, for $w\in W$, the $\mathbb{Q}\left(q\right)$-algebra
automorphism $T_{w}:=T_{i_{1}}\cdots T_{i_{\ell}}\colon\Uq\to\Uq$
does not depend on the choice of $(i_{1},\dots,i_{\ell})\in I(w)$
(recall (\ref{eq:reducedexp})). See \cite[Chapter 39]{Lus:intro}. 
\end{defn}
\begin{defn}
\label{d:qnilp} (1) For $w\in W$, we set $\Uq^{-}\left(w\right):=\Uq^{-}\cap T_{w}\left(\Uq^{\geq0}\right)$.
These subalgebras of $\Uq^{-}$ are called quantum nilpotent subalgebras.

(2) Let $w\in W$ and $\bm{i}=\left(i_{1},\cdots,i_{\ell}\right)\in I\left(w\right)$.
For $\bm{c}=\left(c_{1},\cdots,c_{\ell}\right)\in\mathbb{Z}_{\geq0}^{\ell}$,
we set 
\begin{align*}
F^{{\rm low}}\left(\bm{c},\bm{i}\right) & :=f_{i_{1}}^{\left(c_{1}\right)}T_{i_{1}}\left(f_{i_{2}}^{\left(c_{2}\right)}\right)\cdots\left(T_{i_{1}}\cdots T_{i_{\ell-1}}\right)\left(f_{i_{\ell}}^{\left(c_{\ell}\right)}\right),\\
F^{\mathrm{up}}\left(\bm{c},\bm{i}\right) & :=F^{{\rm low}}\left(\bm{c},\bm{i}\right)/\left(F^{{\rm low}}\left(\bm{c},\bm{i}\right),F^{{\rm low}}\left(\bm{c},\bm{i}\right)\right)_{L}.
\end{align*}
\end{defn}
\begin{prop}[{{\cite[Proposition 2.2]{MR1351503}, \cite[Proposition 2.3]{MR1712630},
\cite[Proposition 38.2.3]{Lus:intro}}}]
\label{p:PBWbasis}

\textup{(1)} $F^{{\rm low}}\left(\bm{c},\bm{i}\right)\in\Uq^{-}\left(w\right)$
for $\bm{c}\in\mathbb{Z}_{\geq0}^{\ell}$.

\textup{(2)} $\{F^{{\rm low}}\left(\bm{c},\bm{i}\right)\}_{\bm{c}\in\mathbb{Z}_{\geq0}^{\ell}}$
is an orthogonal basis of $\Uq^{-}\left(w\right)$ with respect to
the pairing $(\ ,\ )_{L}$, more precisely, we have 
\begin{align}
(F^{{\rm low}}\left(\bm{c},\bm{i}\right),F^{{\rm low}}\left(\bm{c}',\bm{i}\right))_{L}=\delta_{\bm{c},\bm{c}'}\prod_{k=1}^{\ell}\prod_{j=1}^{c_{k}}(1-q_{i_{k}}^{2j})^{-1},\label{norm}
\end{align}

here $\bm{i}=(i_{1,\dots,}i_{\ell})$. 
\end{prop}
By Proposition (\ref{p:PBWbasis}), $\{F^{{\rm up}}\left(\bm{c},\bm{i}\right)\}_{\bm{c}\in\mathbb{Z}_{\geq0}^{\ell}}$
is also an orthogonal basis of $\Uq^{-}\left(w\right)$ with respect
to the Lusztig pairing. The basis $\{F^{{\rm low}}\left(\bm{c},\bm{i}\right)\}_{\bm{c}\in\mathbb{Z}_{\geq0}^{\ell}}$
is called \emph{the (lower) Poincaré-Birkhoff-Witt type basis} associated
with $\bm{i}\in I\left(w\right)$, and the basis $\{F^{{\rm up}}\left(\bm{c},\bm{i}\right)\}_{\bm{c}}$
is called \emph{the dual (or upper) Poincaré-Birkhoff-Witt type basis}. 
\begin{defn}
\label{d:qunip} For $w\in W$, we set 
\begin{align*}
\Uq^{+}(w) & :=\left(\Uq^{-}(w)\right)^{\vee},\\
\mathbf{A}_{q}\!\left[N_{-}\left(w\right)\right] & :=\ast\left(\Uq^{-}(w)\right).
\end{align*}
We call $\Aq[N_{-}\left(w\right)]$ \emph{a quantum unipotent subgroup}.
The quantum unipotent subgroup has a $Q_{-}$-graded algebra structure
induced from that of $\Uq^{-}$. Note that $\varphi\left(\Aq[N_{-}\left(w\right)]\right)=\Uq^{+}(w)$. 
\end{defn}
\begin{prop}[{{{\cite[Theorem 4.25, Theorem 4.29]{MR2914878}}}}]
\label{p:dualcanonical} Let $w\in W$ and $\bm{i}\in I\left(w\right)$.
Then the following hold: 

\textup{(1)} $\Uq^{-}(w)\cap\mathbf{B}^{\mathrm{up}}$ is a basis
of $\Uq^{-}(w)$.

\textup{(2)} each element $\Gup(b)$ of $\Uq^{-}(w)\cap\mathbf{B}^{\mathrm{up}}$
is characterized by the following conditions:
\begin{enumerate}
\item[(DCB1)] $\sigma(\Gup(b))=\Gup(b)$, and 
\item[(DCB2)] $\Gup\left(b\right)=F^{{\rm {up}}}\left(\bm{c},\bm{i}\right)+\sum_{\bm{c}'<\bm{c}}d_{\bm{c},\bm{c}'}^{\bm{i}}F^{{\rm {up}}}\left(\bm{c}',\bm{i}\right)$
with $d_{\bm{c},\bm{c}'}^{\bm{i}}\in q\mathbb{Z}[q]$ for some $\bm{c}\in\mathbb{Z}_{\geq0}^{\ell}$. 
\end{enumerate}
Here $<$ denotes the left lexicographic order on $\mathbb{Z}_{\geq0}^{\ell}$,
that is, we write $(c_{1},\dots,c_{\ell})<(c'_{1},\dots,c'_{\ell})$
if and only if there exists $k\in\{1,\dots,\ell\}$ such that $c_{1}=c'_{1},\dots,c_{k-1}=c'_{k-1}$
and $c_{k}<c'_{k}$. 
\end{prop}
\begin{defn}
\label{d:PBWparam} Proposition \ref{p:dualcanonical} (2) says that
each $F^{{\rm {up}}}\left(\bm{c},\bm{i}\right)$ determines a unique
dual canonical basis element $\Gup(b)$ in $\Uq^{-}(w)$. We write
the corresponding element of $\mathscr{B}(\infty)$ as $b\left(\bm{c},\bm{i}\right)$.
Then 
\[
\Uq^{-}(w)\cap\mathbf{B}^{\mathrm{up}}=\{\Gup(b(\bm{c},\bm{i}))\}_{\bm{c}\in\mathbb{Z}_{\geq0}^{\ell}}.
\]
Write $\mathscr{B}(\Uq^{-}(w)):=\{b\left(\bm{c},\bm{i}\right)\}_{\bm{c}\in\mathbb{Z}_{\geq0}^{\ell}}$.
Note that $\mathscr{B}(\Uq^{-}(w))$ does not depend on the choice
of $\bm{i}\in I(w)$. Set $b_{-1}\left(\bm{c},\bm{i}\right):=\ast\left(b(\bm{c},\bm{i})\right)$.
Then $\Aq[N_{-}\left(w\right)]\cap\mathbf{B}^{\mathrm{up}}=\{\Gup(b_{-1}(\bm{c},\bm{i}))\}_{\bm{c}\in\mathbb{Z}_{\geq0}^{\ell}}$. 
\end{defn}

The following is the specialization result for the quantum unipotent
subgroup which justifies the notation $\Aq\left[N_{-}\left(w\right)\right]$. 
\begin{thm}[{\cite[Theorem 4.44]{MR2914878}}]\label{t:grpspecial}
For $w\in W$, we set $\mathbf{A}_{\mathbb{Q}\left[q^{\pm1}\right]}\!\left[N_{-}\left(w\right)\right]:=\mathbf{A}_{\mathbb{Q}\left[q^{\pm1}\right]}\!\left[\boldsymbol{N}_{-}\right]\cap\mathbf{A}_{q}\!\left[N_{-}\left(w\right)\right]$.
Then we have 
\[
\mathbf{A}_{\mathbb{Q}\left[q^{\pm1}\right]}\!\left[N_{-}\left(w\right)\right]\otimes_{\mathcal{A}}\mathbb{C}\simeq\mathbb{C}\left[N_{-}\left(w\right)\right],
\]
here we regard $\mathbb{C}$ as an $\mathcal{A}$-module via $q^{\pm1}\mapsto1$. 
\end{thm}
\begin{rem}
In \cite{MR2914878}, the $\mathcal{A}$-form $\mathbf{A}_{\mathbb{Q}\left[q^{\pm1}\right]}\!\left[N_{-}\left(w\right)\right]$
is defined by the non-degenerate bilinear form $\left(\ ,\ \right)_{K}$
on $\mathbf{U}_{q}^{-}\left(\mathfrak{g}\right)$ with $\left(f_{i},f_{i}\right)_{K}=1$
for $i\in I$. But this specialization result is not affected since
the structure constants with respect to the dual canonical bases defined
by $\left(\ ,\ \right)_{L}$ and $\left(\ ,\ \right)_{K}$ are the same.
For more details, see \cite[Lemma 2.12]{MR2914878}.
\end{rem}

\subsection{Quantum closed unipotent cells}

In this section, we review the definition of quantum closed unipotent
cells. For more details, see \cite[Section 5]{MR2914878}.
\begin{defn}
\label{d:repbraid} Let $M=\bigoplus_{\mu\in P}M_{\mu}$ be an integrable $\Uq$-module (i.e., $e_{i}$ and $f_{i}$ act
locally nilpotently on $M$ for all $i\in I$) with weight space decomposition. For $i\in I$, there exists a $\mathbb{Q}(q)$-linear
automorphism $T_{i}$ of $M$ given by 
\begin{align*}
 & T_{i}(m):=\sum_{-a+b-c=\langle h_{i,}\mu\rangle}(-1)^{b}q_{i}^{-ac+b}e_{i}^{(a)}f_{i}^{(b)}e_{i}^{(c)}.m,\\
 & T_{i}^{-1}(m)=\sum_{a-b+c=\langle h_{i},\mu\rangle}(-1)^{b}q_{i}^{ac-b}f_{i}^{(a)}e_{i}^{(b)}f_{i}^{(c)}.m
\end{align*}
for $m\in M_{\mu}$, $\mu\in P$. The maps $T_{i}$ and $T_{i}^{-1}$
are denoted by $T_{i}=T''_{i,1}$ and $T_{i}=T'_{i,-1}$ respectively
in \cite[Chapter 5]{Lus:intro}. 
\end{defn}

The following propositions are fundamental properties of $T_{i}$.
See, for example, \cite[Chapter 37, 39]{Lus:intro}: 
\begin{prop}
\label{p:braid} Let $M$ be an integrable $\Uq$-module. (See Definition
\ref{d:repbraid}.) \textup{(1)} For $x\in\Uq$ and $m\in M$, we
have $T_{i}(x.m)=T_{i}(x).T_{i}(m)$.

\textup{(2)} For $w\in W$, the composite map $T_{w}:=T_{i_{1}}\cdots T_{i_{\ell}}\colon M\to M$
does not depend on the choice of $(i_{1},\dots,i_{\ell})\in I(w)$.

\textup{(3)} For $\mu\in P$ and $w\in W$, $T_{w}$ maps $M_{\mu}$
to $M_{w\mu}$.
\end{prop}
\begin{prop}
\label{p:extremal} Let $\lambda\in P_{+}$, $w\in W$ and $\boldsymbol{i}=\left(i_{1},\dots,i_{\ell}\right)\in I\left(w\right)$.
Recall that $u_{\lambda}$ is a highest weight vector of $V(\lambda)$
(Definition \ref{d:highrep}). Then we have 
\[
u_{w\lambda}:=(T_{w^{-1}})^{-1}(u_{\lambda})=f_{i_{1}}^{\left(a_{1}\right)}f_{i_{2}}^{\left(a_{2}\right)}\dots f_{i_{\ell}}^{\left(a_{\ell}\right)}.u_{\lambda},
\]
where $a_{1}=\left\langle h_{i_{1}},s_{i_{2}}\dots s_{i_{\ell}}\lambda\right\rangle ,\dots,a_{\ell}=\left\langle h_{i_{\ell}},\lambda\right\rangle $.
Note that $a_{1},\dots,a_{\ell}\in\mathbb{Z}_{\geq0}$. 
\end{prop}
It is easy to show that $\left(u_{w\lambda},u_{w\lambda}\right)_{\lambda}^{\varphi}=1$
for $\lambda\in P_{+}$ and $w\in W$. Actually, the vector $u_{w\lambda}$
belongs to $\mathbf{B}^{\mathrm{low}}(\lambda)$ and $\mathbf{B}^{\mathrm{up}}(\lambda)$
\cite[subsection 3.2]{MR1240605}. 
\begin{prop}[{{\cite[Proposition 3.2.3, 3.2.5]{MR1240605}}}]
\label{p:Demazure} \textup{(1)} For $\lambda\in P_{+}$, $w\in W$
and $\boldsymbol{i}=\left(i_{1},\cdots,i_{\ell}\right)\in I\left(w\right)$,
we set 
\[
\mathscr{B}_{w}\left(\lambda\right):=\left\{ \widetilde{f}_{i_{1}}^{a_{1}}\cdots\widetilde{f}_{i_{\ell}}^{a_{\ell}}u_{\lambda}\mid\boldsymbol{a}=\left(a_{1},\cdots,a_{\ell}\right)\in\mathbb{Z}_{\geq0}^{\ell}\right\} \setminus\left\{ 0\right\} \subset\mathscr{B}\left(\lambda\right)
\]
and $V_{w}(\lambda):=\Uq^{+}.u_{w\lambda}$. Then we have 
\[
V_{w}(\lambda)=\bigoplus_{b\in\mathscr{B}_{w}\left(\lambda\right)}\mathbb{Q}\left(q\right)G_{\lambda}^{\mathrm{low}}\left(b\right).
\]
In particular, $\mathscr{B}_{w}\left(\lambda\right)$ is independent of the choice of $\bm{i}\in I(w)$. 

\textup{(2)} For $w\in W$ and $\boldsymbol{i}=\left(i_{1},\cdots,i_{\ell}\right)\in I\left(w\right)$,
we set 
\[
\mathscr{B}_{w}\left(\infty\right):=\left\{ \widetilde{f}_{i_{1}}^{a_{1}}\cdots\widetilde{f}_{i_{\ell}}^{a_{\ell}}u_{\infty}\mid\boldsymbol{a}=\left(a_{1},\cdots,a_{\ell}\right)\in\mathbb{Z}_{\geq0}^{\ell}\right\} 
\]
and $\mathbf{U}_{w,q}^{-}:=\sum_{a_{1},\cdots,a_{\ell}\in \mathbb{Z}_{\geq 0}}\mathbb{Q}\left(q\right)f_{i_{1}}^{a_{1}}\cdots f_{i_{\ell}}^{a_{\ell}}$.
Then we have 
\[
\mathbf{U}_{w,q}^{-}=\bigoplus_{b\in\mathscr{B}_{w}\left(\infty\right)}\mathbb{Q}\left(q\right)G^{\mathrm{low}}\left(b\right).
\]
Moreover $\mathscr{B}_{w}\left(\infty\right)$ and $\mathbf{U}_{w,q}^{-}$ are independent of the choice of $\bm{i}\in I(w)$. 
\end{prop}
\begin{rem}\label{r:specialDem}
By Proposition \ref{p:Demazure} (2), $\mathbf{U}_{w,q}^{-}\cap\mathbf{U}_{\mathcal{A}}^{-}=\bigoplus_{b\in\mathscr{B}_{w}\left(\infty\right)}\mathcal{A}G^{\mathrm{low}}\left(b\right)=:\mathbf{U}_{w,\mathcal{A}}^{-}$. Hence, 
\[
\mathbf{U}_{w,\mathcal{A}}^{-}\otimes_{\mathcal{A}}\mathbb{C}\xrightarrow{\sim}\mathbf{U}_{w}^{-}. 
\]
Set  $\mathbf{U}_{w}^{-}:=\sum_{a_{1},\cdots,a_{\ell}}\mathbb{C}f_{i_{1}}^{a_{1}}\cdots f_{i_{\ell}}^{a_{\ell}}\subset\mathbf{U}\left(\mathfrak{n}_{-}\right)$
for a reduced expression $\boldsymbol{i}\in I\left(w\right)$.
\end{rem}
For more details on Demazure modules and its crystal basis, see Kashiwara
\cite{MR1240605}. 
\begin{rem}
\label{r:Demazure} We have 
\[
\bigcup_{\lambda\in P_{+}}\overline{\jmath}_{\lambda}\left(\mathscr{B}_{w}\left(\lambda\right)\right)=\mathscr{B}_{w}\left(\infty\right).
\]
See also Theorem \ref{t:cryKP}. 
\end{rem}
\begin{defn}[{\cite[5.1.3]{MR2914878}}]
 \label{d:qclosed} Let $w\in W$. Set 
\[
\left(\mathbf{U}_{w,q}^{-}\right)^{\perp}:=\{x\in\Uq^{-}\mid(x,\mathbf{U}_{w,q}^{-})_{L}=0\}.
\]
Then, by $\Delta\left(\mathbf{U}_{w,q}^{-}\right)\subset\mathbf{U}_{w,q}^{-}\Uq^{0}\otimes\mathbf{U}_{w,q}^{-}$
and Lemma \ref{l:DLrel}, $\left(\mathbf{U}_{w,q}^{-}\right)^{\perp}$
is a two-sided ideal of $\Uq^{-}$. Hence we obtain a $\mathbb{Q}(q)$-algebra
\[
\Aq\left[N_{-}\cap X_{w}\right]:=\Uq^{-}/\left(\mathbf{U}_{w,q}^{-}\right)^{\perp},
\]
called the \emph{quantum closed unipotent cell}. The quantum closed
unipotent cell has a $Q_{-}$-graded algebra structure induced from
that of $\Uq^{-}$. Note that 
\[
\left(\mathbf{U}_{w,q}^{-}\right)^{\perp}=\bigoplus_{b\in\mathscr{B}\left(\infty\right)\setminus\mathscr{B}_{w}\left(\infty\right)}\mathbb{Q}\left(q\right)G^{\mathrm{up}}\left(b\right).
\]
Describe the canonical projection $\Uq^{-}\to\mathbf{A}_{q}\!\left[N_{-}\cap X_{w}\right]$
as $x\mapsto[x]$. The element $[x]$ clearly depends on $w$, however,
we omit to write $w$ because it will cause no confusion below.
\end{defn}
\begin{rem}
In \cite[5.1.3]{MR2914878}, $\mathbf{A}_{q}\!\left[N_{-}\cap X_{w}\right]$
is denoted by $\mathcal{O}_{q}\left[\overline{N_{w}}\right]$. 
\end{rem}
We set the $\mathcal{A}$-form $\mathbf{A}_{\mathbb{Q}\left[q^{\pm1}\right]}\!\left[N_{-}\cap X_{w}\right]$
of $\mathbf{A}_{q}\!\left[N_{-}\cap X_{w}\right]$ by 
\[
\mathbf{A}_{\mathbb{Q}\left[q^{\pm1}\right]}\!\left[N_{-}\cap X_{w}\right]:=\mathbf{A}_{\mathbb{Q}\left[q^{\pm1}\right]}\!\left[\boldsymbol{N}_{-}\right]/\left(\left(\mathbf{U}_{w,q}^{-}\right)^{\perp}\cap\mathbf{A}_{\mathbb{Q}\left[q^{\pm1}\right]}\!\left[\boldsymbol{N}_{-}\right]\right).
\]
Note that we have 
\begin{equation}
\left(\mathbf{U}_{w,q}^{-}\right)^{\perp}\cap\mathbf{A}_{\mathbb{Q}\left[q^{\pm1}\right]}\!\left[\boldsymbol{N}_{-}\right]=\bigoplus_{b\in\mathscr{B}\left(\infty\right)\setminus\mathscr{B}_{w}\left(\infty\right)}\mathcal{A}G^{\mathrm{up}}\left(b\right).\label{eq:closedcan}
\end{equation}
The following is the specialization result for quantum closed unipotent
cell which justifies the notation $\Aq\left[N_{-}\cap X_{w}\right]$.
\begin{thm}
For $w\in W$, we have 
\[
\mathbf{A}_{\mathbb{Q}\left[q^{\pm1}\right]}\!\left[N_{-}\cap X_{w}\right]\otimes_{\mathcal{A}}\mathbb{C}\simeq\mathbb{C}\left[N_{-}\cap X_{w}\right].
\]
here we regard $\mathbb{C}$ as an $\mathcal{A}$-module via $q^{\pm1}\mapsto1$. 
\end{thm}
\begin{proof}
We have an exact sequence of $\mathcal{A}$-modules

\[
0\to\left(\mathbf{U}_{w}^{-}\right)^{\perp}\cap\mathbf{A}_{\mathbb{Q}\left[q^{\pm1}\right]}\!\left[\boldsymbol{N}_{-}\right]\to\mathbf{A}_{\mathbb{Q}\left[q^{\pm1}\right]}\!\left[\boldsymbol{N}_{-}\right]\to\mathbf{A}_{\mathbb{Q}\left[q^{\pm1}\right]}\!\left[N_{-}\cap X_{w}\right]\to0,
\]
here the second map is the inclusion and the third map is the projection.
Moreover, $\left(\mathbf{U}_{w}^{-}\right)^{\perp}\cap\mathbf{A}_{\mathbb{Q}\left[q^{\pm1}\right]}\!\left[\boldsymbol{N}_{-}\right]$
and $\mathbf{A}_{\mathbb{Q}\left[q^{\pm1}\right]}\!\left[\boldsymbol{N}_{-}\right]$
are free $\mathcal{A}$-modules and an $\mathcal{A}$-basis of the
former can be chosen as the subset of that of the latter (see (\ref{eq:closedcan})).
Therefore $\mathbf{A}_{\mathbb{Q}\left[q^{\pm1}\right]}\!\left[N_{-}\cap X_{w}\right]$
is also a free $\mathcal{A}$-module (more precisely, $\mathbf{A}_{\mathbb{Q}\left[q^{\pm1}\right]}\!\left[N_{-}\cap X_{w}\right]$
admits the projected dual canonical basis), and we have
\begin{align*}
\mathbf{A}_{\mathbb{Q}\left[q^{\pm1}\right]}\!\left[N_{-}\cap X_{w}\right]\otimes_{\mathcal{A}}\mathbb{C} & \simeq \frac{\mathbf{A}_{\mathbb{Q}\left[q^{\pm1}\right]}\!\left[\boldsymbol{N}_{-}\right]\otimes_{\mathcal{A}}\mathbb{C}}{\left(\left(\mathbf{U}_{w}^{-}\right)^{\perp}\cap\mathbf{A}_{\mathbb{Q}\left[q^{\pm1}\right]}\!\left[\boldsymbol{N}_{-}\right]\right)\otimes_{\mathcal{A}}\mathbb{C}}\\
 & =\mathbf{U}\left(\mathfrak{n}_{-}\right)_{\mathrm{gr}}^{*}/\left(\mathbf{U}_{w}^{-}\right)^{\perp}\simeq\mathbb{C}\left[N_{-}\cap X_{w}\right],
\end{align*}
 here the last isomorphism follows from Proposition \ref{p:closedcoord}. 
\end{proof}

\subsection{Unipotent quantum matrix coefficients}
\begin{defn}\label{d:minor} For $\lambda\in P_{+}$ and $u,u'\in V(\lambda)$,
define the element $D_{u,u'}\in\Uq^{-}$ by 
\[
(D_{u,u'},x)_{L}=(u,x.u')_{\lambda}^{\varphi}
\]
for all $x\in\Uq^{-}$. We call an element of this form \emph{a unipotent
quantum matrix coefficient}. Note that $\wt\left(D_{u,u'}\right)=\wt u-\wt u'$ for weight vectors $u,u'\in V(\lambda)$. For $w,w'\in W$, we write
\[
D_{w\lambda,w'\lambda}:=D_{u_{w\lambda},u_{w'\lambda}},
\]
which is called \emph{a unipotent quantum minor}. See \cite[Section 6]{MR2914878}. 
\end{defn}
\begin{defn}
\label{d:picheck} Let $\lambda\in P_{+}$. Define a surjective $\mathbb{Q}\left(q\right)$-linear
map $\pi_{w\lambda}^{\vee}\colon\mathbf{U}_{q}^{-}\to V_{w}\left(\lambda\right)$
by 
\[
\pi_{w\lambda}^{\vee}\left(y\right)=\left(y\right)^{\vee}.u_{w\lambda}.
\]
\end{defn}
\begin{prop}[{\cite[Proposition 25.2.6]{Lus:intro},\cite[8.2.2 (iii), (iv)]{MR1262212}}]
\label{p:picheck} Let $\lambda\in P_{+}$ and $w\in W$. Then there
exists a surjective map $\pi_{w\lambda}^{\vee}\colon\mathscr{B}(\infty)\to\mathscr{B}_{w}(\lambda)\coprod\{0\}$
such that 
\[
\pi_{w\lambda}^{\vee}(\Glow(b))=\Glow_{\lambda}(\pi_{w\lambda}^{\vee}(b))
\]
for $b\in\mathscr{B}(\infty)$, here $\Glow_{\lambda}(0)=0$. Moreover,
$\pi_{w\lambda}^{\vee}$ induces a bijection $(\pi_{w\lambda}^{\vee})^{-1}(\mathscr{B}_{w}(\lambda))\to\mathscr{B}_{w}(\lambda)$. 
\end{prop}
\begin{defn}
\label{d:jcheck} Let $\lambda\in P_{+}$ and $w\in W$. Set $V_{w}\left(\lambda\right)^{\perp}:=\left\{ u\in V\left(\lambda\right)\mid\left(u,V_{w}\left(\lambda\right)\right)_{\lambda}^{\varphi}=0\right\} $.
Define $j_{w\lambda}^{\vee}\colon V\left(\lambda\right)/V_{w}\left(\lambda\right)^{\perp}\hookrightarrow\mathbf{U}_{q}^{-}$
as the dual homomorphism of $\pi_{w\lambda}^{\vee}$ given by the
non-degenerate bilinear forms $\left(\;,\;\right)_{\lambda}^{\varphi}\colon V\left(\lambda\right)\times V\left(\lambda\right)\to\mathbb{Q}\left(q\right)$
and $\left(\;,\;\right)_{L}\colon\mathbf{U}_{q}^{-}\times\mathbf{U}_{q}^{-}\to\mathbb{Q}\left(q\right)$,
that is, 
\begin{align*}
\left(j_{w\lambda}^{\vee}\left(u\right),y\right)_{L}=\left(u,\pi_{w\lambda}^{\vee}\left(y\right)\right)_{\lambda}^{\varphi}(=\left(u,y^{\vee}.u_{w\lambda}\right)_{\lambda}^{\varphi}=\left(\varphi\left(y^{\vee}\right).u,u_{w\lambda}\right)_{\lambda}^{\varphi}).
\end{align*}
\[
\xymatrix{V\left(\lambda\right)/\left(V_{w}\left(\lambda\right)\right)^{\perp}\ar[d]\ar[rr]^{j_{w\lambda}^{\vee}} &  & \mathbf{U}_{q}^{-}\ar[d]\\
\left(V_{w}\left(\lambda\right)\right)_{\mathrm{gr}}^{*}\ar[rr]^{\left(\pi_{w\lambda}^{\vee}\right)_{\mathrm{gr}}^{*}} &  & \left(\mathbf{U}_{q}^{-}\right)_{\mathrm{gr}}^{*}
}
\]

In the following, the map $V\left(\lambda\right)\to\mathbf{U}_{q}^{-}$
given by $u\mapsto j_{w\lambda}^{\vee}(p_{w}(u))$ is also denoted
by $j_{w\lambda}^{\vee}$, here $p_{w}$ denotes the canonical projection
$V\left(\lambda\right)\to V(\lambda)/V_{w}\left(\lambda\right)^{\perp}$. 
\end{defn}
The following proposition immediately follows from Proposition \ref{p:picheck}. 
\begin{prop}
\label{p:jcheck} Let $\lambda\in P_{+}$ and $w\in W$. Then there
is an injective map $\overline{\jmath}_{w\lambda}^{\vee}\colon\mathscr{B}_{w}\left(\lambda\right)\hookrightarrow\mathscr{B}\left(\infty\right)$
such that 
\[
\left(G_{\lambda}^{\mathrm{up}}\left(b\right),G^{\mathrm{low}}\left(b'\right)^{\vee}.u_{w\lambda}\right)_{\lambda}^{\varphi}=\delta_{b',\overline{\jmath}_{w\lambda}^{\vee}\left(b\right)}
\]
for any $b\in\mathscr{B}_{w}\left(\lambda\right)$ and $b'\in\mathscr{B}\left(\infty\right)$.
That is, we have $j_{w\lambda}^{\vee}\left(G_{\lambda}^{\mathrm{up}}\left(b\right)\right)=G^{\mathrm{up}}\left(\overline{\jmath}_{w\lambda}^{\vee}\left(b\right)\right)$.
\end{prop}
\begin{rem}
\label{r:jcheck} Let $\lambda\in P_{+}$ and $w\in W$. Then, 
\begin{itemize}
\item $\wt\overline{\jmath}_{w\lambda}^{\vee}\left(b\right)=-\wt b+w\lambda$
for $b\in\mathscr{B}_{w}(\lambda)$, and 
\item $\overline{\jmath}_{w\lambda}^{\vee}\left(\pi_{w\lambda}^{\vee}\left(b\right)\right)=b$
for $b\in(\pi_{w\lambda}^{\vee})^{-1}(\mathscr{B}_{w}(\lambda))$. 
\end{itemize}
\end{rem}
\begin{prop}
\label{p:minor} Let $\lambda\in P_{+}$ and $w\in W$. Then the following
hold: 
\begin{enumerate}
\item[(1)] $D_{G_{\lambda}^{\mathrm{up}}\left(b\right),u_{\lambda}}=G^{\mathrm{up}}\left(\overline{\jmath}_{\lambda}(b)\right)$
for all $b\in\mathscr{B}\left(\lambda\right)$, 
\item[(2)] $D_{u_{w\lambda},G_{\lambda}^{\mathrm{up}}\left(b\right)}=G^{\mathrm{up}}\left(\ast\overline{\jmath}_{w\lambda}^{\vee}\left(b\right)\right)$
for all $b\in\mathscr{B}_{w}\left(\lambda\right)$, and 
\item[(3)] $D_{u_{w\lambda},G_{\lambda}^{\mathrm{up}}\left(b\right)}=0$ for
all $b\in\mathscr{B}\left(\lambda\right)\setminus\mathscr{B}_{w}\left(\lambda\right)$. 
\end{enumerate}
\end{prop}
\begin{proof}
The equality (1) follows immediately by Proposition \ref{p:jlambda}.
For $y\in\Uq^{-}$, we have 
\begin{align*}
(D_{u_{w\lambda},G_{\lambda}^{\mathrm{up}}\left(b\right)},y)_{L} & =(u_{w\lambda},y.G_{\lambda}^{\mathrm{up}}\left(b\right))_{\lambda}^{\varphi}\\
 & =(G_{\lambda}^{\mathrm{up}}\left(b\right),(\ast(y))^{\vee}.u_{w\lambda})_{\lambda}^{\varphi}\\
 & =(G^{\mathrm{up}}\left(\overline{\jmath}_{w\lambda}^{\vee}\left(b\right)\right),\ast(y))_{L}\\
 & =(G^{\mathrm{up}}\left(\ast\overline{\jmath}_{w\lambda}^{\vee}\left(b\right)\right),y)_{L}.
\end{align*}
This completes the proof of (2). The assertion (3) follows from the
similar calculation and Proposition \ref{p:Demazure}. 
\end{proof}
\begin{prop}[{\cite[Corollary 6.4]{MR2914878}}]
\label{p:minorPBW} Let $w\in W$ and ${\bm{i}}=(i_{1},\dots,i_{\ell})\in I(w)$.
For $i\in I$, define ${\bm{n}}^{(i)}=(n_{1}^{(i)},\dots,n_{\ell}^{(i)})\in\mathbb{Z}_{\geq0}^{\ell}$
by 
\[
n_{k}^{(i)}=\begin{cases}
1 & \text{if}\;i_{k}=i,\\
0 & \text{otherwise.}
\end{cases}
\]
For $\lambda\in P_{+}$, set ${\bm{n}}^{\lambda}:=\sum_{i\in I}\langle\lambda,h_{i}\rangle{\bm{n}}^{(i)}$.
Then we have 
\[
D_{w\lambda,\lambda}=G^{\mathrm{up}}(b_{-1}({\bm{n}}^{\lambda},{\bm{i}})).
\]
\end{prop}

\subsection{Kumar-Peterson identities}

We investigate the map $\overline{\jmath}_{w\lambda}^{\vee}$ a little
more. Kumar and Peterson studied the identity which expresses the
$H$-characters of the coordinate ring $\mathbb{C}\left[X_{w}\cap \boldsymbol{U}_{v}\right]$
of the intersection $X_{w}\cap \boldsymbol{U}_{v}$ of Schubert varieties $X_{w}$
and $v$-translates of the open cell $\boldsymbol{U}_{v}$ as the limit of a family
of ``twisted'' characters of Demazure modules in general Kac-Moody
Lie algebras, see Kumar \cite[Theorem 12.1.3]{MR1923198}. In the
special case with $v=w$, it reduces to the case of Schubert cells,
that is, we have $X_{w}\cap \boldsymbol{U}_{w}=\mathring{X}_{w}$ (see Kumar \cite[Lemma 7.3.10]{MR1923198})
and the following equality can be considered as a crystalized Kumar-Peterson
identity.
\begin{thm}
\label{t:cryKP}We have 
\[
\bigcup_{\lambda\in P_{+}}\overline{\jmath}_{w\lambda}^{\vee}\left(\mathscr{B}_{w}\left(\lambda\right)\right)=\mathscr{B}\left(\mathbf{U}_{q}^{-}\left(w\right)\right).
\]
\end{thm}
The rest of this subsection is devoted to the proof of Theorem \ref{t:cryKP}. 
\begin{lem}[{{{\cite[Lemma 3.19]{Kimura:2016vn}}}}]
\label{l:compl}For $w\in W$, let $\mathbf{U}_{q}^{-}\left(w\right)^{\perp}$
be the orthogonal complement of $\Uq^{-}\left(w\right)$ with respect
to $\left(\ ,\ \right)_{L}$. We have an isomorphism as $\mathbb{Q}\left(q\right)$-vector
spaces: 
\[
\Uq^{-}\left(w\right)\otimes\left(\mathbf{U}_{q}^{-}\cap T_{w}\mathbf{U}_{q}^{-}\cap\mathrm{Ker}\left(\varepsilon\right)\right)\xrightarrow{\sim}\left(\mathbf{U}_{q}^{-}\left(w\right)\right)^{\bot}\subset\mathbf{U}_{q}^{-}
\]
under the multiplication $\mathbf{U}_{q}^{-}\left(w\right)\otimes\left(\mathbf{U}_{q}^{-}\cap T_{w}\mathbf{U}_{q}^{-}\right)\xrightarrow{\sim}\mathbf{U}_{q}^{-}$,
here recall that $\varepsilon$ is the counit of $\mathbf{U}_{q}$
(see Definition \ref{d:QEA}). 
\end{lem}
\begin{lem}
\label{l:orthog}For $y\in\mathbf{U}_{q}^{-}\left(w\right)^{\perp}$,
we have $y^{\vee}.u_{w\lambda}=0$ for all $\lambda\in P_{+}$. 
\end{lem}
\begin{proof}
By Lemma \ref{l:compl}, we write $y=\sum y_{\left(1\right)}y_{\left(2\right)}$
with $y_{\left(1\right)}\in\Uq^{-}\left(w\right)$ and homogeneous
elements $y_{\left(2\right)}\in\Uq^{-}\cap T_{w}\mathbf{U}_{q}^{-}\cap\Ker\left(\varepsilon\right)$.
Then we have 
\[
\left(y\right)^{\vee}.u_{w\lambda}=(T_{w^{-1}})^{-1}\left(\sum T_{w^{-1}}\left(y_{(1)}^{\vee}\right)T_{w^{-1}}\left(y_{(2)}^{\vee}\right).u_{\lambda}\right)=0
\]
because $\wt\left(T_{w^{-1}}\left(y_{(2)}^{\vee}\right)\right)\in Q_{+}\setminus\{0\}$.
\end{proof}
\begin{prop}
\label{p:incl} We have 
\[
\bigcup_{\lambda\in P_{+}}\overline{\jmath}_{w\lambda}^{\vee}\left(\mathscr{B}_{w}\left(\lambda\right)\right)\subset\mathscr{B}\left(\mathbf{U}_{q}^{-}\left(w\right)\right).
\]
\end{prop}
\begin{proof}
Let $\pi\left(w\right)\colon\mathbf{U}_{q}^{-}\to\mathbf{U}_{q}^{-}\left(w\right)$
be the projection with respect to the decomposition $\mathbf{U}_{q}^{-}=\mathbf{U}_{q}^{-}\left(w\right)\oplus\mathbf{U}_{q}^{-}\left(w\right)^{\perp}$. Since $\Uq^{-}\left(w\right)^{\perp}\cap\mathbf{B}^{\mathrm{low}}$
is a basis of $\Uq^{-}\left(w\right)^{\perp}$ by Proposition \ref{p:dualcanonical},
we have $\pi\left(w\right)\left(G^{\mathrm{low}}\left(b\right)\right)\neq0$
if and only if $b\in\mathscr{B}\left(\mathbf{U}_{q}^{-}\left(w\right)\right)$
for $b\in\mathscr{B}\left(\infty\right)$. Let $b\in\bigcup_{\lambda\in P_{+}}\overline{\jmath}_{w\lambda}^{\vee}\left(\mathscr{B}_{w}\left(\lambda\right)\right)$.
Then there exists $\lambda\in P_{+}$ such that $\left(G^{\mathrm{low}}\left(b\right)\right)^{\vee}.u_{w\lambda}\neq0$.
By Proposition \ref{l:orthog}, we have 
\begin{align*}
\left(G^{\mathrm{low}}\left(b\right)\right)^{\vee}.u_{w\lambda}=\left(\pi\left(w\right)\left(G^{\mathrm{low}}\left(b\right)\right)\right)^{\vee}.u_{w\lambda}.
\end{align*}
In particular, we have $\pi\left(w\right)\left(G^{\mathrm{low}}\left(b\right)\right)\neq0$.
This completes the proof. 
\end{proof}
We prove the opposite inclusion. 
\begin{prop}
\label{p:oppincl} We have 
\[
\mathscr{B}\left(\mathbf{U}_{q}^{-}\left(w\right)\right)\subset\bigcup_{\lambda\in P_{+}}\overline{\jmath}_{w\lambda}^{\vee}\left(\mathscr{B}_{w}\left(\lambda\right)\right).
\]
\end{prop}
\begin{proof}
Let $b\in\mathscr{B}\left(\mathbf{U}_{q}^{-}\left(w\right)\right)$,
that is $0\neq\pi\left(w\right)\left(G^{\mathrm{low}}\left(b\right)\right)\in\mathbf{U}_{q}^{-}\left(w\right)$.
(See the proof of Proposition \ref{p:incl}.) By Proposition \ref{p:picheck}
and Remark \ref{r:jcheck}, it suffices to show that $G^{\mathrm{low}}\left(b\right)^{\vee}.u_{w\lambda}=\left(\pi\left(w\right)\left(G^{\mathrm{low}}\left(b\right)\right)\right)^{\vee}.u_{w\lambda}\neq0$
for some $\lambda\in P_{+}$. Note that $\left(\pi\left(w\right)\left(G^{\mathrm{low}}\left(b\right)\right)\right)^{\vee}.u_{w\lambda}\neq0$
is equivalent to $\overline{\left(\pi\left(w\right)\left(G^{\mathrm{low}}\left(b\right)\right)\right)}^{\vee}.u_{w\lambda}\neq0$.

By the way, we have 
\[
\overline{y}^{\vee}.u_{w\lambda}=(T_{w^{-1}})^{-1}\left(\left(T_{w^{-1}}\circ\vee\circ\overline{\phantom{x}}\right)\left(y\right).u_{\lambda}\right)=(T_{w^{-1}})^{-1}\left(\left(\vee\circ\overline{\phantom{x}}\circ T_{w}^{-1}\right)\left(y\right).u_{\lambda}\right).
\]

Since $y_{0}:=\pi\left(w\right)\left(G^{\mathrm{low}}\left(b\right)\right)\in\mathbf{U}_{q}^{-}\cap T_{w}\mathbf{U}_{q}^{\geq0}$,
we have $\left(\vee\circ\overline{\phantom{x}}\circ T_{w}^{-1}\right)\left(y_{0}\right)\in\Uq^{\leq0}$.
It is well-known that, for $\xi\in Q_{-}$, there exists an element
$\lambda\in P_{+}$ such that the projection $\left(\mathbf{U}_{q}^{-}\right)_{\xi}\to V\left(\lambda\right)_{\xi+\lambda}$
given by $y\mapsto y.u_{\lambda}$ is an isomorphism of vector space.
Hence it can be shown that there exists $\lambda\in P_{+}$ such that
$\left(\vee\circ\overline{\phantom{x}}\circ T_{w}^{-1}\right)\left(y_{0}\right).u_{\lambda}\neq0$.
\end{proof}

\section{Quantum unipotent cells and the De Concini-Procesi isomorphisms\label{sec:Quantum-unipotent-cell}}

In this section, we introduce quantum unipotent cells $\mathbf{A}_{q}\!\left[N_{-}^{w}\right]$
following De Concini-Procesi \cite{MR1635678}, and show that they
are isomorphic to the quantum coordinate ring of $N_{-}\left(w\right)\cap wG_{0}^{\min}$
. This isomorphism, called the De Concini-Procesi isomorphism, was
proved in \cite[Theorem 3.2]{MR1635678} under the assumption that
$\mathfrak{g}$ is of finite type. We will prove it in the case of
arbitrary symmetrizable Kac-Moody cases (Theorem \ref{thm:DeCP}).
We also introduce the dual canonical bases of the quantum unipotent
cells (Definition \ref{d:localdualcan}).

\subsection{Quantum unipotent cells}

To define the quantum unipotent cells, we use the localizations of
$\mathbf{A}_{q}\!\left[N_{-}\left(w\right)\right]$ and $\mathbf{A}_{q}\!\left[N_{-}\cap X_{w}\right]$. We recall the Ore properties of the unipotent quantum minors. The following is the multiplicative property of the dual canonical
bases with respect to the unipotent quantum minors. 
\begin{prop}[{\cite[Theorem 6.24, Theorem 6.25]{MR2914878}}]
\label{p:localization} Let $w\in W$. 
\begin{enumerate}
\item[(1)] For $\lambda\in P_{+}$ and $b\in\mathscr{B}_{w}\left(\infty\right)$,
there exists $b'\in\mathscr{B}_{w}\left(\infty\right)$ such that
\[
q^{-(\lambda,\wt b)}[D_{w\lambda,\lambda}][G^{\mathrm{up}}(b)]=[G^{\mathrm{up}}(b')].
\]
\item[(2)] For $\lambda\in P_{+}$, $\bm{i}\in I(w)$ and ${\bm{c}}\in\mathbb{Z}_{\geq0}^{\ell(w)}$,
we have 
\[
q^{-(\lambda,\wt b_{-1}({\bm{c}},{\bm{i}}))}D_{w\lambda,\lambda}G^{\mathrm{up}}(b_{-1}({\bm{c}},{\bm{i}}))=G^{\mathrm{up}}(b_{-1}({\bm{c}}+{\bm{n}}^{\lambda},{\bm{i}})),
\]
where ${\bm{n}}^{\lambda}$ is defined as in Proposition \ref{p:minorPBW}.
\end{enumerate}
\end{prop}
Proposition \ref{p:localization} together with Proposition \ref{p:dualbar} deduces the following (cf.~Remark \ref{r:commproof} below). 
\begin{prop}
\label{p:comm} Let $w\in W$ and set $\mathcal{D}_{w}:=\{q^{m}D_{w\lambda,\lambda}\mid m\in\mathbb{Z},\lambda\in P_{+}\}$.
Then the sets $\mathcal{D}_{w}$ and $[\mathcal{D}_{w}]$ are Ore
sets of $\mathbf{A}_{q}\!\left[N_{-}\left(w\right)\right]$ and $\mathbf{A}_{q}\!\left[N_{-}\cap X_{w}\right]$
respectively consisting of $q$-central elements. More explicitly,
for $\lambda,\lambda'\in P_{+}$ and homogeneous elements $x\in\mathbf{A}_{q}\!\left[N_{-}\left(w\right)\right]$,
$y\in\mathbf{A}_{q}\!\left[N_{-}\cap X_{w}\right]$, we have 
\begin{align*}
q^{-(\lambda,w\lambda'-\lambda')}D_{w\lambda,\lambda}D_{w\lambda',\lambda'} & =D_{w(\lambda+\lambda'),\lambda+\lambda'}\\
D_{w\lambda,\lambda}x & =q^{(\lambda+w\lambda,\wt x)}xD_{w\lambda,\lambda}\;\text{in}\;\mathbf{A}_{q}\!\left[N_{-}\left(w\right)\right],\;\text{and}\\{}
[D_{w\lambda,\lambda}][y] & =q^{(\lambda+w\lambda,\wt y)}[y][D_{w\lambda,\lambda}]\;\text{in}\;\mathbf{A}_{q}\!\left[N_{-}\cap X_{w}\right].
\end{align*}
\end{prop}

Using the Proposition \ref{p:comm}, we obtain the definition of quantum
unipotent cells.
\begin{defn}
\label{d:localization}For $w\in W$, we set 
\begin{align*}
\mathbf{A}_{q}\!\left[N_{-}(w)\cap wG_{0}^{\min}\right] & :=\mathbf{A}_{q}\!\left[N_{-}\left(w\right)\right]\left[\mathcal{D}_{w}^{-1}\right],\\
\mathbf{A}_{q}\!\left[N_{-}^{w}\right] & :=\mathbf{A}_{q}\!\left[N_{-}\cap X_{w}\right]\left[\left[\mathcal{D}_{w}\right]^{-1}\right].
\end{align*}
Those algebras have $Q$-graded algebra structures in an obvious way.
The algebra $\mathbf{A}_{q}\!\left[N_{-}^{w}\right]$ is called a
quantum unipotent cell. 
\end{defn}
\begin{rem}
We note that the notations $\mathbf{A}_{q}\!\left[N_{-}(w)\cap wG_{0}^{\min}\right]$
and $\mathbf{A}_{q}\!\left[N_{-}^{w}\right]$ will be justified after
proving the existence of the dual canonical bases of that.
\end{rem}

\subsection{Dual canonical bases of quantum unipotent cells}

In this subsection, we define the dual canonical bases of quantum
unipotent cells using localization and the ``multiplicative property''
of the dual canonical bases of $\Aq[N_{-}^{w}]$ and $\Aq[N_{-}(w)\cap wG_{0}^{\min}]$.
\begin{prop}
\label{p:localdualcan}Let $w\in W$ and $\bm{i}\in I(w)$. Then the
following hold: 

\textup{(1)} The subset 
\[
\{q^{(\lambda,\wt b+\lambda-w\lambda)}[D_{w\lambda,\lambda}]^{-1}[G^{\mathrm{up}}(b)]\mid\lambda\in P_{+},b\in\mathscr{B}_{w}(\infty)\}
\]
of $\Aq[N_{-}^{w}]$ forms a $\mathbb{Q}(q)$-basis of $\Aq[N_{-}^{w}]$. 

\textup{(2)} The subset 
\[
\{q^{(\lambda,\wt b_{-1}({\bm{c}},{\bm{i}})+\lambda-w\lambda)}D_{w\lambda,\lambda}^{-1}G^{\mathrm{up}}(b_{-1}({\bm{c}},{\bm{i}}))\mid\lambda\in P_{+},\bm{c}\in\mathbb{Z}_{\geq0}^{\ell(w)}\}
\]
of $\Aq[N_{-}(w)\cap wG_{0}^{\min}]$ forms a $\mathbb{Q}(q)$-basis
of $\Aq[N_{-}(w)\cap wG_{0}^{\min}]$. 
\end{prop}
\begin{proof}
We prove only (1). The assertion (2) is proved in the same manner.
The given subset obviously spans the $\mathbb{Q}(q)$-vector space
$\Aq[N_{-}^{w}]$. Hence it remains to show that this set is a linearly
independent set. For $(\lambda,b),(\lambda',b')\in P_{+}\times\mathscr{B}_{w}(\infty)$,
write $(\lambda,b)\sim(\lambda',b')$ if and only if $q^{(\lambda,\wt b+\lambda-w\lambda)}[D_{w\lambda,\lambda}]^{-1}[G^{\mathrm{up}}(b)]=q^{(\lambda',\wt b'+\lambda'-w\lambda')}[D_{w\lambda',\lambda'}]^{-1}[G^{\mathrm{up}}(b')]$.
The relation $\sim$ is clearly an equivalence relation, and we take
a complete set $F$ of coset representatives of $(P_{+}\times\mathscr{B}_{w}(\infty))/\sim$.

Suppose that there exists a finite subset $F'\subset F$ and $a_{\lambda,b}\in\mathbb{Q}(q)$
$((\lambda,b)\in F')$ such that $\sum_{(\lambda,b)\in F'}q^{(\lambda,\wt b+\lambda-w\lambda)}a_{\lambda,b}[D_{w\lambda,\lambda}]^{-1}[G^{\mathrm{up}}(b)]=0$.
There exists $\lambda_{0}\in P_{+}$ such that $\lambda_{0}-\lambda\in P_{+}$
for all $\lambda\in P_{+}$ such that $(\lambda,b)\in F'$ for some
$b\in\mathscr{B}_{w}(\infty)$. Now the equality $\sum_{(\lambda,b)\in F'}q^{(\lambda,\wt b+\lambda-w\lambda)}a_{\lambda,b}[D_{w\lambda,\lambda}]^{-1}[G^{\mathrm{up}}(b)]=0$
is equivalent to the equality 
\begin{align}
[D_{w\lambda_{0},\lambda_{0}}]\left(\sum_{(\lambda,b)\in F'}q^{(\lambda,\wt b+\lambda-w\lambda)}a_{\lambda,b}[D_{w\lambda,\lambda}]^{-1}[G^{\mathrm{up}}(b)]\right)=0.\label{erase}
\end{align}
By Proposition \ref{p:comm} and Proposition \ref{p:localization},
for $(\lambda,b)\in F'$, we have 
\begin{align*}
[D_{w\lambda_{0},\lambda_{0}}]\left(q^{(\lambda,\wt b+\lambda-w\lambda)}[D_{w\lambda,\lambda}]^{-1}[G^{\mathrm{up}}(b)]\right) & =q^{-(\lambda_{0}-\lambda,w\lambda-\lambda)+(\lambda,\wt b+\lambda-w\lambda)}[D_{w(\lambda_{0}-\lambda),(\lambda_{0}-\lambda)}][G^{\mathrm{up}}(b)]\\
 & =q^{(\lambda_{0},\wt b+\lambda-w\lambda)}[G^{\mathrm{up}}(b^{(\lambda_{0}-\lambda)})]
\end{align*}
for some $b^{(\lambda_{0}-\lambda)}\in\mathscr{B}_{w}(\infty)$. Note
that $\wt b+\lambda-w\lambda=\wt b^{(\lambda_{0}-\lambda)}-\wt D_{w\lambda_{0},\lambda_{0}}$.
Therefore if $b^{(\lambda_{0}-\lambda)}=(b')^{(\lambda_{0}-\lambda')}$
for $(\lambda,b),(\lambda',b')\in F'$ then we have the equality 
\[
[D_{w\lambda_{0},\lambda_{0}}]\left(q^{(\lambda,\wt b+\lambda-w\lambda)}[D_{w\lambda,\lambda}]^{-1}[G^{\mathrm{up}}(b)]\right)=[D_{w\lambda_{0},\lambda_{0}}]\left(q^{(\lambda',\wt b'+\lambda'-w\lambda')}[D_{w\lambda',\lambda'}]^{-1}[G^{\mathrm{up}}(b')]\right),
\]
hence $(\lambda,b)=(\lambda',b')$. Thus (\ref{erase}) implies $a_{\lambda,b}=0$
for all $(\lambda,b)\in F'$. This completes the proof. 
\end{proof}
\begin{defn}
\label{d:localdualcan} Let $w\in W$. We call 
\begin{align*}
\widetilde{\mathbf{B}}^{\mathrm{up},w} & :=\{q^{(\lambda,\wt b+\lambda-w\lambda)}[D_{w\lambda,\lambda}]^{-1}[G^{\mathrm{up}}(b)]\mid\lambda\in P_{+},b\in\mathscr{B}_{w}(\infty)\},\;\text{and}\\
\widetilde{\mathbf{B}}^{\mathrm{up}}(w) & :=\{q^{(\lambda,\wt b_{-1}({\bm{c}},{\bm{i}})+\lambda-w\lambda)}D_{w\lambda,\lambda}^{-1}G^{\mathrm{up}}(b_{-1}({\bm{c}},{\bm{i}}))\mid\lambda\in P_{+},\bm{c}\in\mathbb{Z}_{\geq0}^{\ell(w)}\}
\end{align*}
the dual canonical bases of $\Aq[N_{-}^{w}]$ and $\Aq[N_{-}(w)\cap wG_{0}^{\min}]$,
respectively.

For $\lambda\in P$, there exist $\lambda_{1},\lambda_{2}\in P_{+}$
such that $\lambda=-\lambda_{1}+\lambda_{2}$. Set 
\[
D_{w,\lambda}:=q^{(\lambda_{1},w\lambda-\lambda)}D_{w\lambda_{1},\lambda_{1}}^{-1}D_{w\lambda_{2},\lambda_{2}}\in\widetilde{\mathbf{B}}^{\mathrm{up}}(w).
\]
Then $D_{w,\lambda}$ does not depend on the choice of $\lambda_{1},\lambda_{2}\in P_{+}$
by Proposition \ref{p:localdualcan}. Note that $\wt D_{w,\lambda}=w\lambda-\lambda$. 
\end{defn}
The following is straightforwardly proved by Proposition \ref{p:comm}. 
\begin{prop}
\label{p:Elambda} Let $w\in W$ and $\lambda,\lambda'\in P_{+}$.
Then the following hold: 
\begin{enumerate}
\item[(1)] $D_{w,\lambda}=q^{(\lambda,w\lambda_{1}-\lambda_{1})}D_{w\lambda_{2},\lambda_{2}}D_{w\lambda_{1},\lambda_{1}}^{-1}$
for $\lambda_{1},\lambda_{2}\in P_{+}$ with $\lambda=-\lambda_{1}+\lambda_{2}$.
\item[(2)] $D_{w,\lambda}D_{w,\lambda'}=q^{(\lambda,w\lambda'-\lambda')}D_{w,\lambda+\lambda'}$.
In particular, $D_{w,\lambda}^{-1}=q^{(\lambda,w\lambda-\lambda)}D_{w,-\lambda}$.
\item[(3)]  $D_{w,\lambda}x=q^{(\lambda+w\lambda,\wt x)}xD_{w,\lambda}$ for
$\lambda\in P_{+}$and a homogeneous element $x\in\Aq[N_{-}(w)\cap wG_{0}^{\min}]$. 
\end{enumerate}
\end{prop}
\begin{rem}
\label{r:param} By using Proposition \ref{p:localization} (2), we
can parametrize explicitly the elements of $\widetilde{\mathbf{B}}^{\mathrm{up}}(w)$.
Fix ${\bm{i}}=(i_{1},\dots,i_{\ell})\in I(w)$. An element ${\bm{c}}\in\mathbb{Z}_{\geq0}^{\ell}$
is said to have gaps if $\min\{c_{k}\mid i_{k}=i\}=0$ for all $i\in I$.
Then, by Propositions \ref{p:localization} (2) and \ref{p:localdualcan}
(2), we obtain the non-overlapping parametrization of the elements
of $\widetilde{\mathbf{B}}^{\mathrm{up}}(w)$ as follows: 
\[
\widetilde{\mathbf{B}}^{\mathrm{up}}(w)=\{q^{-(\lambda,\wt b_{-1}({\bm{c}},{\bm{i}}))}D_{w,\lambda}G^{\mathrm{up}}(b_{-1}({\bm{c}},{\bm{i}}))\mid\lambda\in P,{\bm{c}}\in\mathbb{Z}_{\geq0}^{\ell}\;\text{has gaps}\}.
\]
\end{rem}
We define the dual bar involutions on $\Aq[N_{-}^{w}]$
and $\Aq[N_{-}(w)\cap wG_{0}^{\min}]$, which are useful when we study
the dual canonical bases. 
\begin{prop}
\label{p:dualbarex} The following hold: 

\textup{(1)} The twisted dual bar involution $\sigma'$ induces $\mathbb{Q}$-algebra
anti-involutions $\Aq[N_{-}\cap X_{w}]\to\Aq[N_{-}\cap X_{w}]$ and
$\Aq[N_{-}(w)]\to\Aq[N_{-}(w)]$. See Definition \ref{d:twdualbar}
for the definition of $\sigma'$. Moreover these maps are extended
to $\mathbb{Q}$-algebra anti-involutions $\sigma'\colon\Aq\left[N_{-}^{w}\right]\to\Aq\left[N_{-}^{w}\right]$
and $\sigma'\colon\Aq[N_{-}(w)\cap wG_{0}^{\min}]\to\Aq[N_{-}(w)\cap wG_{0}^{\min}]$. 

\textup{(2)} Define a $\mathbb{Q}(q)$-linear isomorphism $c_{\mathrm{tw}}\colon\Aq[N_{-}^{w}]\to\Aq[N_{-}^{w}]$
(resp.~$\Aq[N_{-}(w)\cap wG_{0}^{\min}]\to\Aq[N_{-}(w)\cap wG_{0}^{\min}]$)
by 
\[
x\mapsto q^{\left(\wt x,\wt x\right)/2-\left(\wt x,\rho\right)}x
\]
for every homogeneous element $x\in\Aq[N_{-}^{w}]$ (resp.~$x\in\Aq[N_{-}(w)\cap wG_{0}^{\min}]$).

Set $\sigma:=c_{\mathrm{tw}}\circ\sigma'$. Then for homogeneous elements
$x,y\in\Aq[N_{-}^{w}]$ (resp.~$\Aq[N_{-}(w)\cap wG_{0}^{\min}]$)
we have 
\begin{equation}
\sigma(xy)=q^{(\wt x,\wt y)}\sigma(y)\sigma(x).\label{twmulti}
\end{equation}
Moreover the elements of the dual canonical bases $\widetilde{\mathbf{B}}^{\mathrm{up},w}$
and $\widetilde{\mathbf{B}}^{\mathrm{up}}(w)$ are fixed by $\sigma$. 
\end{prop}
\begin{defn}
\label{d:dualbarex} The $\mathbb{Q}$-linear isomorphisms $\sigma$
and $\sigma'\colon\Aq[N_{-}^{w}]\to\Aq[N_{-}^{w}],\Aq[N_{-}(w)\cap wG_{0}^{\min}]\to\Aq[N_{-}(w)\cap wG_{0}^{\min}]$
defined in Proposition \ref{p:dualbarex} will be also called \emph{the dual bar involution} and \emph{the twisted dual bar involution}, respectively. 
\end{defn}
\begin{proof}[{{Proof of Proposition \ref{p:dualbarex}}}]
 Recall that $\sigma'(\Gup(b))=q^{-\left(\wt b,\wt b\right)/2+\left(\wt b,\rho\right)}\Gup(b)$
for all $b\in\mathscr{B}(\infty)$. See Remark \ref{r:twdualbar}.
Hence (1) follows from the compatibility of the algebras $\Aq[N_{-}\cap X_{w}]$,
$\Aq[N_{-}(w)]$ and the dual canonical basis (Definition \ref{d:PBWparam},
Definition \ref{d:qclosed}), and the universality of localization
\cite[Proposition 6.3]{MR1020298}. A direct calculation immediately
shows the equality \ref{twmulti}. For $\lambda\in P_{+}$, we have
\begin{align*}
1 & =\sigma(D_{w\lambda,\lambda}D_{w\lambda,\lambda}^{-1})\\
 & =q^{-(w\lambda-\lambda,w\lambda-\lambda)}\sigma(D_{w\lambda,\lambda}^{-1})\sigma(D_{w\lambda,\lambda})\\
 & =q^{2(\lambda,w\lambda-\lambda)}\sigma(D_{w\lambda,\lambda}^{-1})D_{w\lambda,\lambda}
\end{align*}
in $\Aq[N_{-}(w)\cap wG_{0}^{\min}]$. Hence 
\[
\sigma(D_{w\lambda,\lambda}^{-1})=q^{-2(\lambda,w\lambda-\lambda)}D_{w\lambda,\lambda}^{-1}.
\]
Let $b\in\mathscr{B}_{w}(\infty)$. Then, by Proposition \ref{p:comm}
and the equality above, we have 
\begin{align*}
 & \sigma(q^{(\lambda,\wt b+\lambda-w\lambda)}[D_{w\lambda,\lambda}]^{-1}[G^{\mathrm{up}}(b)])\\
 & =q^{-(\lambda,\wt b+\lambda-w\lambda)+(\lambda-w\lambda,\wt b)}\sigma([G^{\mathrm{up}}(b)])\sigma([D_{w\lambda,\lambda}]^{-1})\\
 & =q^{-(\lambda,\wt b+\lambda-w\lambda)+(\lambda-w\lambda,\wt b)-2(\lambda,w\lambda-\lambda)}[G^{\mathrm{up}}(b)][D_{w\lambda,\lambda}]^{-1}\\
 & =q^{-(\lambda,\wt b+\lambda-w\lambda)+(\lambda-w\lambda,\wt b)-2(\lambda,w\lambda-\lambda)+(\lambda+w\lambda,\wt b)}[D_{w\lambda,\lambda}]^{-1}[G^{\mathrm{up}}(b)]\\
 & =q^{(\lambda,\wt b+\lambda-w\lambda)}[D_{w\lambda,\lambda}]^{-1}[G^{\mathrm{up}}(b)].
\end{align*}
This proves the dual bar invariance property for $\widetilde{\mathbf{B}}^{\mathrm{up},w}$.
The assertion for $\tilde{\mathbf{B}}^{\mathrm{up}}(w)$ is proved
in the same manner.
\end{proof}
As a corollary of the existence of the dual canonical bases of $\mathbf{A}_{q}\!\left[N_{-}\left(w\right)\cap wG_{0}^{\min}\right]$
and $\mathbf{A}_{q}\!\left[N_{-}^{w}\right]$, we have the following
specialization theorem.
\begin{cor}
\label{c:specializationlocal}Let $w\in W$. 

\textup{(1)} Set $\mathbf{A}_{\mathbb{Q}\left[q^{\pm1}\right]}\!\left[N_{-}\left(w\right)\cap wG_{0}^{\min}\right]$
to be the free $\mathcal{A}$-module spanned by $\widetilde{\mathbf{B}}^{\mathrm{up}}\left(w\right)$.
Then it is a $\mathcal{A}$-subalgebra of $\mathbf{A}_{q}\!\left[N_{-}\left(w\right)\cap wG_{0}^{\min}\right]$
and we have an isomorphism 
\[
\mathbf{A}_{\mathbb{Q}\left[q^{\pm1}\right]}\!\left[N_{-}\left(w\right)\cap wG_{0}^{\min}\right]\otimes_{\mathcal{A}}\mathbb{C}\simeq\mathbb{C}\left[N_{-}\left(w\right)\cap wG_{0}^{\min}\right]
\]
 as $\mathbb{C}$-algebras.

\textup{(2)} Set $\mathbf{A}_{\mathbb{Q}\left[q^{\pm1}\right]}\!\left[N_{-}^{w}\right]$
to be the free $\mathcal{A}$-module spanned by $\widetilde{\mathbf{B}}^{\mathrm{up},w}$.
Then it is a $\mathcal{A}$-subalgebra of $\mathbf{A}_{q}\!\left[N_{-}^{w}\right]$
and we have an isomorphism 
\[
\mathbf{A}_{\mathbb{Q}\left[q^{\pm1}\right]}\!\left[N_{-}^{w}\right]\otimes_{\mathcal{A}}\mathbb{C}\simeq\mathbb{C}\left[N_{-}^{w}\right]
\]
as $\mathbb{C}$-algebras.
\end{cor}

\subsection{De Concini-Procesi isomorphisms}\label{ss:DeCP}

In this subsection, we give a proof of the De Concini-Procesi isomorphism between $\mathbf{A}_{q}\!\left[N_{-}\left(w\right)\right]$ and $\mathbf{A}_{q}\!\left[N_{-}\cap X_{w}\right]$ for general symmetrizable Kac-Moody Lie algebras, by using theory of canonical bases and specialization. We should remark that the original proof in \cite{MR1635678} uses the downward induction on the length of elements of the Weyl group $W$ from the longest element, which exists only in finite type cases. 

\begin{prop}[{{{\cite[Theorem 5.13]{MR2914878}}}}]
\label{p:usualinj} Let $w\in W$. Define $\iota_{w}\colon\Aq[N_{-}(w)]\to\Aq[N_{-}\cap X_{w}]$
as a $\mathbb{Q}(q)$-algebra homomorphism induced from the canonical
projection $\Uq^{-}\to\Aq[N_{-}\cap X_{w}]$. Recall Definition \ref{d:qunip}
and \ref{d:qclosed}. Then $\iota_{w}$ is injective, or equivalently,
$\ast(\mathscr{B}(\Uq^{-}(w)))\subset\mathscr{B}_{w}(\infty)$.
\end{prop}
\begin{thm}[{The De Concini-Procesi isomorphism}]
\label{thm:DeCP} Let $w\in W$. Then $\iota_{w}$ induces an isomorphism;
\[
\iota_{w}\colon\Aq[N_{-}(w)\cap wG_{0}^{\min}]\xrightarrow{\sim}\Aq[N_{-}^{w}].
\]
\end{thm}
\begin{proof}
The map $\iota_{w}$ in Proposition \ref{p:usualinj} induces an injective
algebra homomorphism $\iota_{w}\colon\Aq[N_{-}(w)]\to\Aq[N_{-}^{w}]$.
Since this map sends $D_{w\lambda,\lambda}$ to $[D_{w\lambda,\lambda}]$
for $\lambda\in P_{+}$, it is extended to the injective algebra homomorphism
\begin{equation}
\iota_{w}\colon\Aq[N_{-}(w)\cap wG_{0}^{\min}]=\Aq[N_{-}(w)][\mathcal{D}_{w}^{-1}]\to\Aq[N_{-}^{w}]\label{iotaex}
\end{equation}
by the universality of localization. It follows immediately from the
definition of dual canonical bases and Proposition \ref{p:usualinj}
that $\iota_{w}$ induces an injective map from $\widetilde{\mathbf{B}}^{\mathrm{up}}(w)$
to $\widetilde{\mathbf{B}}^{\mathrm{up},w}$. Therefore the map (\ref{iotaex})
is an isomorphism if and only if the (well-defined) map 
\begin{equation}
\iota_{w}\mid_{q=1}\colon\mathbf{A}_{\mathbb{Q}[q^{\pm1}]}[N_{-}(w)\cap wG_{0}^{\min}]\otimes_{\mathcal{A}}\mathbb{C}\to\mathbf{A}_{\mathbb{Q}[q^{\pm1}]}[N_{-}^{w}]\otimes_{\mathcal{A}}\mathbb{C}\label{itotacl}
\end{equation}
is an isomorphism. Through the isomorphisms in Corollary \ref{c:specializationlocal},
the map $\iota_{w}\mid_{q=1}$ coincides with the map in Corollary
\ref{c:spusualisom} by definition of $\iota_{w}$; hence it is an
isomorphism. This completes the proof.
\end{proof}

\section{Quantum twist isomorphisms\label{sec:Quantum-twist-isomorphisms}}

In this section, we construct the quantum twist isomorphisms between
$\mathbf{A}_{q}\!\left[N_{-}\left(w\right)\cap wG_{0}^{\min}\right]$ and
$\mathbf{A}_{q}\!\left[N_{-}^{w}\right]$ (see Theorem \ref{t:BZisom})
and define the quantum twist automorphisms on $\mathbf{A}_{q}\!\left[N_{-}^{w}\right]$
as a composite of the quantum twist isomorphism and the De Concini-Procesi
isomorphism.

\subsection{Quantized coordinate algebras}

In this subsection, we give a brief review on the quantized coordinate
rings. For more details, see \cite[Chapter 9, 10]{MR1315966}.
\begin{defn}
\label{d:mat_coeff} Let $M$ be a $\Uq$-module. For $f\in M^{\ast}:=\mathrm{Hom}_{\mathbb{\mathbb{Q}}(q)}\left(M,\mathbb{\mathbb{Q}}(q)\right)$
and $u\in M$, define a $\mathbb{Q}\left(q\right)$-linear map $c_{f,u}^{M}\in\Uq^{\ast}$
given by 
\[
x\mapsto\langle f,x.u\rangle
\]
for $x\in\Uq$. When $M=V\left(\lambda\right)$ ($\lambda\in P_{+}$),
we abbreviate $c_{f,u}^{V(\lambda)}$ to $c_{f,u}^{\lambda}$. For
$w,w'\in W$ and $\lambda\in P_{+}$, we write 
\[
c_{w\lambda,w'\lambda}^{\lambda}:=c_{f_{w\lambda},u_{w'\lambda}}^{\lambda},
\]
here $f_{w\lambda}\in V(\lambda)^{\ast}$ is defined by $u\mapsto(u_{w\lambda},u)_{\lambda}^{\varphi}$. 
\end{defn}
\begin{defn}
\label{d;grdual} Let $M$ be a $\Uq$-module. For $\mu\in P$, we
set 
\[
M_{\mu}:=\{m\in M\mid q^{h}.m=q^{\langle\mu,h\rangle}m\;\text{for all}\;h\in P^{*}\}.
\]
For a $\Uq$-module $M=\bigoplus_{\mu\in P}M_{\mu}$ with weight space
decomposition, we write its graded dual $\bigoplus_{\mu\in P}M_{\mu}^{\ast}$
as $M^{\star}$. Note that $M^{\star}$ is a right $\Uq$-module.
For $\lambda\in P_{+}$, $V(\lambda)^{\star}$ is an integrable highest
weight right $\Uq$-module with highest weight $\lambda$. For $u\in V(\lambda)$,
define $u^{\ast}\in V(\lambda)^{\star}$ by $u'\mapsto(u,u')_{\lambda}^{\varphi}$.
Then we have $V(\lambda)^{\star}=\{u^{\ast}\mid u\in V(\lambda)\}$
since the bilinear form $(\ ,\ )_{\lambda}^{\varphi}$ is non-degenerate.
Note that $f_{w\lambda}=u_{w\lambda}^{\ast}$ for $w\in W$.
\end{defn}
Let $\Rq$ be the $\mathbb{Q}(q)$-vector subspace of $\Uq^{\ast}$
spanned by the elements 
\[
\left\{ c_{f,u}^{\lambda}\mid f\in V(\lambda)^{\star},u\in V(\lambda)\;\text{and}\;\lambda\in P_{+}\right\} .
\]
Henceforth, we consider the algebra structure of $\Uq^{\ast}$ induced
from the coalgebra structure of $\Uq$. 
\begin{prop}[{\cite[Definition 7.2.1, Proposition 7.2.2]{MR1203234}}]
 The subspace $\Rq$ is a subalgebra of $\Uq^{\ast}$, which is isomorphic
to $\bigoplus_{\lambda\in P_{+}}V(\lambda)^{\star}\otimes V(\lambda)$
as a $\Uq$-bimodule. 
\end{prop}
The $\mathbb{Q}(q)$-algebra $\Rq$ is called \emph{the quantized coordinate algebra} associated with $\Uq$. 
\begin{defn}
\label{d:Qhomog} Let $v,w\in W$ and $\lambda\in P_{+}$. Set 
\begin{align*}
 & \Rq^{w(+)}(\lambda):=\{c_{f,u_{w\lambda}}^{\lambda}\mid f\in V(\lambda)^{\star}\},\;\Rq^{w(+)}:=\sum_{\lambda'\in P_{+}}\Rq^{w(+)}(\lambda')\subset\Rq,\\
 & \Qq_{v}^{w(+)}(\lambda):=\{c_{f,u_{w\lambda}}^{\lambda}\mid f\in V(\lambda)^{\star},\;\langle f,\Uq^{+}.u_{v\lambda}\rangle=0\},\;\Qq_{v}^{w(+)}:=\sum_{\lambda'\in P_{+}}\Qq_{v}^{w(+)}(\lambda')\subset\Rq.
\end{align*}
When $w=e$, we write $\Rq^{e(+)}$ (resp.~$\Qq_{v}^{e(+)}$) as
$\Rq^{+}$ (resp.~$\Qq_{u}^{+}$). It is easy to show that, for all
$w\in W$, $\Rq^{w(+)}$ is a subalgebra of $\Rq$, and isomorphic
to $\Rq^{+}$ as algebras via $c_{f,u_{w\lambda}}^{\lambda}\mapsto c_{f,u_{\lambda}}^{\lambda}$.
See, for example, \cite[Chapter 3]{Tani:QFA}. Moreover, for all $v,w\in W$,
$\Qq_{v}^{w(+)}$ is a two-sided ideal of $\Rq^{w(+)}$, and the above
isomorphism induces an isomorphism from $\Rq^{w(+)}/\Qq_{v}^{w(+)}$
to $\Rq^{+}/\Qq_{v}^{+}$. 
\end{defn}

\subsection{Other descriptions of quantum unipotent subgroups and quantum closed unipotent
cells}

In this subsection, we describe the algebras, quantum unipotent subgroups
and quantum unipotent cells, by using the quantized coordinate algebra
$\Rq$. The following descriptions are essentially shown in \cite[9.1.7]{MR1315966},
\cite[Theorem 3.7]{MR2679698}. However, we restate them emphasizing
the terms of dual canonical bases. Actually, we can now prove each
statement immediately. 
\begin{notation} Let $v,w\in W$. By abuse of notation, we describe
the canonical projection $\Rq^{w(+)}\to\Rq^{w(+)}/\Qq_{v}^{w(+)}$
as $c\mapsto[c]$. \end{notation} 
\begin{defn}
\label{d:formhom} As a bridge between quantized enveloping algebras
and quantized coordinate algebras, we consider the following two linear
maps: 
\begin{align*}
\Phi\colon\cUq^{\leq0}\to(\Uq^{\leq0})^{\ast}, & \;y_{1}\mapsto\left(y_{2}\mapsto\left(\psi(y_{1}),y_{2}\right)_{D}\right),\\
\Phi^{+}\colon\cUq^{\geq0}\to(\Uq^{\geq0})^{\ast}, & \;x_{1}\mapsto\left(x_{2}\mapsto\left(x_{1},\psi(x_{2})\right)_{D}\right).
\end{align*}
By the properties of the Drinfeld pairing $\left(\ ,\ \right)_{D}$,
$\Phi$ is an injective algebra homomorphism and $\Phi^{+}$ is an
injective algebra anti-homomorphism. 
\end{defn}
\begin{defn}
Let $\lambda\in P_{+}$. Set 
\[
\Uq^{-}(\lambda):=j_{\lambda}\left(V(\lambda)\right)=\sum\nolimits _{b\in\mathscr{B}\left(\lambda\right)}\mathbb{Q}(q)G^{\mathrm{up}}\left(\overline{\jmath}_{\lambda}(b)\right).
\]
\end{defn}
The following propositions follow from the non-degeneracy of the Drinfeld
pairing, Lemma \ref{l:DLrel} and Proposition \ref{p:minor}. 
\begin{prop}
\label{p:isom} The restriction map $\Uq^{\ast}\to(\Uq^{\leq0})^{\ast}$
induces the injective algebra homomorphism $r_{\leq0}\colon\Rq^{+}\to(\Uq^{\leq0})^{\ast}$,
and $\Image r_{\leq0}\subset\Image\Phi$. Moreover the well-defined
$\mathbb{Q}(q)$-algebra homomorphism $\Rq^{+}\to\cUq^{\leq0},\;c\mapsto(\Phi^{-1}\circ r_{\leq0})(c)$
induces the $\mathbb{Q}(q)$-algebra isomorphism $\mathcal{I}\colon\Rq^{+}\to\sum_{\lambda\in P_{+}}\Uq^{-}(\lambda)q^{-\lambda}$. 
\end{prop}
\begin{prop}
\label{p:idealisom} For $\lambda\in P_{+}$ and $b\in\mathscr{B}\left(\lambda\right)$,
we have 
\[
\mathcal{I}\left(c_{G_{\lambda}^{\mathrm{up}}\left(b\right)^{\ast},u_{\lambda}}^{\lambda}\right)=G^{\mathrm{up}}\left(\overline{\jmath}_{\lambda}(b)\right)q^{-\lambda}=D_{G_{\lambda}^{\mathrm{up}}\left(b\right),u_{\lambda}}q^{-\lambda}.
\]
In particular, we have 
\[
\mathcal{I}\left(\Qq_{w}^{+}(\lambda)\right)=\sum\nolimits _{b\in\mathscr{B}\left(\lambda\right)\setminus\mathscr{B}_{w}\left(\lambda\right)}\mathbb{Q}(q)G^{\mathrm{up}}\left(\overline{\jmath}_{\lambda}(b)\right)q^{-\lambda}.
\]
\end{prop}
\begin{defn}
An element $z$ of $\Rq^{+}$ (resp.~$\Rq^{+}/\Qq_{w}^{+}$) is said
to be \emph{$q$-central} if, for every weight vector $f\in V(\lambda)^{\star}$
and $\lambda\in P_{+}$, there exists $l\in\mathbb{Z}$ such that
\[
zc_{f,u_{\lambda}}^{\lambda}=q^{l}c_{f,u_{\lambda}}^{\lambda}z\hspace{10pt}(\text{resp.}\;z[c_{f,u_{\lambda}}^{\lambda}]=q^{l}[c_{f,u_{\lambda}}^{\lambda}]z).
\]
\end{defn}
\begin{cor}
\label{c:qcent} The set $\{c_{\lambda,\lambda}^{\lambda}\}_{\lambda\in P_{+}}$
is an Ore set in $\Rq^{+}$ consisting of $q$-central elements. In
particular, $\mathcal{S}:=\{[c_{\lambda,\lambda}^{\lambda}]\}_{\lambda\in P_{+}}$
is an Ore set in $\Rq^{+}/\Qq_{w}^{+}$ consisting of $q$-central
elements. 
\end{cor}
By Corollary \ref{c:qcent}, we can consider the algebra $(\Rq^{+}/\Qq_{w}^{+})[\mathcal{S}^{-1}]$.
Proposition \ref{p:isom} and \ref{p:idealisom} together with Remark
\ref{r:Demazure} immediately imply the following proposition. This
gives the description of $\Aq[N_{-}\cap X_{w}]$ in terms of the quantized
coordinate algebra $\Rq$. This kind of description appears in \cite[9.1.7]{MR1315966}. 
\begin{prop}
\label{p:description} Let $w\in W$. Set $\Aq[N_{-}\cap X_{w}]^{\mathrm{ex}}:=\cUq^{\leq0}/(\mathbf{U}_{w}^{-})^{\perp}\cUq^{0}$.
Note that $(\mathbf{U}_{w}^{-})^{\perp}\cUq^{0}$ is a two-sided ideal
of $\cUq^{\leq0}$. Then the $\mathbb{Q}(q)$-algebra isomorphism
$\mathcal{I}\colon\Rq^{+}\to\sum_{\lambda\in P_{+}}\Uq^{-}(\lambda)q^{-\lambda}$
induces the $\mathbb{Q}(q)$-algebra isomorphism 
\[
\mathcal{I}_{w}\colon\left(\Rq^{+}/\Qq_{w}^{+}\right)[\mathcal{S}^{-1}]\to\Aq[N_{-}\cap X_{w}]^{\mathrm{ex}}.
\]
Moreover the $\mathbb{Q}(q)$-algebra $\sum_{\lambda\in P_{+}}\left(\Rq^{+}(\lambda)/\Qq_{w}^{+}\right)[c_{\lambda,\lambda}^{\lambda}]^{-1}(\subset\left(\Rq^{+}/\Qq_{w}^{+}\right)[\mathcal{S}^{-1}])$
is isomorphic to $\Aq[N_{-}\cap X_{w}]$.
\end{prop}
Next, we study the quantum unipotent subgroups via the quantized coordinate
rings following Joseph and Yakimov. We consider the algebra $\Rq^{w(+)}/\Qq_{w}^{w(+)}$, which is
isomorphic to $\Rq^{+}/\Qq_{w}^{+}$. See Definition \ref{d:Qhomog}.
\begin{defn}
Let $w\in W$ and $\lambda\in P_{+}$. Set 
\[
\Uq^{+}(w,\lambda):=\left(j_{w\lambda}^{\vee}\left(V(\lambda)/V_{w}(\lambda)^{\perp}\right)\right)^{\vee}=\sum\nolimits _{b\in\mathscr{B}_{w}\left(\lambda\right)}\mathbb{Q}(q)G^{\mathrm{up}}\left(\overline{\jmath}_{w\lambda}^{\vee}\left(b\right)\right)^{\vee}.
\]
\end{defn}
The following proposition follows again from the non-degeneracy of
the Drinfeld pairing, the equality (\ref{formrel}), Lemma \ref{l:DLrel},
Proposition \ref{p:Kasinv} and Proposition \ref{p:minor}. 
\begin{prop}
\label{p:isom^+} Let $w\in W$. The restriction map $\Uq^{\ast}\to(\Uq^{\geq0})^{\ast}$
induces the algebra homomorphism $r_{\geq0}^{w}\colon\Rq^{w(+)}\to(\Uq^{\geq0})^{\ast}$,
and it satisfies $\Ker(r_{\geq0}^{w})=\Qq_{w}^{w(+)}$ and $\Image r_{\geq0}^{w}\subset\Image\Phi^{+}$.
Hence $r_{\geq0}^{w}$ induces the $\mathbb{Q}(q)$-algebra isomorphism
$\overline{r}_{\geq0}^{w}\colon\Rq^{w(+)}/\Qq_{w}^{w(+)}\to\Image r_{\geq0}^{w}$.
Moreover we have a well-defined algebra anti-isomorphism $\mathcal{I}_{w}^{+}\colon\Rq^{+}/\Qq_{w}^{+}\to\sum_{\lambda\in P_{+}}\Uq^{+}(w,\lambda)q^{-w\lambda}$
given by $[c_{f,u_{\lambda}}^{\lambda}]\mapsto\left((\Phi^{+})^{-1}\circ\overline{r}_{\geq0}^{w}\right)\left([c_{f,u_{w\lambda}}^{\lambda}]\right)$
for $f\in V(\lambda)^{\star}$, $\lambda\in P_{+}$. We have 
\[
\mathcal{I}_{w}^{+}\left([c_{G_{\lambda}^{\mathrm{up}}\left(b\right)^{\ast},u_{\lambda}}^{\lambda}]\right)=G^{\mathrm{up}}\left(\overline{\jmath}_{w\lambda}^{\vee}\left(b\right)\right)^{\vee}q^{-w\lambda}=\varphi(D_{u_{w\lambda},G_{\lambda}^{\mathrm{up}}\left(b\right)})q^{-w\lambda}
\]
for $b\in\mathscr{B}_{w}(\lambda)$. 
\end{prop}
\begin{cor}
\label{c:qcent^+} The set $\mathcal{S}_{w}:=\{[c_{w\lambda,\lambda}^{\lambda}]\}_{\lambda\in P_{+}}$
is an Ore set in $\Rq^{+}/\Qq_{w}^{+}$ consisting of $q$-central
elements. 
\end{cor}
\begin{rem}
\label{r:zerodiv} The description in Proposition \ref{p:isom^+}
implies that the algebra $\Rq^{+}/\Qq_{w}^{+}$ has no zero divisors.
\end{rem}
By Corollary \ref{c:qcent^+}, we can consider the $\mathbb{Q}(q)$-algebra
$(\Rq^{+}/\Qq_{w}^{+})[\mathcal{S}_{w}^{-1}]$. Proposition \ref{p:isom^+}
immediately implies the following proposition. This gives the description
of $\Aq[N_{-}(w)]$ in terms of the quantized coordinate algebra $\Rq$.
This description appears in \cite[Theorem 3.7]{MR2679698} modulo some difference of conventions. 
\begin{prop}
\label{p:description^+} Let $w\in W$. Then $\mathcal{I}_{w}^{+}$
induces the algebra anti-isomorphism 
\[
\mathcal{I}_{w}^{+}\colon\left(\Rq^{+}/\Qq_{w}^{+}\right)[\mathcal{S}_{w}^{-1}]\to\Uq^{+}(w)\cUq^{0}.
\]
Moreover the $\mathbb{Q}(q)$-algebra $\sum_{\lambda\in P_{+}}\left(\Rq^{+}(\lambda)/\Qq_{w}^{+}\right)[c_{w\lambda,\lambda}^{\lambda}]^{-1}(\subset\left(\Rq^{+}/\Qq_{w}^{+}\right)[\mathcal{S}_{w}^{-1}])$
is anti-isomorphic to $\Uq^{+}(w)$, and is isomorphic to $\Aq[N_{-}(w)]$
via $\varphi$. 
\end{prop}
\begin{proof}
It suffices to show that $\sum_{\lambda\in P_{+}}\Uq^{+}(w,\lambda)=\Uq^{+}(w)$.
This follows from Theorem \ref{t:cryKP}. 
\end{proof}

\subsection{Quantum twist isomorphisms and dual canonical bases}
In this subsection, we prove the existence of quantum twist isomorphisms (Theorem \ref{t:BZisom}). 
The following lemma easily follows from Corollary
\ref{c:qcent} and \ref{c:qcent^+}. See also \cite[Proposition 6.3]{MR1020298}. 
\begin{lem}
\label{l:local} Let $w\in W$. Then the set $\tilde{\mathcal{S}}_{w}:=\{q^m[c_{w\lambda,\lambda}^{\lambda}c_{\lambda',\lambda'}^{\lambda'}]\mid m\in \mathbb{Z}, \lambda,\lambda'\in P_{+}\}$
is an Ore set in $\Rq^{+}/\Qq_{w}^{+}$ consisting of $q$-central
elements.

Moreover the maps $(\Rq^{+}/\Qq_{w}^{+})[\mathcal{S}^{-1}]\to(\Rq^{+}/\Qq_{w}^{+})[\tilde{\mathcal{S}}_{w}^{-1}]$,
$[c_{f,u_{\lambda}}^{\lambda}][c_{\lambda',\lambda'}^{\lambda'}]^{-1}\mapsto[c_{f,u_{\lambda}}^{\lambda}][c_{\lambda',\lambda'}^{\lambda'}]^{-1}$
and $(\Rq^{+}/\Qq_{w}^{+})[\mathcal{S}_{w}^{-1}]\to(\Rq^{+}/\Qq_{w}^{+})[\tilde{\mathcal{S}}_{w}^{-1}]$,
$[c_{f,u_{\lambda}}^{\lambda}][c_{w\lambda',\lambda'}^{\lambda'}]^{-1}\mapsto[c_{f,u_{\lambda}}^{\lambda}][c_{w\lambda',\lambda'}^{\lambda'}]^{-1}$
are injective $\mathbb{Q}(q)$-algebra homomorphisms. 
\end{lem}
\begin{thm}
\label{t:BZisom} There exists an isomorphism of the $\mathbb{Q}(q)$-algebras
\[
\gamma_{w,q}\colon\Aq[N_{-}^{w}]\to\Aq[N_{-}(w)\cap wG_{0}^{\min}]
\]
given by 
\begin{align*}
[D_{u,u_{\lambda}}] & \mapsto q^{-(\lambda,\wt u-\lambda)}D_{w\lambda,\lambda}^{-1}D_{u_{w\lambda},u} & [D_{w\lambda,\lambda}]^{-1}\mapsto & q^{(\lambda,w\lambda-\lambda)}D_{w\lambda,\lambda}
\end{align*}
for a weight vector $u\in V(\lambda)$ and $\lambda\in P_{+}$. 
\end{thm}
\begin{defn}
We call $\gamma_{w,q}$ \emph{a quantum twist isomorphism} (cf.~Proposition \ref{prop:clstwist}). 
\end{defn}
\begin{proof}[{{Proof of Theorem \ref{t:BZisom}}}]
By Proposition \ref{p:description} (see also Proposition \ref{p:idealisom}),
we have the algebra isomorphism 
\begin{align}
\Aq[N_{-}\cap X_{w}]\xrightarrow{\mathcal{I}_{w}^{-1}}\sum_{\lambda\in P_{+}}\left(\Rq^{+}(\lambda)/\Qq_{w}^{+}\right)[c_{\lambda,\lambda}^{\lambda}]^{-1}\label{isom}
\end{align}
given by 
\begin{align}
[D_{u,u_{\lambda}}]\mapsto[c_{u^{\ast},u_{\lambda}}^{\lambda}][c_{\lambda,\lambda}^{\lambda}]^{-1}\label{Dimage}
\end{align}
for $\lambda\in P_{+}$ and $u\in V(\lambda)$. In particular, $\mathcal{I}_{w}^{-1}([D_{w\lambda,\lambda}])=[c_{w\lambda,\lambda}^{\lambda}][c_{\lambda,\lambda}^{\lambda}]^{-1}$.

By Lemma \ref{l:local}, $\sum_{\lambda\in P_{+}}\left(\Rq^{+}(\lambda)/\Qq_{w}^{+}\right)[c_{\lambda,\lambda}^{\lambda}]^{-1}$
is naturally regarded as a subalgebra of $(\Rq^{+}/\Qq_{w}^{+})[\tilde{\mathcal{S}}_{w}^{-1}]$,
and in the latter algebra, the set $\{q^m[c_{w\lambda,\lambda}^{\lambda}][c_{\lambda,\lambda}^{\lambda}]^{-1}\mid m\in \mathbb{Z}, \lambda\in P_{+}\}$
is a multiplicative set consisting of invertible $q$-central elements.
Hence the algebra isomorphism (\ref{isom}) is extended to the algebra
isomorphism 
\begin{align}
\mathcal{J}_{1}\colon\Aq[N_{-}^{w}]\to\sum_{\substack{\lambda,\lambda',\lambda''\in P_{+}\\
\lambda=\lambda'+\lambda''
}
}\left(\Rq^{+}(\lambda)/\Qq_{w}^{+}\right)[c_{\lambda',\lambda'}^{\lambda'}c_{w\lambda'',\lambda''}^{\lambda''}]^{-1}.\label{isomext}
\end{align}
On the other hand, by Proposition \ref{p:description^+} (see also
Proposition \ref{p:isom^+}), we have an algebra isomorphism 
\begin{align}
\sum_{\lambda\in P_{+}}\left(\Rq^{+}(\lambda)/\Qq_{w}^{+}\right)[c_{w\lambda,\lambda}^{\lambda}]^{-1}\xrightarrow{\varphi\circ\mathcal{I}_{w}^{+}}\Aq[N_{-}(w)],\label{isom^+}
\end{align}
given by 
\begin{align}
[c_{w\lambda,\lambda}^{\lambda}]^{-1}[c_{u^{\ast},u_{\lambda}}^{\lambda}]\mapsto D_{u_{w\lambda},u}\label{Dimage^+}
\end{align}
for $\lambda\in P_{+}$ and $u\in V(\lambda)$. In particular, $(\varphi\circ\mathcal{I}_{w}^{+})([c_{w\lambda,\lambda}^{\lambda}]^{-1}[c_{\lambda,\lambda}^{\lambda}])=D_{w\lambda,\lambda}$.

As above, the set $\{q^m[c_{w\lambda,\lambda}^{\lambda}]^{-1}[c_{\lambda,\lambda}^{\lambda}]\mid m\in \mathbb{Z}, \lambda\in P_{+}\}$
is a multiplicative set consisting of invertible $q$-central elements
of $(\Aq^{+}/\Qq_{w}^{+})[\tilde{\mathcal{S}}_{w}^{-1}]$. Hence the
algebra isomorphism (\ref{isom^+}) is extended to the algebra isomorphism
\begin{align}
\mathcal{J}_{2}\colon\sum_{\substack{\lambda,\lambda',\lambda''\in P_{+}\\
\lambda=\lambda'+\lambda''
}
}\left(\Rq^{+}(\lambda)/\Qq_{w}^{+}\right)[c_{\lambda',\lambda'}^{\lambda'}c_{w\lambda'',\lambda''}^{\lambda''}]^{-1}\to\Aq[N_{-}(w)\cap wG_{0}^{\min}].\label{isomext^+}
\end{align}

By (\ref{isomext}) and (\ref{isomext^+}), we obtain the $\mathbb{Q}(q)$-algebra
isomorphism 
\[
\gamma_{w,q}:=\mathcal{J}_{2}\circ\mathcal{J}_{1}\colon\Aq[N_{-}^{w}]\to\Aq[N_{-}(w)\cap wG_{0}^{\min}].
\]
Moreover, for $\lambda\in P_{+}$ and a weight vector $u\in V(\lambda)$,
we have 
\begin{align*}
\gamma_{w,q}(D_{u,u_{\lambda}}) & =\mathcal{J}_{2}([c_{u^{\ast},u_{\lambda}}^{\lambda}][c_{\lambda,\lambda}^{\lambda}]^{-1})\;\text{by}\;\eqref{Dimage},\\
 & =\mathcal{J}_{2}(q^{-(\lambda,\wt u-\lambda)}[c_{\lambda,\lambda}^{\lambda}]^{-1}[c_{u^{\ast},u_{\lambda}}^{\lambda}])\;\text{by Proposition \ref{p:idealisom}},\\
 & =\mathcal{J}_{2}(q^{-(\lambda,\wt u-\lambda)}[c_{\lambda,\lambda}^{\lambda}]^{-1}[c_{w\lambda,\lambda}^{\lambda}][c_{w\lambda,\lambda}^{\lambda}]^{-1}[c_{u^{\ast},u_{\lambda}}^{\lambda}])\\
 & =q^{-(\lambda,\wt u-\lambda)}D_{w\lambda,\lambda}^{-1}D_{u_{w\lambda,u}}\;\text{by}\;\eqref{Dimage^+}.
\end{align*}
Moreover, 
\begin{align*}
1 & =\gamma_{w,q}([D_{w\lambda,\lambda}][D_{w\lambda,\lambda}]^{-1})\\
 & =q^{-(\lambda,w\lambda-\lambda)}D_{w\lambda,\lambda}^{-1}\gamma_{w,q}([D_{w\lambda,\lambda}]^{-1}).
\end{align*}
Hence, 
\[
\gamma_{w,q}([D_{w\lambda,\lambda}]^{-1})=q^{(\lambda,w\lambda-\lambda)}D_{w\lambda,\lambda}.
\]
This completes the proof of the theorem. 
\end{proof}
\begin{rem}\label{r:commproof}
We can also deduce Proposition \ref{p:comm} from the descriptions 
\begin{align*}
[D_{w\lambda,\lambda}]&=\mathcal{I}_{w}([c_{w\lambda,\lambda}^{\lambda}][c_{\lambda,\lambda}^{\lambda}]^{-1})& 
D_{w\lambda,\lambda}&=(\varphi\circ\mathcal{I}_{w}^{+})([c_{w\lambda,\lambda}^{\lambda}]^{-1}[c_{\lambda,\lambda}^{\lambda}])
\end{align*}
appearing in the proof of Theorem \ref{t:BZisom} together with Propositions \ref{p:idealisom} and \ref{p:isom^+}. 
\end{rem}
The quantum twist isomorphism $\gamma_{w,q}$ is compatible with the
dual canonical bases as follows: 
\begin{thm}
\label{t:preserve} Let $w\in W$. Then the quantum twist isomorphism
$\gamma_{w,q}\colon\Aq[N_{-}^{w}]\to\Aq[N_{-}(w)\cap wG_{0}^{\min}]$
restricts to the bijection $\widetilde{\mathbf{B}}^{\mathrm{up},w}\to\widetilde{\mathbf{B}}^{\mathrm{up}}(w)$
given by 
\[
q^{(\lambda,\wt(\overline{\jmath}_{\lambda'}(b))+\lambda-w\lambda)}[D_{w\lambda,\lambda}]^{-1}[G^{\mathrm{up}}\left(\overline{\jmath}_{\lambda'}(b)\right)]\mapsto q^{-(\lambda-\lambda',\wt(\ast\overline{\jmath}_{w\lambda'}^{\vee}\left(b\right)))}D_{w,\lambda-\lambda'}G^{\mathrm{up}}\left(\ast\overline{\jmath}_{w\lambda'}^{\vee}\left(b\right)\right)
\]
for $\lambda,\lambda'\in P_{+},b\in\mathscr{B}_{w}\left(\lambda'\right)$.
In particular, $\gamma_{w,q}([D_{w,\lambda}])=D_{w,-\lambda}$ for
$\lambda\in P$, and $\gamma_{w,q}\circ\sigma=\sigma\circ\gamma_{w,q}$.
\end{thm}
\begin{proof}
By Proposition \ref{p:minor}, for $\lambda,\lambda'\in P_{+}$ and
$b\in\mathscr{B}_{w}(\lambda')$, we have 
\begin{align*}
 & \gamma_{w,q}(q^{(\lambda,\wt(\overline{\jmath}_{\lambda'}(b))+\lambda-w\lambda)}[D_{w\lambda,\lambda}]^{-1}[G^{\mathrm{up}}\left(\overline{\jmath}_{\lambda'}(b)\right)])\\
 & =\gamma_{w,q}(q^{(\lambda,\wt b-\lambda'+\lambda-w\lambda)}[D_{w\lambda,\lambda}]^{-1}[D_{G_{\lambda'}^{\mathrm{up}}(b),u_{\lambda'}}])\\
 & =q^{(\lambda,\wt b-\lambda'+\lambda-w\lambda)}(q^{(\lambda,w\lambda-\lambda)}D_{w\lambda,\lambda})(q^{-(\lambda',\wt b-\lambda')}D_{w\lambda',\lambda'}^{-1}D_{u_{w\lambda'},G_{\lambda'}^{\mathrm{up}}(b)})\\
 & =q^{-(\lambda-\lambda',\wt(\ast\overline{\jmath}_{w\lambda'}^{\vee}\left(b\right)))}D_{w,\lambda-\lambda'}G^{\mathrm{up}}\left(\ast\overline{\jmath}_{w\lambda'}^{\vee}\left(b\right)\right).
\end{align*}
This completes the proof.
\end{proof}

\section{Twist automorphisms on quantum unipotent cells}\label{sec:Twist-aut}
We now obtain the twist automorphisms on quantum unipotent cells. 
\begin{thm}
\label{t:BZauto} Let $w\in W$. Then there exists a $\mathbb{Q}(q)$-algebra
automorphism 
\[
\eta_{w,q}:=\iota_{w}\circ\gamma_{w,q}\colon\Aq[N_{-}^{w}]\to\Aq[N_{-}^{w}]
\]
given by 

\begin{align*}
[D_{u,u_{\lambda}}]\mapsto & q^{-(\lambda,\wt u-\lambda)}[D_{w\lambda,\lambda}]^{-1}[D_{u_{w\lambda},u}] & [D_{w\lambda,\lambda}]^{-1}\mapsto & q^{(\lambda,w\lambda-\lambda)}[D_{w\lambda,\lambda}]
\end{align*}
for a weight vector $u\in V(\lambda)$ and $\lambda\in P_{+}$. In particular, $\wt\eta_{w,q}([x])=-\wt[x]$ for homogeneous elements
$[x]\in\Aq[N_{-}^{w}]$. Moreover $\eta_{w,q}$ restricts to a permutation
on the dual canonical bases $\widetilde{\mathbf{B}}^{\mathrm{up},w}$.
In particular, $\eta_{w,q}$ commutes with the dual bar involution
$\sigma$, and $\eta_{w,q}([D_{w,\lambda}])=[D_{w,-\lambda}]$ for
$\lambda\in P_{+}$. 
\end{thm}
The following follows from the theorem above and Proposition \ref{prop:clstwist}. 
\begin{cor}
\label{c:spBZauto} Let $w\in W$. Then the $\mathbb{Q}(q)$-algebra
automorphism $\eta_{w,q}\colon\Aq[N_{-}^{w}]\to\Aq[N_{-}^{w}]$ induces
a $\mathcal{A}$-algebra automorphism $\eta_{w,\mathcal{A}}\colon\mathbf{A}_{\mathcal{\mathbb{Q}}[q^{\pm1}]}[N_{-}^{w}]\to\mathbf{A}_{\mathbb{Q}[q^{\pm1}]}[N_{-}^{w}]$
and a $\mathbb{C}$-algebra automorphism 
\[
\eta_{w,q}\mid_{q=1}\colon\mathbf{A}_{\mathbb{Q}[q^{\pm1}]}[N_{-}^{w}]\otimes_{\mathcal{A}}\mathbb{C}\to\mathbf{A}_{\mathbb{Q}[q^{\pm1}]}[N_{-}^{w}]\otimes_{\mathcal{A}}\mathbb{C}.
\]
Moreover, through the isomorphism in Corollary \ref{c:specializationlocal},
the automorphism $\eta_{w,q}\mid_{q=1}$ coincides with $\eta_{w}^{\ast}$. 
\end{cor}
\begin{defn}
Let $w\in W$. Then we call the $\mathbb{Q}(q)$-algebra automorphism
$\eta_{w,q}\colon\Aq[N_{-}^{w}]\to\Aq[N_{-}^{w}]$ \emph{a twist automorphism} on the quantum unipotent cell $N_{-}^{w}$. 
\end{defn}
\begin{rem}
\label{r:minlambda} In order to apply quantum twist automorphisms
to a dual canonical basis element $[\Gup(\tilde{b})]$, $\tilde{b}\in\mathscr{B}(\infty)$,
we have to find $\lambda\in P_{+}$ and $b\in\mathscr{B}(\lambda)$
such that $\Gup(\tilde{b})=D_{\Gup_{\lambda}(b),u_{\lambda}}=\Gup(\overline{\jmath}_{\lambda}(b))$.
By Proposition \ref{p:jlambdaimage}, we can take $\lambda$ as $\lambda_{\tilde{b}}:=\sum_{i\in I}\varepsilon_{i}^{\ast}(\tilde{b})\varpi_{i}$.
Note that $\lambda_{\tilde{b}}$ is ``minimal'' in an appropriate
sense. 
\end{rem}
\section{Quantum twist automorphisms and quantum cluster algebras}\label{sec:GLS}

\label{s:GLS} In this section, we consider an additive categorification
of the twist automorphism $\eta_{w,q}$ on a quantum unipotent cell
$\Aq[N_{-}^{w}]$ in the sense of Gei\ss-Leclerc-Schr\"oer. When $\mathfrak{g}$
is symmetric, Gei\ss-Leclerc-Schr\"oer \cite{MR2833478} obtained a categorification
of the twist automorphism $\eta_{w}^{\ast}$ on the coordinate algebra
of a unipotent cell $N_{-}^{w}$ (Proposition \ref{p:GLStwist}).
They used subcategories $\mathcal{C}_{w}$, introduced by Buan-Iyama-Reiten-Scott
\cite{MR2521253}, of the module category of the preprojective algebra
$\Pi$ corresponding to the Dynkin diagram for $\mathfrak{g}$. Gei\ss-Leclerc-Schr\"oer
\cite{MR3090232} have also shown that the quantum unipotent subgroup
$\Aq[N_{-}(w)]$ is isomorphic to a certain quantum cluster algebra
$\mathscr{A}_{\mathbb{Q}(q)}(\mathcal{C}_{w})$, which is determined
by data of $\mathcal{C}_{w}$ (Proposition \ref{p:GLSqclus}). Combining
these results, we obtain a categorification of the twist automorphism
$\eta_{w,q}$ (Theorem \ref{t:qGLS}). See also Corollary \ref{c:clustermonom}.

In this section, we always consider the case that $\mathfrak{g}$
is symmetric. We assume that $(\alpha_{i},\alpha_{i})=2$ for all
$i\in I$, and thus $q_{i}=q$ for all $i\in I$.
\begin{notation}
\label{n:index} For $m,m'\in\mathbb{Z}_{\geq0}$ with $m\leq m'$,
set $[m,m']:=\{k\in\mathbb{Z}\mid m\leq k\leq m'\}$. 
\end{notation}

\subsection{Quantum cluster algebras}

In this subsection, we briefly review quantum cluster algebras. The
main references are \cite{MR2146350} and \cite{MR3090232}.
\begin{defn}
\label{d:qcluster} Let $n,\ell$ be positive integers such that $n\leq\ell$.
Let $\Lambda=(\lambda_{ij})_{i,j\in[1,\ell]}$ be a skew-symmetric
integer matrix. The skew-symmetric integer matrix $\Lambda$ determines
a skew-symmetric $\mathbb{Z}$-bilinear form $\mathbb{Z}^{\ell}\times\mathbb{Z}^{\ell}\to\mathbb{Z}$
by $\Lambda(\bm{e}_{i},\bm{e}_{j})=\lambda_{ij}$ for $i,j\in[1,\ell]$,
denoted also by $\Lambda$. Here $\{\bm{e}_{i}\mid i\in[1,\ell]\}$
denotes the standard basis of $\mathbb{Z}^{\ell}$. \emph{The based
quantum torus $\mathcal{T}(=\mathcal{T}(\Lambda))$ associated with
$\Lambda$ }is the $\mathbb{Q}[q^{\pm1/2}]$-algebra defined as follows:
as a $\mathbb{Q}[q^{\pm1/2}]$-module $\mathcal{T}$ is free and has
a $\mathbb{Q}[q^{\pm1/2}]$-basis $\{X^{\bm{a}}\mid\bm{a}\in\mathbb{Z}^{\ell}\}$.
The multiplication is defined by 
\[
X^{\bm{a}}X^{\bm{b}}=q^{\Lambda(a,b)/2}X^{\bm{a}+\bm{b}}
\]
for $\bm{a},\bm{b}\in\mathbb{Z}^{\ell}$. Then 
\begin{itemize}
\item $\mathcal{T}$ is an associative algebra, 
\item $X^{\bm{a}}X^{\bm{b}}=q^{\Lambda(\bm{a},\bm{b})}X^{\bm{b}}X^{\bm{a}}$
for $\bm{a},\bm{b}\in\mathbb{Z}^{\ell}$, 
\item $X^{\bm{0}}=1$ and $(X^{\bm{a}})^{-1}=X^{-\bm{a}}$ for $\bm{a}\in\mathbb{Z}^{\ell}$. 
\end{itemize}
The based quantum torus $\mathcal{T}$ is contained in its skew-field
of fractions $\mathcal{F}(=\mathcal{F}(\Lambda))$ \cite[Appendix A]{MR2146350}.
Note that $\mathcal{F}$ is a $\mathbb{Q}(q^{1/2})$-algebra. Write
$X_{i}:=X^{\bm{e}_{i}}$ for $i\in[1,\ell]$. 

Next we define an important operation, called \emph{mutation}. Let
$\widetilde{B}=(b_{ij})_{i\in[1,\ell],j\in[1,\ell-n]}$ be an $\ell\times(\ell-n)$
integer matrix. The submatrix $B=(b_{ij})_{i,j\in[1,\ell-n]}$ of
$\widetilde{B}$ is called \emph{the principal part of $\widetilde{B}$}.
The pair $(\Lambda,\widetilde{B})$ is said to be \emph{compatible}
if, for $i\in[1,\ell]$ and $j\in[1,\ell-n]$, 
\[
\sum_{k=1}^{\ell}b_{kj}\lambda_{ki}=\delta_{ij}d_{j}\;\text{for some}\;d_{j}\in\mathbb{Z}_{>0}.
\]
Note that, when $(\Lambda,\widetilde{B})$ is compatible, $\widetilde{B}$
has full rank $\ell-n$ and its principal part $B=(b_{ij})_{i,j\in[1,\ell-n]}$
is skew-symmetrizable \cite[Proposition 3.3]{MR2146350}. We will
assume that $B$ is skew-symmetric.

For $k\in[1,\ell-n]$, define $E^{(k)}=(e_{ij})_{i,j\in[1,\ell]}$
and $F^{(k)}=(f_{ij})_{i,j\in[1,\ell-n]}$ as follows:
\begin{align*}
e_{ij}=\begin{cases}
\delta_{i,j} & \text{if}\;j\neq k,\\
-1 & \text{if}\;i=j=k,\\
\max(0,-b_{ik}) & \text{if}\;i\neq j=k,
\end{cases} &  & f_{ij}=\begin{cases}
\delta_{i,j} & \text{if}\;i\neq k,\\
-1 & \text{if}\;i=j=k,\\
\max(0,b_{kj}) & \text{if}\;i=k\neq j.
\end{cases}
\end{align*}
Set 
\begin{align*}
\mu_{k}(\Lambda)=(E^{(k)})^{T}\Lambda E^{(k)} &  & \mu_{k}(\widetilde{B})=E^{(k)}\widetilde{B}F^{(k)}.
\end{align*}
Then $\mu_{k}(\Lambda,\widetilde{B}):=(\mu_{k}(\widetilde{B}),\mu_{k}(\Lambda))$
is again compatible \cite[Proposition 3.4]{MR2146350}. It is said
that\emph{ $\mu_{k}(\Lambda,\widetilde{B})$ is obtained from $(\Lambda,\widetilde{B})$
by the mutation in direction $k$.} Note that $\mu_{k}(\mu_{k}(\Lambda,\widetilde{B}))=(\Lambda,\widetilde{B})$.

The pair $\mathscr{S}=(\{X_{i}\}_{i\in[1,\ell]},\widetilde{B},\Lambda)$
is called\emph{ a quantum seed in $\mathcal{F}$,} and $\{X_{i}\}_{i\in[1,\ell]}$
is called\emph{ the quantum cluster of $\mathscr{S}$. }For $k\in[1,\ell-n]$,
define $\mu_{k}(\{X_{i}\}_{i\in[1,\ell]})=\{X'_{i}\}_{i\in[1,\ell]}\subset\mathcal{F}\setminus\{0\}$
by 
\begin{itemize}
\item $X'_{i}=X_{i}$ if $i\neq k$, 
\item $X'_{k}=X^{-\bm{e}_{k}+\sum_{j;b_{jk}>0}b_{jk}\bm{e}_{j}}+X^{-\bm{e}_{k}-\sum_{j;-b_{jk}>0}b_{jk}\bm{e}_{j}}$. 
\end{itemize}
Then there is an injective $\mathbb{Q}[q^{\pm1/2}]$-algebra homomorphisms
$\mathcal{T}(\mu_{k}(\Lambda))\to\mathcal{F}(\Lambda)$ given by $X_{i}^{\pm1}\mapsto(X'_{i})^{\pm1}$
($i\in[1,\ell]$). Moreover there exist a basis $\{\bm{c}_{i}\}_{i\in[1,\ell]}$
of $\mathbb{Z}^{\ell}$ and a $\mathbb{Q}(q^{1/2})$-algebra automorphism
$\tau\colon\mathcal{F}(\Lambda)\to\mathcal{F}(\Lambda)$ such that
$\tau(X^{\bm{c}_{i}})=X'_{i}$ for $i\in[1,\ell]$ \cite[Proposition 4.7]{MR2146350}.
Hence the map above is extended to the isomorphism $\mathcal{F}(\mu_{k}(\Lambda))\to\mathcal{F}(\Lambda)$.
Through this isomorphism, we identify $\mathcal{F}(\mu_{k}(\Lambda))$
with $\mathcal{F}(\Lambda)$, and henceforth always write $\mathcal{F}$
for this skew-field. Write 
\[
\mu_{k}(\mathscr{S}):=(\mu_{k}(\{X_{i}\}_{i\in[1,\ell]}),\mu_{k}(\widetilde{B}),\mu_{k}(\Lambda))
\]
and this is called \emph{a quantum seed obtained from the mutation
of $\mathscr{S}$ in direction $k$.} Note that $\mu_{k}(\mu_{k}(\mathscr{S}'))=\mathscr{S}'$
for any quantum seed $\mathscr{S}'$ and $k\in[1,\ell-n]$. By the
argument above, we can consider the iterated mutations in arbitrary
various directions $k\in[1,\ell-n]$. The subset $\{X_{i}\mid i\in[\ell-n+1,\ell]\}$,
called the set of \emph{frozen variables}, is contained in the quantum
cluster of an arbitrary seed obtained by iterated mutations of $\mathscr{S}$.

\emph{The quantum cluster algebra} $\mathscr{A}_{q^{\pm1/2}}(\mathcal{\mathscr{S}})$
is defined as the $\mathbb{Q}[q^{\pm1/2}]$-subalgebra of $\mathcal{F}$
generated by the union of the quantum clusters of all quantum seeds
obtained by iterated mutations of $\mathscr{S}$. An element $M\in\mathscr{A}_{q^{\pm1/2}}(\mathscr{S})$
is called \emph{a quantum cluster monomial} if there exists a quantum
cluster $\{X'_{i}=(X')^{\bm{e}_{i}}\}_{i\in[1,\ell]}$ of a quantum
seed obtained by iterated mutations of $\mathscr{S}$ such that $M=(X')^{\bm{a}}$
for some $\bm{a}\in\mathbb{Z}_{\geq0}^{\ell}$. 
\end{defn}
\begin{prop}[{{\cite[Corollary 5.2]{MR2146350}}}]
\label{p:laurentpheno} The quantum cluster algebra $\mathscr{A}_{q^{\pm1/2}}(\mathscr{S})$
is contained in the based quantum torus generated by the quantum cluster
of an arbitrary quantum seed obtained by iterated mutations of $\mathscr{S}$.
\end{prop}

\subsection{Quantum cluster algebra structures on quantum unipotent subgroups
and quantum unipotent cells}

In this subsection, we review the construction of the quantum cluster
algebra structure on $\mathbf{A}_{q}\!\left[N_{-}\left(w\right)\right]$
following \cite{MR2822235,MR2833478,MR3090232}. We note that our
convention is slightly different from Gei\ss-Leclerc-Schr\"oer's one,
see Remark \ref{r:covention}.
\begin{defn}
\label{d:preproj} \emph{A finite quiver} $\mathbf{Q}=(\mathbf{Q}_{0},\mathbf{Q}_{1},s,t)$
is a datum such that 

\textup{(1)} $\mathbf{Q}_{0}$ is a finite set, called the set of
vertices, 

\textup{(2)} $\mathbf{Q}_{1}$ is a finite set, called the set of
arrows, 

\textup{(3)} $s,t\colon\mathbf{Q}_{1}\to\mathbf{Q}_{0}$ are maps,
and it is said that $a\in\mathbf{Q}_{1}$ is an arrow from $s(a)$
to $t(a)$. 

Here we take a finite quiver $\mathbf{Q}$ such that $\mathbf{Q}_{0}=I$,
$s(a)\neq t(a)$ for all $a\in\mathbf{Q}_{1}$ and $a_{ij}(:=\langle h_{i},\alpha_{j}\rangle)=-\#\{a\in\mathbf{Q}_{1}\mid s(a)=i,t(a)=j\}-\#\{a\in\mathbf{Q}_{1}\mid s(a)=j,t(a)=i\}$.
Such a quiver $\mathbf{Q}$ is called \emph{a finite quiver without
edge loops which corresponds to the symmetric generalized Cartan matrix
$A$} \cite[Subsection 2.1 and 4.1]{MR2822235}. Let $\overline{\mathbf{Q}}=(\mathbf{Q}_{0},\overline{\mathbf{Q}}_{1}:=\mathbf{Q}_{1}\coprod\mathbf{Q}_{1}^{\ast},s,t)$
be the double quiver of $\mathbf{Q}$, which is obtained from $\mathbf{Q}$
by adding to each arrow $a\in\mathbf{Q}_{1}$ an arrow $a^{\ast}\in\mathbf{Q}_{1}^{\ast}$
such that $s(a^{\ast})=t(a)$ and $t(a^{\ast})=s(a)$. Set 
\[
\Pi:=\mathbb{C}\overline{\mathbf{Q}}/\left(\sum_{a\in\mathbf{Q}_{1}}(a^{\ast}a-aa^{\ast})\right),
\]
Here $\mathbb{C}\overline{\mathbf{Q}}$ is a path algebra of $\overline{\mathbf{Q}}$,
and $\left(\sum_{a\in\mathbf{Q}_{1}}(a^{\ast}a-aa^{\ast})\right)$
is the two-sided ideal generated by $\sum_{a\in\mathbf{Q}_{1}}(a^{\ast}a-aa^{\ast})$.
This is called \emph{the preprojective algebra associated with $\mathbf{Q}$.}
Denote by $\epsilon_{i}$ the idempotent of $\Pi$ corresponding to
$i\in I$. For a finite dimensional $\Pi$-module $X$, write $\dimv X:=-\sum_{i\in I}(\dim_{\mathbb{C}}\epsilon_{i}.X)\alpha_{i}\in Q_{-}$.
Remark that we do not regard $\dimv X$ as an element of $Q_{+}$.
A finite dimensional $\Pi$-module $X$ is said to be\emph{ nilpotent
}if there exists $N\in\mathbb{Z}_{\ge0}$ such that $a_{1}\cdots a_{N}.X=0$
for any sequence $(a_{1},\dots,a_{N})\in\overline{\mathbf{Q}}_{1}^{N}$
.
\end{defn}
\begin{prop}[{{\cite[Lemma 1]{MR1781930}}}]
\label{p:CB} For any finite dimensional $\Pi$-module $X,Y$, we
have 
\[
(\dimv X,\dimv Y)=\dim_{\mathbb{C}}\Hom_{\Pi}(X,Y)+\dim_{\mathbb{C}}\Hom_{\Pi}(Y,X)-\dim_{\mathbb{C}}\Ext_{\Pi}^{1}(X,Y).
\]
\end{prop}
\begin{rem}
For any finite dimensional $\Pi$-modules $X,Y$, we have 
\[
\dim_{\mathbb{C}}\Ext_{\Pi}^{1}(X,Y)=\dim_{\mathbb{C}}\Ext_{\Pi}^{1}(Y,X)
\]
by Proposition \ref{p:CB}.
\end{rem}
A finite dimensional nilpotent $\Pi$-module $X$ determines an element
$\varphi_{X}$ of $\mathbb{C}[\boldsymbol{N}_{-}]=\mathbf{U}\left(\mathfrak{n}_{-}\right)_{\mathrm{gr}}^{*}$
through Lusztig's construction of the universal enveloping algebra
$\mathbf{U}\left(\mathfrak{n}_{-}\right)$ as a space $\mathcal{M}$
consisting of certain constructible functions with convolution product
\cite{MR1758244}. There is a short summary, for instance, in \cite[Subsection 2.2]{MR2822235}.
However we remark that, in this paper, we consider the convolution
product on $\mathcal{M}$ opposite to the one in \cite[Subsection 2.2]{MR2822235}.
See also Remark \ref{r:covention}. The following are important properties
of $\varphi_{X}$. 
\begin{prop}[\cite{MR1758244,MR2144987,MR2360317}]
\label{p:phimap} Let $X,Y$ be finite dimensional nilpotent $\Pi$-modules.
The following hold:

\textup{\textup{(1)} $\wt\varphi_{X}=\dimv X$. }

\textup{\textup{(2)}} $\varphi_{X}\varphi_{Y}=\varphi_{X\oplus Y}$. 

\textup{\textup{(3)}} Suppose that $\dim_{\mathbb{C}}\Ext_{\Pi}^{1}(X,Y)=1$.
Write non-split short exact sequences as 
\begin{align*}
0\to X\to Z_{1}\to Y\to0 &  &  & 0\to Y\to Z_{2}\to X\to0.
\end{align*}
 Then we have $\varphi_{X}\varphi_{Y}=\varphi_{Z_{1}}+\varphi_{Z_{2}}$.
\end{prop}
\begin{notation}
Let $w\in W$ and fix $\bm{i}=(i_{1},\dots,i_{\ell})\in I(w)$ . Then,
for $k=1,\dots,\ell$, we set 
\begin{align*}
k^{+} & :=\min(\{\ell+1\}\cup\{k+1\leq j\leq\ell\mid i_{j}=i_{k}\}),\\
k^{-} & :=\max(\{0\}\cup\{1\leq j\leq k-1\mid i_{j}=i_{k}\}),\\
k^{\mathrm{max}} & :=\max\{1\leq j\leq\ell\mid i_{j}=i_{k}\},\\
k^{\mathrm{min}} & :=\min\{1\leq j\leq\ell\mid i_{j}=i_{k}\}.
\end{align*}

Moreover, set $I_{w}:=\{i\in I\mid i=i_{k}\ \text{for some }k=1,\dots,\ell\}$
. Then we can easily check that $I_{w}$ does not depend on the choice
of $\bm{i}$. 
\end{notation}
\begin{defn}
\label{d:C_w} Let $S_{i}$ be the (simple) $\Pi$-module such that
$\dimv S_{i}=-\alpha_{i}$ for $i\in I$. For a $\Pi$-module $X$
and $i\in I$, define $\soc_{i}(X)\subset X$ by the sum of all submodules
of $X$ isomorphic to $S_{i}$. For a sequence $(i_{1},\dots,i_{k})\in I^{k}$
($k\in\mathbb{Z}_{>0}$), there exists a unique chain 
\[
X\supset X_{0}\supset X_{1}\supset X_{2}\supset\cdots\supset X_{k}=0
\]
of submodules of $X$ such that $X_{j-1}/X_{j}\simeq\soc_{i_{j}}(X/X_{j})$
for $j=1,\dots,k$. Set $\soc_{(i_{1},\dots,i_{k})}(X):=X_{0}$. For
$i\in I$, denote by $\hat{I}_{i}$ the indecomposable injective $\Pi$-module
with socle $S_{i}$. Let $w\in W$ and $\bm{i}=(i_{1},\dots,i_{\ell})\in I(w)$.
For $k=1,\dots,\ell$, set 
\[
V_{\bm{i},k}:=\soc_{(i_{1},\dots,i_{k})}(\hat{I}_{i_{k}}).
\]
Set $V_{\bm{i}}:=\bigoplus_{k=1,\dots,\ell}V_{\bm{i},k}$. Define
$\mathcal{C}_{w}$ as a full subcategory of the category of $\Pi$-modules
consisting of all $\Pi$-modules $X$ such that there exist $t\in\mathbb{Z}_{>0}$
and a surjective homomorphism $V_{\bm{i}}^{\oplus t}\to X$. Then
it is known that $\mathcal{C}_{w}$ does not depend on the choice
of $\bm{i}\in I(w)$. Note that all objects of $\mathcal{C}_{w}$
are nilpotent $\Pi$-modules. An object $C\in\mathcal{C}_{w}$ is
called\emph{ $\mathcal{C}_{w}$-projective} (resp.~\emph{$\mathcal{C}_{w}$-injective})
if $\Ext_{\Pi}^{1}(C,X)=0$ (resp.~$\Ext_{\Pi}^{1}(X,C)=0$) for
all $X\in\mathcal{C}_{w}$. The category $\mathcal{C}_{w}$ is closed
under extension and is Frobenius. In particular, an object $X\in\mathcal{C}_{w}$
is $\mathcal{C}_{w}$-projective if and only if it is $\mathcal{C}_{w}$-injective.
An object $T$ of $\mathcal{C}_{w}$ is called\emph{ $\mathcal{C}_{w}$-maximal
rigid} if $\Ext_{\Pi}^{1}(T\oplus X,X)=0$ with $X\in\mathcal{C}_{w}$
implies that $X$ is isomorphic to a direct summand of a direct sum
of copies of $T$. A $\Pi$-module $M$ is called \emph{basic} if
it is written as a direct sum of pairwise non-isomorphic indecomposable
modules. Then, in fact, $V_{\bm{i}}$ is a basic $\mathcal{C}_{w}$-maximal
rigid module and $V_{\bm{i},k^{\mathrm{max}}}$ is the $\mathcal{C}_{w}$-projective-injective
module with socle $S_{i_{k}}$ for $k=1,\dots,\ell$. See \cite{MR2521253}
for more details, and \cite[Subsection 2.4]{MR2822235} for more detailed
summaries.

Let $T$ be a basic $\mathcal{C}_{w}$-maximal rigid module and $T=T_{1}\oplus\cdots\oplus T_{\ell}$
its indecomposable decomposition. From now on, we write $I_{w}=[1,n]$
for simplicity, and always number indecomposable summands of $T$
so that $T_{\ell-n+i}$, $i\in I_{w}$ is the $\mathcal{C}_{w}$-projective-injective
module with socle $S_{i}$. Note that this labelling is different
from the labelling $V_{\bm{i}}=\bigoplus_{k\in[1,\ell]}V_{\bm{i},k}$.
Let $\Gamma_{T}$ be the Gabriel quiver of $A_{T}:=\End_{\Pi}(T)^{\mathrm{op}}$,
that is, the vertex set of $\Gamma_{T}$ is $[1,\ell]$ and $d_{ij}:=\dim_{\mathbb{C}}\Ext_{A_{T}}^{1}(S_{T_{i}},S_{T_{j}})$
arrows from $i$ to $j$, where $S_{T_{i}}$ is the head of a (projective)
$A_{T}$-module $\Hom_{\Pi}(T,T_{i})$. Define $\widetilde{B}_{T}=(b_{ij})_{i\in[1,\ell],j\in[1,\ell-n]}$
by $b_{ij}:=d_{ji}-d_{ij}$. The following proposition is an essential
results for the additive categorification of cluster algebras. 
\end{defn}
\begin{prop}[\cite{MR2521253,GLS:uniparXiv}]
\label{p:mutcat} In the setting above, the following hold: 

\textup{\textup{(1)}} $\ell=\ell(w)$.

\textup{\textup{(2)} }For any $k\in[1,\ell-n]$, there exists a unique
indecomposable $\Pi$-module in $\mathcal{C}_{w}$ such that $T_{k}^{\ast}\not\simeq T_{k}$
and $(T/T_{k})\oplus T_{k}^{\ast}$ is a basic $\mathcal{C}_{w}$-maximal
rigid module. This basic $\mathcal{C}_{w}$-maximal rigid module is
denoted by $\mu_{T_{k}}(T)$ and called \emph{the mutation of $T$
in direction $T_{k}$.}

\textup{\textup{(3)} }For any $k\in[1,\ell-n]$, $\mu_{k}(\widetilde{B}_{T})=\widetilde{B}_{\mu_{T_{k}}(T)}$.

\textup{\textup{(4)} }For any $k\in[1,\ell-n]$, we have $\dim_{\mathbb{C}}\Ext_{\Pi}^{1}(T_{k},T_{k}^{\ast})=1$,
and there exists non-split exact sequences 
\begin{align*}
0\to T_{k}\to T_{-}\to T_{k}^{\ast}\to0 &  &  & 0\to T_{k}^{\ast}\to T_{+}\to T_{k}\to0
\end{align*}
such that $T_{-}\simeq\bigoplus_{j;b_{jk}<0}T_{j}^{\oplus(-b_{jk})}$
and $T_{+}\simeq\bigoplus_{j;b_{jk}>0}T_{j}^{\oplus b_{jk}}$.
\end{prop}
Note that, by Proposition \ref{p:phimap} and \ref{p:mutcat}, we
have 
\begin{align}
\varphi_{T_{k}}\varphi_{T_{k}^{\ast}}=\varphi_{T_{+}}+\varphi_{T_{-}}=\prod_{j;b_{jk}>0}\varphi_{T_{j}}^{b_{jk}}+\prod_{j;b_{jk}<0}\varphi_{T_{j}}^{-b_{jk}}.\label{clmut}
\end{align}
This is \emph{an additive categorification of mutation.} See \cite[Subsection 2.7]{MR2822235}
and references therein for more details. An object $T$ of $\mathcal{C}_{w}$
is said to be\emph{ reachable (in $\mathcal{C}_{w}$)} if $T$ is
isomorphic to a direct summand of a direct sum of copies of a basic
$\mathcal{C}_{w}$-maximal rigid module which is obtained from $V_{\bm{i}}$
by iterated mutations. In fact, the notion of reachable does not depend
on the choice of $\bm{i}$ \cite[Proposition III.4.3]{MR2521253}.
\begin{rem}
\label{r:CB} Let $T$ be a basic reachable $\mathcal{C}_{w}$-maximal
rigid module, and $T=T_{1}\oplus\cdots\oplus T_{\ell}$ its indecomposable
decomposition. By Proposition \ref{p:CB}, for any $i,j\in[1,\ell]$,
we have 
\[
(\dimv T_{i},\dimv T_{j})=\dim_{\mathbb{C}}\Hom_{\Pi}(T_{i},T_{j})+\dim_{\mathbb{C}}\Hom_{\Pi}(T_{j},T_{i}).
\]
\end{rem}
\begin{defn}
Let $\bm{i}=(i_{1},\dots,i_{\ell})\in I(w)$. For $1\leq a\leq b\leq\ell$
with $i_{a}=i_{b}$, there exists a natural injective homomorphism
$V_{\bm{i},a^{-}}\to V_{\bm{i},b}$ of $\Pi$-modules, and the cokernel
of this homomorphism is denoted by $M_{\bm{i}}[b,a]$. Here we set
$V_{\bm{i},0}:=0$. In particular, $M_{\bm{i}}[b,b^{\mathrm{min}}]$
is isomorphic to $V_{\bm{i},b}$. Gei\ss-Leclerc-Schr\"oer shows that
$M_{\bm{i}}[b,a]$ is reachable for all $1\leq a\leq b\leq\ell$ with
$i_{a}=i_{b}$ \cite[section 13]{MR2822235}. 
\end{defn}
We use the notation in Definition \ref{d:C_w}. Gei\ss-Leclerc-Schr\"oer
construct a quantum cluster algebra $\mathscr{A}_{\mathbb{Q}(q)}(\mathcal{C}_{w})$
associated with $\mathcal{C}_{w}$ as we shall recall. Let $T$ be
a basic $\mathcal{C}_{w}$-maximal rigid module and $T=T_{1}\oplus\cdots\oplus T_{\ell}$
its indecomposable decomposition. Define $\Lambda_{T}:=(\lambda_{ij})_{i,j\in[1,\ell]}$
by 
\begin{align*}
\lambda_{ij} & :=\dim_{\mathbb{C}}\Hom_{\Pi}(T_{i},T_{j})-\dim_{\mathbb{C}}\Hom_{\Pi}(T_{j},T_{i}).
\end{align*}

Gei\ss-Leclerc-Schr\"oer have shown the following properties: 
\begin{prop}[{\cite[Proposition 10.1, Proposition 10.2]{MR3090232}}]
 \textup{\textup{(1)}} $(\widetilde{B}_{T},\Lambda_{T})$ is compatible
in the sense of Definition \ref{d:qcluster}. 

\textup{\textup{(2)}}$\mu_{k}(\widetilde{B}_{T},\Lambda_{T})=(\widetilde{B}_{\mu_{T_{k}}(T)},\Lambda_{\mu_{T_{k}}(T)})$
for $k\in[1,\ell-n]$.
\end{prop}
\begin{defn}
The quantum cluster algebra $\mathscr{A}_{q^{\pm1/2}}\left(\mathcal{C}_{w}\right)$ is defined as the quantum cluster algebra with the initial seed $((X_{T})_{i})_{i\in[1,\ell]},\widetilde{B}_{T},\Lambda_{T})$
for a basic reachable $\mathcal{C}_{w}$-maximal rigid module $T$.
Note that this algebra $\mathscr{A}_{q^{\pm1/2}}(C_{w})$ does not
depend on the choice of $T$. By the properties above, we may write
\[
\mu_{k}(((X_{T})_{i})_{i\in[1,\ell]},\widetilde{B}_{T},\Lambda_{T})=(((X_{\mu_{T_{k}}(T)})_{i})_{i\in[1,\ell]},\widetilde{B}_{\mu_{T_{k}}(T)},\Lambda_{\mu_{T_{k}}(T)})
\]
for $k\in[1,\ell-n]$. Moreover, for $\bm{a}=(a_{1},\dots,a_{\ell})\in\mathbb{Z}_{\geq0}^{\ell}$,
set $X_{\bigoplus_{i\in[1,\ell]}T_{i}^{\oplus a_{i}}}:=(X_{T})^{\bm{a}}$.
Then the quantum cluster monomials of $\mathscr{A}_{q^{\pm1/2}}(\mathcal{C}_{w})$
are indexed by reachable $\Pi$-modules in $\mathcal{C}_{w}$.

Set 
\[
Y_{R}:=q^{(\dimv R,\dimv R)/4}X_{R}.
\]
for every reachable $\Pi$-module $R$ in $\mathcal{C}_{w}$. Recall
that $\dimv R\in Q_{-}$. Define \emph{the rescaled quantum cluster
algebra $\mathscr{A}_{q^{\pm1}}(\mathcal{C}_{w})$} as an $\mathcal{A}(:=\mathbb{Q}[q^{\pm1}])$-subalgebra of $\mathscr{A}_{q^{\pm1/2}}(\mathcal{C}_{w})$
generated by $\{Y_{R}\mid R\;\text{is reachable in}\;\mathcal{C}_{w}\}$. For any basic reachable $\mathcal{C}_{w}$-maximal rigid module $T=T_{1}\oplus\cdots\oplus T_{\ell}$, the rescaled quantum cluster algebra $\mathscr{A}_{q^{\pm1}}(\mathcal{C}_{w})$ is contained in \emph{the rescaled based quantum torus $\mathcal{T}_{\mathcal{A},T}:=\mathcal{A}[Y_{T_{k}}^{\pm1}\mid k\in[1,\ell]](\subset\mathcal{F})$}
\cite[Lemma 10.4 and Proposition 10.5]{MR3090232} (they are cited
as (\ref{qmonom}) and Proposition \ref{p:qmutation} below). Note
that, for $(a_{1},\dots,a_{\ell})\in\mathbb{Z}_{\geq0}^{\ell}$, we
have 
\begin{equation}
Y_{R}=q^{\alpha(R)}Y_{T_{1}}^{a_{1}}\cdots Y_{T_{\ell}}^{a_{\ell}},\label{qmonom}
\end{equation}
here we set $R:=\bigoplus_{i\in[1,\ell]}T_{i}^{\oplus a_{i}}$ and
\[
\alpha(R):=\sum_{i\in[1,\ell]}a_{i}(a_{i}-1)\dim_{\mathbb{C}}\Hom_{\Pi}(T_{i},T_{i})/2+\sum_{i<j;i,j\in[1,\ell]}a_{i}a_{j}\dim_{\mathbb{C}}\Hom_{\Pi}(T_{j},T_{i}).
\]
Note that $\mathbf{I}:=\{q^{m}Y_{\bigoplus_{i\in[\ell-n+1,\ell]}T_{i}^{a_{i}}}\mid(a_{\ell-n+1},\dots,a_{\ell})\in\mathbb{Z}_{\geq0}^{n},m\in\mathbb{Z}\}$
is an Ore set in $\mathscr{A}_{q^{\pm1}}(\mathcal{C}_{w})$. Set $\widetilde{\mathscr{A}}_{q^{\pm1}}(\mathcal{C}_{w}):=\mathscr{A}_{q^{\pm1}}(\mathcal{C}_{w})[\mathbf{I}^{-1}]$,
and $\mathscr{A}_{\mathbb{Q}(q)}(\mathcal{C}_{w}):=\mathbb{Q}(q)\otimes_{\mathcal{A}}\mathscr{A}_{q^{\pm1}}(\mathcal{C}_{w})$,
$\widetilde{\mathscr{A}}_{\mathbb{Q}(q)}(\mathcal{C}_{w}):=\mathbb{Q}(q)\otimes_{\mathcal{A}}\widetilde{\mathscr{A}}_{q^{\pm1}}(\mathcal{C}_{w})$.

For $X\in\mathcal{C}_{w}$, denote by $I(X)$ the injective hull of
$X$ in $\mathcal{C}_{w}$, and by $\Omega_{w}^{-1}(X)$ the cokernel
of the corresponding injective homomorphism $X\to I(X)$. Hence we
have an exact sequence 
\[
0\to X\to I(X)\to\Omega_{w}^{-1}(X)\to0.
\]
\end{defn}
\begin{prop}[{{\cite[Proposition 13.4]{MR2822235}}}]
\label{p:syzygy} Let $w\in W$, $T$ a basic reachable $\mathcal{C}_{w}$-maximal
rigid module and $T=T_{1}\oplus\cdots\oplus T_{\ell}$ its indecomposable
decomposition. Then $T':=\Omega_{w}^{-1}(T)\oplus\bigoplus_{i\in I_{w}}T_{\ell-n+i}$
is also a basic reachable $\mathcal{C}_{w}$-maximal rigid module;
hence there exists a bijection $[1,\ell-n]\to[1,\ell-n],k\mapsto k^{\ast}$
such that $T'_{k^{\ast}}=\Omega_{w}^{-1}(T_{k})$.

Let $k\in[1,\ell-n]$ and write $\mu_{T_{k}}(T)=(T/T_{k})\oplus T_{k}^{\ast}$.
Then we have 
\[
\mu_{T'_{k^{\ast}}}(T')=(T'/T'_{k^{\ast}})\oplus\Omega_{w}^{-1}(T_{k}^{\ast}).
\]
\end{prop}
\begin{rem}
\label{r:covention} Let $w\in W$. In this remark, we explain the
difference between our convention and Gei\ss-Leclerc-Schr\"oer's one in
\cite{MR2822235,MR2833478,MR3090232}. An object $\mathcal{X}$ in
Gei\ss-Leclerc-Schr\"oer's papers is denoted by $\mathcal{X}^{\mathrm{GLS}}$
here.

The category $\mathcal{C}_{w}$ is the same category as $\mathcal{C}_{w^{-1}}^{\mathrm{GLS}}$.
Moreover $N_{-}(w)=(N(w^{-1})^{\mathrm{GLS}})^{T}$ and $N_{-}^{w}=((N^{w^{-1}})^{\mathrm{GLS}})^{T}$.
We omitted the definition of $\varphi_{X}$ for a finite dimensional
nilpotent $\Pi$-module $X$, however the algebra $\mathcal{M}$ used
for its precise definition (see Definition \ref{d:preproj}) is the
same space as $\mathcal{M}^{\mathrm{GLS}}$ in \cite[Subsection 2.2]{MR2822235}
equipped with the opposite convolution product.

Thus there exist algebra isomorphisms $\mathbb{C}[N_{-}(w)]\to\mathbb{C}[N(w^{-1})^{\mathrm{GLS}}]$
and $\mathbb{C}[N_{-}^{w}]\to\mathbb{C}[N^{w^{-1},\mathrm{GLS}}]$
given by $f\to f\circ(-)^{T}$. Moreover $\varphi_{X}=\varphi_{X}^{\mathrm{GLS}}\circ(-)^{T}$
for all $X\in\mathcal{C}_{w}=\mathcal{C}_{w^{-1}}^{\mathrm{GLS}}$.
See also \cite[Chapter 6]{MR2822235}. (This is the reason why we
consider the opposite product on $\mathcal{M}$.)

The quantum nilpotent subalgebra $\Uq(\mathfrak{n}(w^{-1}))^{\mathrm{GLS}}$
in \cite{MR3090232} is equal to $\Aq[N_{-}(w)]^{\vee}$. Gei\ss-Leclerc-Schr\"oer
consider a $\mathbb{Q}(q)$-algebra $\Aq(\mathfrak{n}(w^{-1}))^{\mathrm{GLS}}$,
called the quantum coordinate ring, which is defined in $(\Uq^{+})^{\ast}$
\cite[(4.6)]{MR3090232}, and define an algebra isomorphism $\Psi^{\mathrm{GLS}}\colon\Uq(\mathfrak{n}(w^{-1}))^{\mathrm{GLS}}\to\Aq(\mathfrak{n}(w^{-1}))^{\mathrm{GLS}}$
by using a non-degenerate bilinear form $(\ ,\ )^{\mathrm{GLS}}$ \cite[Proposition 4.1]{MR3090232}.
Actually, for $x\in(\Uq^{+})_{\beta}$, $y\in(\Uq^{+})_{\beta'}$
($\beta,\beta'\in Q_{+}$), we have 
\begin{align*}
(x,y)^{\mathrm{GLS}} & =\delta_{\beta,\beta'}(1-q^{-2})^{\height\beta}\overline{(\overline{x},\overline{y})_{L}^{+}}\\
 & =\delta_{\beta,\beta'}(1-q^{-2})^{\height\beta}\overline{(\overline{x^{\vee}},\overline{y^{\vee}})_{L}}\\
 & =\delta_{\beta,\beta'}(1-q^{-2})^{\height\beta}(x^{\vee},\sigma(\overline{y^{\vee}}))_{L}\\
 & =q^{(\beta,\beta)/2}(q^{-1}-q)^{\height\beta}(x^{\vee},\varphi(y))_{L}.
\end{align*}
The last equality follows from Proposition \ref{p:dualbar}. By the
way, there exists a $\mathbb{Q}(q)$-algebra automorphism $m_{\mathrm{norm}}\colon\Uq^{-}\to\Uq^{-}$
given by $f_{i}\mapsto(q^{-1}-q)^{-1}f_{i}$ for $i\in I$. We now
have the following $\mathbb{Q}(q)$-algebra isomorphism; 
\[
I_{\mathrm{norm}}\colon\Aq[N_{-}(w)]\xrightarrow{m_{\mathrm{norm}}}\Aq[N_{-}(w)]\xrightarrow{\vee}\Uq(\mathfrak{n}(w^{-1}))^{\mathrm{GLS}}\xrightarrow{\Psi^{\mathrm{GLS}}}\Aq(\mathfrak{n}(w^{-1}))^{\mathrm{GLS}},
\]
which maps $x\in(\Uq^{-})_{\beta}$ ($\beta\in-Q_{+}$) to $q^{(\beta,\beta)/2}(x,\varphi(-))_{L}$.
By using this isomorphism, we describe their results. Note that $I_{\mathrm{norm}}(D_{w\lambda,w'\lambda})=q^{(w\lambda-w'\lambda,w\lambda-w'\lambda)/2}D_{w'\lambda,w\lambda}^{\mathrm{GLS}}$
for $w,w'\in W$ and $\lambda\in P_{+}$ \cite[(5.5)]{MR3090232}.

The definitions of the quantum cluster algebra $\mathscr{A}_{q^{\pm1/2}}(\mathcal{C}_{w})=\mathscr{A}_{q^{\pm1/2}}(\mathcal{C}_{w^{-1}}^{\mathrm{GLS}})$
are the same. We have $Y_{R}=q^{(\dimv R,\dimv R)/2}Y_{R}^{\mathrm{GLS}}$
for every reachable $\Pi$-module $R$ \cite[(10.16)]{MR3090232}.
Note that $(\dimv R,\dimv R)/2\in\mathbb{Z}$. Therefore we have $\mathscr{A}_{q^{\pm1}}(\mathcal{C}_{w})=\mathscr{A}_{\mathbb{A}}(\mathcal{C}_{w^{-1}}^{\mathrm{GLS}})^{\mathrm{GLS}}$. 
\end{rem}
The following propositions describe mutations of quantum clusters
and twisted dual bar involutions in $\mathscr{A}_{q^{\pm1}}(\mathcal{C}_{w})$. 
\begin{prop}[{{\cite[Proposition 10.5]{MR3090232}}}]
\label{p:qmutation} Let $T$ be a basic reachable $\mathcal{C}_{w}$-maximal
rigid module, and $T=T_{1}\oplus\cdots\oplus T_{\ell}$ its indecomposable
decomposition. Fix $k\in[1,\ell-n]$. Write $\tilde{B}_{T}=(b_{ij})_{i\in[1,\ell],j\in[1,\ell-n]}$
and $\mu_{T_{k}}(T)=(T/T_{k})\oplus T_{k}^{\ast}$. Set $T_{+}:=\bigoplus_{j;b_{jk}>0}T_{j}^{\oplus b_{jk}}$
and $T_{-}:=\bigoplus_{j;b_{jk}<0}T_{j}^{\oplus(-b_{jk})}$. Then
we have 
\[
Y_{T_{k}^{\ast}}Y_{T_{k}}=q^{-\dim_{\mathbb{C}}\Hom_{\Pi}(T_{k},T_{k}^{\ast})}(qY_{T_{+}}+Y_{T_{-}}).
\]
\end{prop}
\begin{prop}[{{\cite[Lemma 10.6, Lemma 10.7]{MR3090232}}}]
\label{p:GLSdualbar} Let $T$ be a basic reachable $\mathcal{C}_{w}$-maximal
rigid module. Then there exists a unique $\mathbb{Q}$-algebra anti-involution
$\sigma'_{T}$ on $\mathcal{T}_{\mathcal{A},T}$ such that 
\begin{align*}
q\mapsto q^{-1},\ Y_{R}\mapsto q^{-(\dimv R,\dimv R)/2-\dim_{\mathbb{C}}R}Y_{R}
\end{align*}
for every direct summand $R$ of a direct sum of copies of $T$. Moreover
$\sigma'_{T}$ induces $\mathbb{Q}$-algebra anti-involutions $\sigma'$
on $\mathscr{A}_{q^{\pm1}}(\mathcal{C}_{w})$ and $\widetilde{\mathscr{A}}_{q^{\pm1}}(\mathcal{C}_{w})$,
and $\sigma'$ does not depend on the choice of a basic reachable
$\mathcal{C}_{w}$-maximal rigid module $T$. 
\end{prop}
Gei\ss-Leclerc-Schr\"oer showed that a rescaled quantum cluster algebra
$\mathscr{A}_{\mathbb{Q}(q)}(\mathcal{C}_{w})$ gives an additive
categorification of the quantum unipotent subgroup $\Aq[N_{-}(w)]$
as follows. 
\begin{prop}[{{\cite[Theorem 12.3]{MR3090232}}}]
\label{p:GLSqclus} Let $w\in W$ and ${\bm{i}}=(i_{1},\dots,i_{\ell})\in I(w)$.
Then there is an isomorphism of $\mathbb{Q}(q)$-algebras $\kappa\colon\Aq[N_{-}(w)]\to\mathscr{A}_{\mathbb{Q}(q)}(\mathcal{C}_{w})$
given by 
\[
D_{s_{i_{1}}\cdots s_{i_{b}}\varpi_{i_{b}},s_{i_{1}}\cdots s_{i_{a-}}\varpi_{i_{a}}}\mapsto Y_{M[b,a]}
\]
for all $1\leq a\leq b\leq\ell$ with $i_{a}=i_{b}$. Moreover we
have $\sigma'\circ\kappa=\kappa\circ\sigma'$. Recall Definition \ref{d:twdualbar}.
\end{prop}
By Theorem \ref{thm:DeCP}, this result also gives an additive categorification
of the quantum unipotent cell $\Aq[N_{-}^{w}]$.
\begin{cor}
\label{c:GLSqclus}Let $w\in W$ and ${\bm{i}}=(i_{1},\dots,i_{\ell})\in I(w)$.
Then there is an isomorphism of $\mathbb{Q}(q)$-algebras $\widetilde{\kappa}\colon\Aq[N_{-}^{w}]\to\widetilde{\mathscr{A}}_{\mathbb{Q}(q)}(\mathcal{C}_{w})$
given by 
\[
[D_{s_{i_{1}}\cdots s_{i_{b}}\varpi_{i_{b}},s_{i_{1}}\cdots s_{i_{a^{-}}}\varpi_{i_{a}}}]\mapsto Y_{M[b,a]}
\]
for all $1\leq a\leq b\leq\ell$ with $i_{a}=i_{b}$. Moreover we
have $\sigma'\circ\widetilde{\kappa}=\widetilde{\kappa}\circ\sigma'$.
Recall Definition \ref{d:dualbarex}.
\end{cor}
The following is the classical counterpart of the results above due
to Gei\ss-Leclerc-Schr\"oer. Note that we explain it as a ``specialization''
of the results above but it is actually the preceding result of them.
\begin{prop}[{\cite[Theorem 3.1, Theorem 3.3]{MR2822235}}]
\label{p:GLSclus} Let $w\in W$ . For every reachable $\Pi$-module
$R$ in $\mathcal{C}_{w}$, we have $\varphi_{R}\in\mathbb{C}[N_{-}(w)]$,
and the correspondence 
\[
\varphi_{R}(\text{resp.~}[\varphi_{R}])\mapsto1\otimes Y_{R}.
\]
gives the $\mathbb{C}$-algebra isomorphism from $\mathbb{C}[N_{-}(w)]$
(resp.~$\mathbb{C}[N_{-}^{w}]$) to $\mathbb{C}\otimes_{\mathcal{A}}\mathscr{A}_{q^{\pm1}}(\mathcal{C}_{w})$
(resp.~$\mathbb{C}\otimes_{\mathcal{A}}\widetilde{\mathscr{A}}_{q^{\pm1}}(\mathcal{C}_{w})$). 
\end{prop}
\begin{defn}
\label{d:GLSdualbar} Let $T$ be a basic reachable $\mathcal{C}_{w}$-maximal
rigid module and $T=T_{1}\oplus\cdots\oplus T_{\ell}$ its indecomposable
decomposition. Then the $Q_{-}$-grading on $\mathbb{Q}[q^{\pm1}][Y_{T_{k}}\mid k=1,\dots,\ell](\subset\mathcal{T}_{\mathcal{A},T})$
given by $\wt Y_{T_{k}}=\dimv T_{k}$ is extended to the $Q$-grading
on $\mathcal{T}_{\mathcal{A},T}$. A homogeneous element $X\in\mathcal{T}_{\mathcal{A},T}$
is said to be \emph{dual bar invariant} if 
\[
\sigma'_{T}(X)=q^{-(\wt X,\wt X)/2+(\wt X,\rho)}X,
\]
here recall that $\rho:=\sum_{i\in I}\varpi_{i}$ (Definition \ref{d:rootdatum}).
When $X\in\mathscr{A}_{\mathbb{Q}(q)}(\mathcal{C}_{w})$ (resp.~$\widetilde{\mathscr{A}}_{\mathbb{Q}(q)}(\mathcal{C}_{w})$),
the $Q$-grading and the definition of dual bar invariance of homogeneous
elements are compatible with the corresponding notions in $\Aq[N_{-}(w)]$
(resp.~$\Aq[N_{-}^{w}]$) via $\kappa$ (resp.~$\widetilde{\kappa}$).
See Remark \ref{r:twdualbar}. Note that $Y_{R}$ is dual bar invariant
for any reachable $\Pi$-module $R$.
\end{defn}
\begin{rem}
Through $\kappa$ (resp. $\widetilde{\kappa}$), we can translate
also the \emph{nontwisted} dual bar involution $\sigma$ on $\Aq[N_{-}(w)]$ (resp.~$\Aq[N_{-}^{w}]$)
into the involution on $\mathscr{A}_{\mathbb{Q}(q)}(\mathcal{C}_{w})$
(resp. $\widetilde{\mathscr{A}}_{\mathbb{Q}(q)}(\mathcal{C}_{w})$).
Then this involution coincides with the twisted bar involution 
in the sense of \cite[Section 6]{MR2146350} if we take a \emph{grading}
datum $\Sigma=(\sigma_{ij})_{i,j\in[1,\ell]}$ associated with $T$
in Definition \ref{d:GLSdualbar} as $\sigma_{ij}=-(\dimv T_{i},\dimv T_{j})$
for $i,j\in[1,\ell]$ (see \cite[Definition 6.5]{MR2146350} for the
definition of the notion of grading).
\end{rem}
Gei\ss-Leclerc-Schr\"oer also obtained an additive categorification of
the twist automorphism $\eta_{w}^{\ast}$ on the coordinate algebra
$\mathbb{C}[N_{-}^{w}]$ of a unipotent cell $N_{-}^{w}$ in non-quantum
settings. Here the image of $\varphi_{X}\in\mathbb{C}\left[N_{-}\right]$
under the restriction map $\mathbb{C}\left[N_{-}\right]\to\mathbb{C}\left[N_{-}^{w}\right]$
is denoted by $\left[\varphi_{X}\right]\in\mathbb{C}\left[N_{-}^{w}\right]$.
\begin{prop}[{{\cite[Theorem 6]{MR2833478}}}]
\label{p:GLStwist} Let $w\in W$. Then for every $X\in\mathcal{C}_{w}$
we have 
\[
\eta_{w}^{\ast}\left(\left[\varphi_{X}\right]\right)=\frac{\left[\varphi_{\Omega_{w}^{-1}(X)}\right]}{\left[\varphi_{I(X)}\right]}.
\]
\end{prop}

\subsection{Quantum twist automorphisms and the quantum algebra structures}

Our main result in this subsection is the following quantum analogue
of Proposition \ref{p:GLStwist}. Recall Proposition \ref{p:syzygy}. 
\begin{thm}
\label{t:qGLS} Let $w\in W$, $T$ a basic reachable $\mathcal{C}_{w}$-maximal
rigid module, and $T=T_{1}\oplus\cdots\oplus T_{\ell}$ its indecomposable
decomposition. Through $\widetilde{\kappa}$ in Corollary \ref{c:GLSqclus},
we regard the quantum twist map $\eta_{w,q}$ as an algebra automorphism
on $\widetilde{\mathscr{A}}_{\mathbb{Q}(q)}(\mathcal{C}_{w})$. Then,
for every direct summand $R$ of a direct sum of copies of $T$ (that
is, every reachable $\Pi$-module $R$ in $\mathcal{C}_{w}$), we
have 
\begin{equation}
\eta_{w,q}(Y_{R})=q^{\sum_{i\in I_{w}}\lambda_{i}\dim_{\mathbb{C}}\epsilon_{i}.R}Y_{I(R)}^{-1}Y_{\Omega_{w}^{-1}(R)}.\label{qGLS}
\end{equation}
here we write $I(R)=\bigoplus_{i\in I_{w}}T_{\ell-n+i}^{\oplus\lambda_{i}}$. 
\end{thm}
As a corollary of the above result, we obtain the following.
\begin{cor}
\label{c:clustermonom} Let $R$ be a reachable $\Pi$-module in $\mathcal{C}_{w}$.
Then $\kappa^{-1}(Y_{R})\in\mathbf{B}^{\mathrm{up}}\cap\Aq[N_{-}(w)]$
if and only if $\kappa^{-1}(Y_{\Omega_{w}^{-1}(R)})\in\mathbf{B}^{\mathrm{up}}\cap\Aq[N_{-}(w)]$.
\end{cor}
Before proving Theorem \ref{t:qGLS}, we show its corollary.
\begin{proof}
By Theorem \ref{t:BZauto} and \ref{t:qGLS}, $\kappa^{-1}(Y_{R})\in\mathbf{B}^{\mathrm{up}}$
if and only if $\widetilde{\kappa}^{-1}(q^{\sum_{i\in I_{w}}\lambda_{i}\dim_{\mathbb{C}}\epsilon_{i}.R}Y_{I(R)}^{-1}Y_{\Omega_{w}^{-1}(R)})\in\widetilde{\mathbf{B}}^{\mathrm{up},w}$.
By Theorem \ref{t:BZauto} and the dual bar invariance of $Y_{R}$,
the element $q^{\sum_{i\in I_{w}}\lambda_{i}\dim_{\mathbb{C}}\epsilon_{i}.R}Y_{I(R)}^{-1}Y_{\Omega_{w}^{-1}(R)}$
is also dual bar invariant. Combining this fact with the definition
of $\widetilde{\mathbf{B}}^{\mathrm{up},w}=\iota_{w}(\mathbf{\widetilde{B}}^{\mathrm{up}}(w))$
and the dual bar invariance of $Y_{\Omega_{w}^{-1}(R)}$, we have
$\widetilde{\kappa}^{-1}(q^{\sum_{i\in I_{w}}\lambda_{i}\dim_{\mathbb{C}}\epsilon_{i}.R}Y_{I(R)}^{-1}Y_{\Omega_{w}^{-1}(R)})\in\widetilde{\mathbf{B}}^{\mathrm{up},w}$
if and only if $\kappa^{-1}(Y_{\Omega_{w}^{-1}(R)})\in\mathbf{B}^{\mathrm{up}}$. 
\end{proof}
\begin{rem}
\label{r:KKKO} Kang-Kashiwara-Kim-Oh \cite{KKKO} have shown
that all (rescaled) quantum cluster monomials belong to $\mathbf{B}^{\mathrm{up}}$
by using the categorification via representations of quiver Hecke
algebras. (See \cite[Introduction]{KKKO} for several results of this direction before \cite{KKKO}.) Hence we have already known that $Y_{R}$ is an element
of $\mathbf{B}^{\mathrm{up}}$ for an arbitrary reachable $\Pi$-module
in $\mathcal{C}_{w}$. However there is now no proof of this strong
result through the additive categorification above. Therefore it would
be interesting to determine the quantum monomials in $\mathbf{B}^{\mathrm{up}}$
which are obtained from Corollary \ref{c:clustermonom} and, for example,
$(Y_{V_{\bm{i}}})^{\bm{a}}$ for $\bm{a}\in\mathbb{Z}_{\geq0}^{\ell(w)}$
and $\bm{i}\in I(w)$. Actually, it is easy to show that $(Y_{V_{\bm{i}}})^{\bm{a}}\in\mathbf{B}^{\mathrm{up}}$
by Proposition \ref{p:localization}. Moreover it is unclear whether
a quantum twist automorphism $\eta_{w,q}$ is categorified by using
finite dimensional representations of quiver Hecke algebras. In particular,
we do not know that quantum twist automorphisms preserve the basis
coming from the simple modules of quiver Hecke algebras. 
\end{rem}
The rest of this subsection is devoted to the proof of Theorem \ref{t:qGLS}.
In this proof, we essentially use Gei\ss-Leclerc-Schr\"oer's theory. 
\begin{lem}
\label{l:dualbarinv} Let $T$ be a basic reachable $\mathcal{C}_{w}$-maximal
rigid module and $T=T_{1}\oplus\cdots\oplus T_{\ell}$ its indecomposable
decomposition. Take $(a_{1},\dots,a_{\ell})\in\mathbb{Z}^{\ell}$.
Then there exists a unique integer $m$ such that $q^{m}Y_{T_{1}}^{a_{1}}\cdots Y_{T_{\ell}}^{a_{\ell}}$
is dual bar invariant in $\mathcal{T}_{\mathcal{A},T}$. 
\end{lem}
\begin{proof}
We have 
\begin{align*}
\sigma'_{T}(q^{m}Y_{T_{1}}^{a_{1}}\cdots Y_{T_{\ell}}^{a_{\ell}}) & =q^{-m}\sigma'_{T}(Y_{T_{\ell}})^{a_{\ell}}\cdots\sigma'_{T}(Y_{T_{1}})^{a_{1}}\\
 & =q^{-m+\sum_{i\in[1,\ell]}a_{i}(-(\dimv T_{i},\dimv T_{i})/2+(\dimv T_{i},\rho))}Y_{T_{\ell}}^{a_{\ell}}\cdots Y_{T_{1}}^{a_{1}}\\
 & =q^{-m+\sum_{i\in[1,\ell]}a_{i}(-(\dimv T_{i},\dimv T_{i})/2+(\dimv T_{i},\rho))-\sum_{i<j}a_{i}a_{j}\lambda_{ij}}Y_{T_{1}}^{a_{1}}\cdots Y_{T_{\ell}}^{a_{\ell}}.
\end{align*}
Here we write $\Lambda_{T}=(\lambda_{ij})_{i,j\in[1,\ell]}$. Therefore
$q^{m}Y_{T_{1}}^{a_{1}}\cdots Y_{T_{\ell}}^{a_{\ell}}$ is dual bar
invariant if and only if 
\begin{align*}
 & m-\sum_{i\in[1,\ell]}a_{i}^{2}(\dimv T_{i},\dimv T_{i})/2-\sum_{i<j}a_{i}a_{j}(\dimv T_{i},\dimv T_{j})+\sum_{i\in[1,\ell]}a_{i}(\dimv T_{i},\rho)\\
 & =-m+\sum_{i\in[1,\ell]}a_{i}(-(\dimv T_{i},\dimv T_{i})/2+(\dimv T_{i},\rho))-\sum_{i<j}a_{i}a_{j}\lambda_{ij}.
\end{align*}
By Remark \ref{r:CB}, this is equivalent to 
\[
2m=\sum_{i\in[1,\ell]}a_{i}(a_{i}-1)(\dimv T_{i},\dimv T_{i})/2+2\sum_{i<j}a_{i}a_{j}\dim_{\mathbb{C}}\Hom_{\Pi}(T_{j},T_{i}).
\]
The right-hand side is an element of $2\mathbb{Z}$. Therefore we
can take an integer $m\in\mathbb{Z}$ uniquely which satisfies this
equality. 
\end{proof}
\begin{rem}
For $(a_{1},\dots,a_{\ell})\in\mathbb{Z}_{\geq0}^{\ell}$, the unique
dual bar invariant element in $\{q^{m}Y_{T_{1}}^{a_{1}}\cdots Y_{T_{\ell}}^{a_{\ell}}\mid m\in\mathbb{Z}\}$
is $Y_{\bigoplus_{i\in[1,\ell]}T_{i}^{\oplus a_{i}}}$. 
\end{rem}
\begin{lem}
\label{l:dualbar} With the notation in Theorem \ref{t:qGLS}, $q^{\sum_{i\in I_{w}}\lambda_{i}\dim_{\mathbb{C}}\epsilon_{i}.R}Y_{I(R)}^{-1}Y_{\Omega_{w}^{-1}(R)}$
is dual bar invariant. 
\end{lem}
\begin{proof}
By Proposition \ref{p:GLSqclus}, 
\[
\kappa^{-1}(Y_{I(R)})=D_{w\lambda,\lambda},
\]
here $\lambda:=\sum_{i\in I_{w}}\lambda_{i}\varpi_{i}$. Hence, by
Proposition \ref{p:comm}, we have 
\begin{align*}
\kappa^{-1}(Y_{I(R)}Y_{\Omega_{w}^{-1}(R)}) & =D_{w\lambda,\lambda}\kappa^{-1}(Y_{\Omega_{w}^{-1}(R)})\\
 & =q^{(\lambda+w\lambda,\dimv\Omega_{w}^{-1}(R))}\kappa^{-1}(Y_{\Omega_{w}^{-1}(R)})D_{w\lambda,\lambda}\\
 & =q^{(\lambda+w\lambda,\dimv\Omega_{w}^{-1}(R))}\kappa^{-1}(Y_{\Omega_{w}^{-1}(R)}Y_{I(R)}).
\end{align*}
By the way, $\dimv\Omega_{w}^{-1}(R)=\dimv I(R)-\dimv R=w\lambda-\lambda-\dimv R$.
Hence $(\lambda+w\lambda,\dimv\Omega_{w}^{-1}(R))=-(\lambda+w\lambda,\dimv R)$.
Therefore 
\[
Y_{I(R)}^{-1}Y_{\Omega_{w}^{-1}(R)}=q^{(\lambda+w\lambda,\dimv R)}Y_{\Omega_{w}^{-1}(R)}Y_{I(R)}^{-1}
\]
Note that $\sum_{i\in I_{w}}\lambda_{i}\dim_{\mathbb{C}}\epsilon_{i}.R=-(\lambda,\dimv R)$.
We have 
\begin{align*}
 & q^{(\dimv\Omega_{w}^{-1}(R)-\dimv I(R),\dimv\Omega_{w}^{-1}(R)-\dimv I(R))/2-(\dimv\Omega_{w}^{-1}(R)-\dimv I(R),\rho)}\sigma'_{T}(q^{-(\lambda,\dimv R)}Y_{I(R)}^{-1}Y_{\Omega_{w}^{-1}(R)})\\
 & =q^{(\dimv R,\dimv R)/2+(\dimv R,\rho)}\sigma'_{T}(q^{(w\lambda,\dimv R)}Y_{\Omega_{w}^{-1}(R)}Y_{I(R)}^{-1})\\
 & =q^{(\dimv R,\dimv R)/2+(\dimv R,\rho)-(w\lambda,\dimv R)}\sigma'_{T}(Y_{I(R)}^{-1})\sigma'_{T}(Y_{\Omega_{w}^{-1}(R)})\\
 & =q^{(\dimv R,\dimv R)/2-(\dimv\Omega_{w}^{-1}(R),\dimv\Omega_{w}^{-1}(R))/2+(\dimv I(R),\dimv I(R))/2-(w\lambda,\dimv R)}Y_{I(R)}^{-1}Y_{\Omega_{w}^{-1}(R)}\\
 & =q^{(\dimv I(R),\dimv R)-(w\lambda,\dimv R)}Y_{I(R)}^{-1}Y_{\Omega_{w}^{-1}(R)}\\
 & =q^{-(\lambda,\dimv R)}Y_{I(R)}^{-1}Y_{\Omega_{w}^{-1}(R)}.
\end{align*}
This competes the proof. 
\end{proof}
\begin{lem}
\label{l:indec} Let $T$ be a basic reachable $\mathcal{C}_{w}$-maximal
rigid module and $T=T_{1}\oplus\cdots\oplus T_{\ell}$ its indecomposable
decomposition. Then the equality (\ref{qGLS}) with $R=T_{k}$ holds
for all $k=1,\dots,\ell$ if and only if the one with $R=T_{1}^{\oplus a_{1}}\oplus\cdots\oplus T_{\ell}^{\oplus a_{\ell}}$
holds for all $(a_{1},\dots,a_{\ell})\in\mathbb{Z}_{\geq0}^{\ell}$. 
\end{lem}
\begin{proof}
The latter obviously implies the former. Suppose that the equality
(\ref{qGLS}) holds for $R=T_{k}$, $k=1,\dots,\ell$. Write 
\[
\eta_{w,q}(Y_{T_{k}})=q^{m_{k}}Y_{I(T_{k})}^{-1}Y_{\Omega_{w}^{-1}(T_{k})},\ m_{k}\in\mathbb{Z},
\]
for $k=1,\dots,\ell$. Set $R=T_{1}^{\oplus a_{1}}\oplus\cdots\oplus T_{\ell}^{\oplus a_{\ell}}$
for $(a_{1},\dots,a_{\ell})\in\mathbb{Z}_{\geq0}^{\ell}$. Note that
$I(R)=I(T_{1})^{\oplus a_{1}}\oplus\cdots\oplus I(T_{\ell})^{\oplus a_{\ell}}$
and $\Omega_{w}^{-1}(R)=\Omega_{w}^{-1}(T_{1})^{\oplus a_{1}}\oplus\cdots\oplus\Omega_{w}^{-1}(T_{\ell})^{\oplus a_{\ell}}$.
(Actually $I(T_{\ell-n+i})=T_{\ell-n+i}$ and $\Omega_{w}^{-1}(T_{\ell-n+i})=0$
for $i\in I_{w}$.) There exist unique $A_{1},A_{2},A_{3}\in\mathbb{Z}$
such that the following hold: 
\begin{align*}
\eta_{w,q}(Y_{R}) & =q^{A_{1}}\eta_{w,q}(Y_{T_{1}}^{a_{1}}\cdots Y_{T_{\ell}}^{a_{\ell}})\\
 & =q^{A_{1}}(q^{m_{1}}Y_{I(T_{1})}^{-1}Y_{\Omega_{w}^{-1}(T_{1})})^{a_{1}}\cdots(q^{m_{\ell}}Y_{I(T_{\ell})}^{-1}Y_{\Omega_{w}^{-1}(T_{\ell})})^{a_{\ell}}\\
 & =q^{A_{2}}(Y_{I(T_{1})}^{a_{1}}\cdots Y_{I(T_{\ell})}^{a_{\ell}})^{-1}Y_{\Omega_{w}^{-1}(T_{1})}^{a_{1}}\cdots Y_{\Omega_{w}^{-1}(T_{\ell})}^{a_{\ell}}\\
 & =q^{A_{3}}Y_{I(R)}^{-1}Y_{\Omega_{w}^{-1}(R)}.
\end{align*}
Moreover $\eta_{w,q}(Y_{R})$ is dual bar invariant because of the
dual bar invariance of $Y_{R}$ and Theorem \ref{t:BZauto}. Hence,
by Lemma \ref{l:dualbarinv} and Lemma \ref{l:dualbar}, the equality
(\ref{qGLS}) also holds for $R$. 
\end{proof}
\begin{proof}[{{Proof of Theorem \ref{t:qGLS}}}]
Recall that we always assume that $T_{\ell-n+i}$ is a $\mathcal{C}_{w}$-projective-injective
module with socle $S_{i}$ for all $i\in I_{w}=[1,n]$, in particular,
the isomorphism class of $T_{\ell-n+i}$ does not depend on the choice
of $T$. From now on, we identify $\widetilde{\mathscr{A}}_{\mathbb{Q}(q)}(\mathcal{C}_{w})$
with $\Aq[N_{-}^{w}]$ via $\widetilde{\kappa}$. First we consider
the case that $R$ in the statement of Theorem \ref{t:qGLS} is equal
to $T_{\ell-n+i}$ for $i\in I_{w}$. Then 
\begin{align*}
\eta_{w,q}(Y_{T_{\ell-n+i}}) & =\eta_{w,q}([D_{w\varpi_{i},\varpi_{i}}])\\
 & =q^{-(\varpi_{i},w\varpi_{i}-\varpi_{i})}[D_{w\varpi_{i},\varpi_{i}}]^{-1}\\
 & =q^{\dim_{\mathbb{C}}\epsilon_{i}.T_{\ell-n+i}}Y_{T_{\ell-n+i}}^{-1},
\end{align*}
which is the desired equality in this case since $I(T_{\ell-n+i})=T_{\ell-n+i}$
and $\Omega_{w}^{-1}(T_{\ell-n+i})=0$. 

Let $\bm{i}\in I(w)$. Henceforth, we will prove the theorem by induction
on the minimal length of sequences of mutations which we need to obtain
$T$ from $V_{\bm{i}}$. We begin with the case that $R=V_{\bm{i},k}$
for some $k\in[1,\ell]$ with $k^{+}\neq\ell+1$. Then $I(V_{\bm{i},k})=V_{\bm{i},k^{\mathrm{max}}}$
and $\Omega_{w}^{-1}(V_{\bm{i},k})=M_{\bm{i}}[k^{\mathrm{max}},k^{+}]$.
Therefore we have 
\begin{align*}
\eta_{w,q}(Y_{V_{\bm{i},k}}) & =\eta_{w,q}([D_{s_{i_{1}}\cdots s_{i_{k}}\varpi_{i_{k}},\varpi_{i_{k}}}])\\
 & =q^{-(\varpi_{i_{k}},s_{i_{1}}\cdots s_{i_{k}}\varpi_{i_{k}}-\varpi_{i_{k}})}[D_{w\varpi_{i_{k}},\varpi_{i_{k}}}]^{-1}[D_{u_{w\varpi_{i_{k}}},s_{i_{1}}\cdots s_{i_{k}}\varpi_{i_{k}}}]\\
 & =q^{-(\varpi_{i_{k}},\dimv V_{\bm{i},k})}Y_{V_{\bm{i},k^{\mathrm{max}}}}^{-1}Y_{M_{\bm{i}}[k^{\mathrm{max}},k^{+}]}\\
 & =q^{\dim_{\mathbb{C}}\epsilon_{i}.V_{\bm{i},k}}Y_{I(V_{\bm{i},k})}^{-1}Y_{\Omega_{w}^{-1}(V_{\bm{i},k})}.
\end{align*}
Hence, by Lemma \ref{l:indec}, the equality (\ref{qGLS}) holds when
$R=V_{\bm{i}}$. Next, suppose that the equality (\ref{qGLS}) holds
for $R=T_{1}^{\oplus a_{1}}\oplus\cdots\oplus T_{\ell}^{\oplus a_{\ell}}$,
where $T=T_{1}\oplus\cdots\oplus T_{\ell}$ is a basic reachable $\mathcal{C}_{w}$-maximal
rigid module. Let $k\in[1,\ell-n]$ and write $\mu_{T_{k}}(T)=(T/T_{k})\oplus T_{k}^{\ast}$,
$I(T_{k}^{\ast})=\bigoplus_{i\in I_{w}}T_{\ell-n+i}^{\oplus\lambda_{i}}$.
Then, by Lemma \ref{l:indec}, it remains to prove the following equality:
\begin{align}
\eta_{w,q}(Y_{T_{k}^{\ast}})=q^{\sum_{i\in I_{w}}\lambda_{i}\dim_{\mathbb{C}}\epsilon_{i}.T_{k}^{\ast}}Y_{I(T_{k}^{\ast})}^{-1}Y_{\Omega_{w}^{-1}(T_{k}^{\ast})}.\label{aimeq}
\end{align}
Write $\tilde{B}_{T}=(b_{ij})_{i\in[1,\ell],j\in[1,\ell-n]}$. Set
$T_{+}:=\bigoplus_{j;b_{jk}>0}T_{j}^{\oplus b_{jk}}$ and $T_{-}:=\bigoplus_{j;b_{jk}<0}T_{j}^{\oplus(-b_{jk})}$.
By (\ref{clmut}), Proposition \ref{p:phimap} (2) and Proposition
\ref{p:GLStwist}, we have 
\begin{align*}
\eta_{w}^{\ast}([\varphi_{T_{k}}\varphi_{T_{k}^{\ast}}])=\eta_{w}^{\ast}([\varphi_{T_{+}}]+[\varphi_{T_{-}}])=\frac{[\varphi_{\Omega_{w}^{-1}(T_{+})}]}{[\varphi_{I(T_{+})}]}+\frac{[\varphi_{\Omega_{w}^{-1}(T_{-})}]}{[\varphi_{I(T_{-})}]},
\end{align*}
and 
\begin{align*}
\eta_{w}^{\ast}([\varphi_{T_{k}}\varphi_{T_{k}^{\ast}}]) & =\frac{[\varphi_{\Omega_{w}^{-1}(T_{k})}]}{[\varphi_{I(T_{k})}]}\cdot\frac{[\varphi_{\Omega_{w}^{-1}(T_{k}^{\ast})}]}{[\varphi_{I(T_{k}^{\ast})}]}=\frac{[\varphi_{\Omega_{w}^{-1}(T_{k})}\varphi_{\Omega_{w}^{-1}(T_{k}^{\ast})}]}{[\varphi_{I(T_{k}\oplus T_{k}^{\ast})}]}.
\end{align*}
Therefore , 
\begin{align}
[\varphi_{\Omega_{w}^{-1}(T_{k})}\varphi_{\Omega_{w}^{-1}(T_{k}^{\ast})}]=[\varphi_{I(T_{k}\oplus T_{k}^{\ast})}]\left(\frac{[\varphi_{\Omega_{w}^{-1}(T_{+})}]}{[\varphi_{I(T_{+})}]}+\frac{[\varphi_{\Omega_{w}^{-1}(T_{-})}]}{[\varphi_{I(T_{-})}]}\right).\label{classical}
\end{align}
By Proposition \ref{p:syzygy}, $T':=\Omega_{w}^{-1}(T)\oplus\bigoplus_{i\in I_{w}}T_{\ell-n+i}$
is a basic reachable $\mathcal{C}_{w}$-maximal rigid module; hence
there exists a bijection $[1,\ell-n]\to[1,\ell-n],j\mapsto j^{\ast}$
such that $T'_{j^{\ast}}=\Omega_{w}^{-1}(T_{j})$. Moreover we have
\[
\mu_{T'_{k^{\ast}}}(T')=(T'/T'_{k^{\ast}})\oplus\Omega_{w}^{-1}(T_{k}^{\ast}).
\]
Write $\tilde{B}_{T'}=(b'_{ij})_{i\in[1,\ell],j\in[1,\ell-n]}$ and
$(T'_{k^{\ast}})^{\ast}:=\Omega_{w}^{-1}(T_{k}^{\ast})$. Set $T'_{+}:=\bigoplus_{j;b'_{j^{\ast}k^{\ast}}>0}(T'_{j^{\ast}})^{\oplus b'_{j^{\ast}k^{\ast}}}$
and $T'_{-}:=\bigoplus_{j;b'_{j^{\ast}k^{\ast}}<0}(T'_{j^{\ast}})^{\oplus(-b'_{j^{\ast}k^{\ast}})}$.
Then, by (\ref{clmut}) and (\ref{classical}), we have 
\begin{align}
[\varphi_{I(T_{k}\oplus T_{k}^{\ast})}]\left(\frac{[\varphi_{\Omega_{w}^{-1}(T_{+})}]}{[\varphi_{I(T_{+})}]}+\frac{[\varphi_{\Omega_{w}^{-1}(T_{-})}]}{[\varphi_{I(T_{-})}]}\right)=[\varphi_{T'_{+}}]+[\varphi_{T'_{-}}].\label{classical'}
\end{align}

Note that all $\Pi$-modules appearing in the equality (\ref{classical'})
are direct summands of a direct sum of copies of $T'$. Therefore,
by Proposition \ref{p:phimap} (2), we have $I(T_{k}\oplus T_{k}^{\ast})=I(T_{+})\oplus I_{+}=I(T_{-})\oplus I_{-}$
for some $\mathcal{C}_{w}$-projective-injective modules $I_{+},I_{-}$,
and 
\begin{equation}
\{I_{+}\oplus\Omega_{w}^{-1}(T_{+}),I_{-}\oplus\Omega_{w}^{-1}(T_{-})\}=\{T'_{+},T'_{-}\}\label{eq:modeq}
\end{equation}

By the way, we recall our assumption that the equality (\ref{qGLS})
holds for $R=T_{1}^{\oplus a_{1}}\oplus\cdots\oplus T_{\ell}^{\oplus a_{\ell}}$.
By Proposition \ref{p:qmutation} and our assumption, there exist
unique $A_{1},A'_{1},A_{2},A'_{2},A_{3}\in\mathbb{Z}$ such that 
\begin{align*}
\eta_{w,q}(Y_{T_{k}}Y_{T_{k}^{\ast}}) & =\eta_{w,q}(q^{A_{1}}Y_{T_{+}}+q^{A_{2}}Y_{T_{-}})\\
 & =q^{A'_{1}}Y_{I(T_{+})}^{-1}Y_{\Omega_{w}^{-1}(T_{+})}+q^{A'_{2}}Y_{I(T_{-})}^{-1}Y_{\Omega_{w}^{-1}(T_{-})},
\end{align*}
and 
\begin{align*}
\eta_{w,q}(Y_{T_{k}}Y_{T_{k}^{\ast}}) & =q^{A_{3}}Y_{I(T_{k})}^{-1}Y_{T'_{k^{\ast}}}\eta_{w,q}(Y_{T_{k}^{\ast}}).
\end{align*}
These equalities together with (\ref{eq:modeq}) imply that there
exist unique $A',A''_{1},A''_{2}\in\mathbb{Z}$ such that 
\begin{align*}
\eta_{w,q}(Y_{T_{k}^{\ast}}) & =q^{-A_{3}}Y_{T'_{k^{\ast}}}^{-1}Y_{I(T_{k})}(q^{A'_{1}}Y_{I(T_{+})}^{-1}Y_{\Omega_{w}^{-1}(T_{+})}+q^{A'_{2}}Y_{I(T_{-})}^{-1}Y_{\Omega_{w}^{-1}(T_{-})})\\
 & =q^{A'}Y_{T'_{k^{\ast}}}^{-1}Y_{I(T_{k}^{\ast})}^{-1}Y_{I(T_{k}\oplus T_{k}^{\ast})}(q^{A'_{1}}Y_{I(T_{+})}^{-1}Y_{\Omega_{w}^{-1}(T_{+})}+q^{A'_{2}}Y_{I(T_{-})}^{-1}Y_{\Omega_{w}^{-1}(T_{-})})\\
 & =Y_{I(T_{k}^{\ast})}^{-1}Y_{T'_{k^{\ast}}}^{-1}\left(q^{A''_{1}}Y_{T'_{+}}+q^{A''_{2}}Y_{T'_{-}}\right).
\end{align*}
Note that all rescaled quantum cluster monomials appearing in the
rightmost hand side of the equality above are monomials of the based
quantum torus $\mathcal{T}_{\mathcal{A},T'}$. Moreover, $\eta_{w,q}(Y_{T_{k}^{\ast}})$
is dual bar invariant because, by Theorem \ref{t:BZauto}, the quantum
twist automorphism $\eta_{w,q}$ preserves dual bar invariance property
of elements of $\widetilde{\mathscr{A}}_{\mathbb{Q}(q)}(\mathcal{C}_{w})$
(recall Definition \ref{d:GLSdualbar}). Hence $q^{A''_{1}}Y_{I(T_{k}^{\ast})}^{-1}Y_{T'_{k^{\ast}}}^{-1}Y_{T'_{+}}$
and $q^{A''_{2}}Y_{I(T_{k}^{\ast})}^{-1}Y_{T'_{k^{\ast}}}^{-1}Y_{T'_{-}}$
are dual bar invariant elements of $\mathcal{T}_{\mathcal{A},T'}$.
By Lemma \ref{l:dualbarinv}, $A''_{1}$ and $A''_{2}$ are uniquely
determined by this property. On the other hand, by Proposition \ref{p:qmutation},
$q^{\sum_{i\in I_{w}}\lambda_{i}\dim_{\mathbb{C}}\epsilon_{i}.T_{k}^{\ast}}Y_{I(T_{k}^{\ast})}^{-1}Y_{(T'_{k^{\ast}})^{\ast}}$
is of the following form as an element of $\mathcal{T}_{q^{\pm1},T'}$:
\[
Y_{I(T_{k}^{\ast})}^{-1}Y_{T'_{k^{\ast}}}^{-1}\left(q^{M_{1}}Y_{T'_{+}}+q^{M_{2}}Y_{T'_{-}}\right),\ M_{1},M_{2}\in\mathbb{Z}.
\]
Moreover, by Lemma \ref{l:dualbar}, $q^{\sum_{i\in I_{w}}\lambda_{i}\dim_{\mathbb{C}}\epsilon_{i}.T_{k}^{\ast}}Y_{I(T_{k}^{\ast})}^{-1}Y_{(T'_{k^{\ast}})^{\ast}}=q^{\sum_{i\in I_{w}}\lambda_{i}\dim_{\mathbb{C}}\epsilon_{i}.T_{k}^{\ast}}Y_{I(T_{k}^{\ast})}^{-1}Y_{\Omega_{w}^{-1}(T_{k}^{\ast})}$
is dual bar invariant. Hence, by the argument above, $M_{1}=A'_{1}$
and $M_{2}=A''_{2}$. Therefore we obtain 
\[
\eta_{w,q}(Y_{T_{k}^{\ast}})=q^{\sum_{i\in I_{w}}\lambda_{i}\dim_{\mathbb{C}}\epsilon_{i}.T_{k}^{\ast}}Y_{I(T_{k}^{\ast})}^{-1}Y_{\Omega_{w}^{-1}(T_{k}^{\ast})},
\]
 which completes the proof.
\end{proof}

\section{Finite type cases : 6-periodicity\label{sec:fin}}

Since the map $\eta_{q,w}$ is an automorphism, we can apply it repeatedly.
In this section, we show the ``$6$-periodicity'' of specific quantum
twist automorphisms. Assume that $\mathfrak{g}$ is a finite dimensional
Lie algebra, and let $w_{0}$ be the longest element of $W$. 
\begin{thm}
\label{t:period} For a homogeneous element $x\in\Aq[N_{-}^{w_{0}}]$,
we have 
\[
\eta_{w_{0},q}^{6}(x)=q^{(\wt x+w_{0}\wt x,\wt x)}D_{w_{0},-\wt x-w_{0}\wt x}x.
\]
\end{thm}
\begin{rem}
When the action of $w_{0}$ on $P$ is given by $\mu\mapsto-\mu$, the theorem above states that $\eta_{w_{0},q}^{6}=\mathrm{id}$. Hence
$\eta_{w_{0},q}$ is ``really'' periodic. If $\mathfrak{g}$ is
simple, then this condition is satisfied in the case that
$\mathfrak{g}$ is of type $\mathrm{B}_{n}$, $\mathrm{C}_{n}$, $\mathrm{D}_{2n}$
for $n\in\mathbb{Z}_{>0}$ and $\mathrm{E}_{7}$, $\mathrm{E}_{8}$,
$\mathrm{F}_{4}$, $\mathrm{G}_{2}$. See \cite[Section 3.7]{MR1066460}.

When $\mathfrak{g}$ is symmetric, such periodicity is also explained
by Gei\ss-Leclerc-Schr\"oer's additive categorification of twist automorphisms
(see Section \ref{s:GLS}). The periodicity corresponds to the well-known
$6$-periodicity of syzygy functors \cite{MR1388043,MR1648626}, that
is, the property that $(\Omega_{w_{0}}^{-1})^{6}(M)\simeq M$ for
an indecomposable non-projective-injective module $M$ of $\Pi$ in
the notation of Section \ref{s:GLS}.

We can consider the similar periodicity problems for every $w\in W$.
It would be interesting to find the necessary and sufficient condition
on $w\in W$ for periodicity. Since quantum twist automorphisms restrict
to permutations on dual canonical bases, the periodicity of a quantum
twist automorphism $\eta_{w,q}$ is equivalent to the periodicity
of a (non-quantum) twist automorphism $\eta_{w}$. See also Remark \ref{r:decrease} below. 
\end{rem}
\begin{lem}
\label{l:Dcrystal} Let $\lambda\in P_{+}$. Take $u,u'\in V(\lambda)$
such that $D_{u,u'}=\Gup(\tilde{b})$ for some $\tilde{b}\in\mathscr{B}(\infty)$.
Then, for $i\in I$, 
\begin{align*}
\varepsilon_{i}(\tilde{b}) & =\max\{k\in\mathbb{Z}_{\geq0}\mid D_{e_{i}^{k}.u,u'}\neq0\} & \varepsilon_{i}^{\ast}(\tilde{b}) & =\max\{k\in\mathbb{Z}_{\geq0}\mid D_{u,f_{i}^{k}.u'}\neq0\}.
\end{align*}
In particular, 
\begin{align*}
\varepsilon_{i}(\overline{\jmath}_{\lambda}(b)) & =\varepsilon_{i}(b) & \varepsilon_{i}(\overline{\jmath}_{w_{0}\lambda}^{\vee}(b)) & =\varphi_{i}(b)(=\varepsilon_{i}(b)+\langle h_{i},\wt b\rangle).
\end{align*}
\end{lem}
\begin{proof}
By Proposition \ref{p:EFaction}, 
\begin{align}
\varepsilon_{i}(\tilde{b})=\max\{k\in\mathbb{Z}_{\geq0}\mid(e_{i}')^{k}(D_{u,u'})\neq0\}.\label{varepsilon}
\end{align}
For $k\in\mathbb{Z}_{\geq0}$ and $x\in\Uq^{-}$, we have 
\begin{align*}
((e_{i}')^{k}(D_{u,u'}),x)_{L} & =(1-q_{i}^{2})^{k}(D_{u,u'},f_{i}^{k}x)_{L}\\
 & =(1-q_{i}^{2})^{k}(u,f_{i}^{k}x.u')_{\lambda}^{\varphi}\\
 & =(1-q_{i}^{2})^{k}(e_{i}^{k}.u,x.u')_{\lambda}^{\varphi}=(1-q_{i}^{2})^{k}(D_{e_{i}^{k}.u,u'},x)_{L}.
\end{align*}
Hence $(e_{i}')^{k}(D_{u,u'})=(1-q_{i}^{2})^{k}D_{e_{i}^{k}.u,u'}$.
Combining this equality with (\ref{varepsilon}), we obtain the first
equality. The second equality is proved in the same manner. The last
two equalities are deduced from Proposition \ref{p:EFaction} and
\ref{p:minor}. 
\end{proof}
\begin{proof}[{{Proof of Theorem \ref{t:period}}}]
 It is easily seen that we need only check the case that $x\in\Uq^{-}$.
For $i\in I$, we have $D_{s_{i}\varpi_{i},\varpi_{i}}=(1-q_{i}^{2})f_{i}$.
We first consider the images of $D_{s_{i}\varpi_{i},\varpi_{i}}$,
$i\in I$ under iterated application of $\eta_{w_{0},q}$. If $I=\{i\}$,
that is, $\mathfrak{g}=\mathfrak{sl}_{2}$, the quantum unipotent
cell $\Aq[N_{-}^{w_{0}}]$ is generated by $D_{s_{i}\varpi_{i},\varpi_{i}}^{\pm1}(=D_{w_{0}\varpi_{i},\varpi_{i}}^{\pm1})$.
In this case, $\eta_{w_{0},q}^{2}(D_{s_{i}\varpi_{i},\varpi_{i}})=D_{s_{i}\varpi_{i},\varpi_{i}}$.
Hence $\eta_{w_{0},q}^{2}=\mathrm{id}$, in particular, the theorem
holds. Henceforth, we consider the case that $\mathfrak{g}$ does
not have ideals of Lie algebras which are isomorphic to $\mathfrak{sl}_{2}$.
We have 
\[
\eta_{w_{0},q}(D_{s_{i}\varpi_{i},\varpi_{i}})\simeq D_{w_{0}\varpi_{i},\varpi_{i}}^{-1}D_{w_{0}\varpi_{i},s_{i}\varpi_{i}}.
\]
Here $\simeq$ stands for the coincidence up to some powers of $q$.
Now, by Proposition \ref{p:minor}, $D_{w_{0}\varpi_{i},s_{i}\varpi_{i}}=\Gup(\ast\overline{\jmath}_{w_{0}\varpi_{i}}^{\vee}(u_{s_{i}\varpi_{i}}))$.
By Lemma \ref{l:Dcrystal}, 
\[
\varepsilon_{j}^{\ast}(\ast\overline{\jmath}_{w_{0}\varpi_{i}}^{\vee}(u_{s_{i}\varpi_{i}}))=\varepsilon_{j}(\overline{\jmath}_{w_{0}\varpi_{i}}^{\vee}(u_{s_{i}\varpi_{i}}))=\varphi_{j}(u_{s_{i}\varpi_{i}})=\begin{cases}
-a_{ji} & \text{if}\;j\neq i,\\
0 & \text{if}\;j=i.
\end{cases}
\]
Therefore $\sum_{j\in I}\varepsilon_{j}^{\ast}(\ast\overline{\jmath}_{w_{0}\varpi_{i}}^{\vee}(u_{s_{i}\varpi_{i}}))\varpi_{j}=\varpi_{i}+s_{i}\varpi_{i}(=:\lambda_{1})$.
Recall Remark \ref{r:minlambda}. Then there exists $b_{1}\in\mathscr{B}(\lambda_{1})$
such that $D_{w_{0}\varpi_{i},s_{i}\varpi_{i}}=D_{\Gup_{\lambda_{1}}(b_{1}),u_{\lambda_{1}}}$,
that is, $\overline{\jmath}_{\lambda_{1}}(b_{1})=\ast\overline{\jmath}_{w_{0}\varpi_{i}}^{\vee}(u_{s_{i}\varpi_{i}})$.
Then 
\[
\eta_{w_{0},q}^{2}(D_{s_{i}\varpi_{i},\varpi_{i}})\simeq D_{w_{0}\varpi_{i},\varpi_{i}}D_{w_{0}\lambda_{1},\lambda_{1}}^{-1}D_{u_{w_{0}\lambda_{1}},\Gup(b_{1})}.
\]
As above, $D_{w_{0}\lambda_{1},\Gup_{\lambda_{1}}(b_{1})}=\Gup(\ast\overline{\jmath}_{w_{0}\lambda_{1}}^{\vee}(b_{1}))$,
and by Lemma \ref{l:Dcrystal}, 
\begin{align*}
\varepsilon_{j}^{\ast}(\ast\overline{\jmath}_{w_{0}\lambda_{1}}^{\vee}(b_{1})) & =\varepsilon_{j}(\overline{\jmath}_{w_{0}\lambda_{1}}^{\vee}(b_{1}))\\
 & =\varepsilon_{j}(b_{1})+\langle h_{j},\wt b_{1}\rangle\\
 & =\varepsilon_{j}(\overline{\jmath}_{\lambda_{1}}(b_{1}))+\langle h_{j},w_{0}\varpi_{i}-s_{i}\varpi_{i}+\lambda_{1}\rangle\\
 & =\varepsilon_{j}(\ast\overline{\jmath}_{w_{0}\varpi_{i}}^{\vee}(u_{s_{i}\varpi_{i}}))+\langle h_{j},w_{0}\varpi_{i}+\varpi_{i}\rangle.
\end{align*}
 By Proposition \ref{p:minor} and Lemma \ref{l:Dcrystal}, 
\[
\varepsilon_{j}(\ast\overline{\jmath}_{w_{0}\varpi_{i}}^{\vee}(u_{s_{i}\varpi_{i}}))=\max\{k\in\mathbb{Z}_{\geq0}\mid D_{e_{j}^{k}.u_{w_{0}\varpi_{i}},u_{s_{i}\varpi_{i}}}\neq0\}.
\]
By the way, there is an involution $\theta$ on $I$ defined by $w_{0}\alpha_{i}=-\alpha_{\theta(i)}$.
Then $w_{0}\varpi_{i}=-\varpi_{\theta(i)}$ and $s_{\theta(i)}w_{0}\varpi_{i}=w_{0}s_{i}\varpi_{i}$.
When $\mathfrak{g}$ does not have ideals of Lie algebras which are
isomorphic to $\mathfrak{sl}_{2}$, we have $D_{w_{0}s_{i}\varpi_{i},s_{i}\varpi_{i}}\neq0$.
Therefore $\varepsilon_{j}(\ast\overline{\jmath}_{w_{0}\varpi_{i}}^{\vee}(u_{s_{i}\varpi_{i}}))=\delta_{j,\theta(i)}$.
Hence 
\[
\varepsilon_{j}^{\ast}(\ast\overline{\jmath}_{w_{0}\lambda_{1}}^{\vee}(b_{1}))=\delta_{j,\theta(i)}-\delta_{j,\theta(i)}+\delta_{ij}=\delta_{ij}.
\]
Therefore $\sum_{j\in I}\varepsilon_{j}^{\ast}(\ast\overline{\jmath}_{w_{0}\lambda_{1}}^{\vee}(b_{1}))\varpi_{j}=\varpi_{i}$.
Then there exists $b_{2}\in\mathscr{B}(\varpi_{i})$ such that $D_{w_{0}\lambda_{1},\Gup_{\lambda_{1}}(b_{1})}=D_{\Gup_{\varpi_{i}}(b_{2}),u_{\varpi_{i}}}$.
Then 
\begin{align*}
\eta_{w_{0},q}^{3}(D_{s_{i}\varpi_{i},\varpi_{i}}) & \simeq D_{w_{0}\varpi_{i},\varpi_{i}}^{-1}D_{w_{0}\lambda_{1},\lambda_{1}}D_{w_{0}\varpi_{i},\varpi_{i}}^{-1}D_{u_{w_{0}\varpi_{i}},\Gup_{\varpi_{i}}(b_{2})}\\
 & \simeq D_{w_{0},-\alpha_{i}}D_{u_{w_{0}\varpi_{i}},\Gup_{\varpi_{i}}(b_{2})}.
\end{align*}
Here, 

\begin{align*}
\wt D_{u_{w_{0}\varpi_{i}},\Gup_{\varpi_{i}}(b_{2})} & =w_{0}\varpi_{i}-\wt b_{2}=w_{0}\varpi_{i}-(w_{0}\lambda_{1}-\wt b_{1}+\varpi_{i})\\
 & =w_{0}\varpi_{i}-(w_{0}\lambda_{1}-(w_{0}\varpi_{i}-s_{i}\varpi_{i}+\lambda_{1})+\varpi_{i})=-\alpha_{\theta(i)}.
\end{align*}
Hence $D_{u_{w_{0}\varpi_{i}},\Gup_{\varpi_{i}}(b_{2})}=D_{s_{\theta(i)}\varpi_{\theta(i)},\varpi_{\theta(i)}}$
because both sides are unique elements of the dual canonical basis
of weight $-\alpha_{\theta(i)}$. Therefore, 
\[
\eta_{w_{0},q}^{6}(D_{s_{i}\varpi_{i},\varpi_{i}})\simeq D_{w_{0},\alpha_{i}-\alpha_{\theta(i)}}D_{s_{i}\varpi_{i},\varpi_{i}}.
\]
Moreover, by Theorem \ref{t:BZauto}, $\eta_{w_{0},q}^{6}(D_{s_{i}\varpi_{i},\varpi_{i}})$
is an element of dual canonical basis, in particular, dual bar-invariant.
Therefore, 
\[
\eta_{w_{0},q}^{6}(D_{s_{i}\varpi_{i},\varpi_{i}})=q^{(\alpha_{i}-\alpha_{\theta(i)},\alpha_{i})}D_{w_{0},\alpha_{i}-\alpha_{\theta(i)}}D_{s_{i}\varpi_{i},\varpi_{i}}.
\]
By this result and Proposition \ref{p:comm}, \ref{p:Elambda}, for
$i_{1},\dots,i_{\ell}\in I$, we have 
\begin{align*}
 & \eta_{w_{0},q}^{6}(D_{s_{i_{1}}\varpi_{i_{1}},\varpi_{i_{1}}}\cdots D_{s_{i_{\ell}}\varpi_{i_{\ell}},\varpi_{i_{\ell}}})\\
 & =q^{\sum_{k=1}^{\ell}(\alpha_{i_{k}}-\alpha_{\theta(i_{k})},\alpha_{i_{k}})}D_{w_{0},\alpha_{i_{1}}-\alpha_{\theta(i_{1})}}D_{s_{i_{1}}\varpi_{i_{1}},\varpi_{i_{1}}}\cdots D_{w_{0},\alpha_{i_{\ell}}-\alpha_{\theta(i_{\ell})}}D_{s_{i_{\ell}}\varpi_{i_{\ell}},\varpi_{i_{\ell}}}\\
 & =q^{(\sum_{k=1}^{\ell}\alpha_{i_{k}}-\sum_{k=1}^{\ell}\alpha_{\theta(i_{k})},\sum_{k=1}^{\ell}\alpha_{i_{k}})}D_{w_{0},\sum_{k=1}^{\ell}\alpha_{i_{k}}-\sum_{k=1}^{\ell}\alpha_{\theta(i_{k})}}D_{s_{i_{1}}\varpi_{i_{1}},\varpi_{i_{1}}}\cdots D_{s_{i_{\ell}}\varpi_{i_{\ell}},\varpi_{i_{\ell}}}.
\end{align*}
Hence the desired equality in the theorem holds for all $x\in\Aq[N_{-}]$
since the elements $D_{s_{i}\varpi_{i},\varpi_{i}}=(1-q_{i}^{2})f_{i},i\in I$
generate the quantum unipotent subgroup $\Aq[N_{-}]$. Then we can
easily extend this result to that for $\Aq[N_{-}^{w_{0}}]$ by straightforward
calculation. The explicit calculation is left to the reader.
\end{proof}
\begin{rem}\label{r:decrease}
In the essential part of the proof of Theorem \ref{t:period}, we check the periodicity on generators of $\Aq[N_{-}^{w_{0}}]$. We should note that this set of generators is not the set of generators of $\mathbb{C}[N_{-}^{w_{0}}]$ after specialization unless $\mathfrak{g}=\mathfrak{sl}_2$.

Indeed, in general Kac-Moody cases, the quantum unipotent cell $\Aq[N_{-}^{w}]$ is generated by $\{[f_i]\mid i\in I\}\cup\{[D_{w\rho, \rho}]^{-1}\}$ (recall that $\rho:=\sum_{i\in I}\varpi_i$). In particular, 
\begin{center}
(the number of the generators of $\Aq[N_{-}^{w}]$) $\leq \# I+1$. 
\end{center}
On the other hand, 
\begin{center}
(the number of the generators of $\mathbb{C}[N_{-}^{w}]$)$\geq \dim N_-^w=\ell(w)$. 
\end{center}
Therefore the periodicity might be checked in the quantum setting more easily than in the classical setting by this decrease of numbers of generators. 
\end{rem}
\begin{acknowledgement*}
The authors are grateful to Bernard Leclerc for his enlightening advice
concerning cluster algebras and their categorifications. They
would like to express our sincere gratitude to Yoshihisa Saito, the
supervisor of the second author, for his helpful comments. They 
wish to thank Ryo Sato and Bea Schumann for several interesting
comments and discussions. They are also thankful to the anonymous referees whose suggestions significantly improve the present paper. The second author thanks the University of Caen Normandy, where a part of this paper was written, for the hospitality. 
\end{acknowledgement*}
\providecommand{\bysame}{\leavevmode\hbox to3em{\hrulefill}\thinspace}

\bibliographystyle{amsalpha}


\end{document}